%% file: compactnessHhol.tex
\documentclass[english]{article}
\usepackage[T1]{fontenc}
\usepackage[latin9]{inputenc}
\usepackage{geometry}
\usepackage{color}
\usepackage{mathrsfs}
\usepackage{amsmath}
\usepackage{amsthm}
\usepackage{amssymb}
\usepackage{setspace}
\makeatletter
\numberwithin{equation}{section}
\numberwithin{figure}{section}
\newcommand{\lyxaddress}[1]{
\par {\raggedright #1
\vspace{1.4em}
\noindent\par}
}
\theoremstyle{plain}
\newtheorem{thm}{\protect\theoremname}
  \theoremstyle{remark}
  \newtheorem{rem}[thm]{\protect\remarkname}
  \theoremstyle{remark}
  \newtheorem*{acknowledgement*}{\protect\acknowledgementname}
  \theoremstyle{definition}
  \newtheorem{defn}[thm]{\protect\definitionname}
  \theoremstyle{plain}
  \newtheorem{cor}[thm]{\protect\corollaryname}
  \theoremstyle{plain}
  \newtheorem{lem}[thm]{\protect\lemmaname}
  \theoremstyle{plain}
  \newtheorem{prop}[thm]{\protect\propositionname}

\newcommand{\bigslant}[2]{{\raisebox{.2em}{$#1$}\left/\raisebox{-.2em}{$#2$}\right.}}
\usepackage{import}  
\usepackage{graphicx}
\usepackage{xcolor}
\usepackage{calc}
\usepackage{amssymb}

\usepackage{ccfonts,eulervm}
\usepackage[T1]{fontenc}

\usepackage{hyperref}

\newcommand{\xdownarrow}[1]{%
  {\left\downarrow\vbox to #1{}\right.\kern-\nulldelimiterspace}
}

\newenvironment{rcases}
  {\left.\begin{aligned}}
  {\end{aligned}\right\rbrace}

\makeatother

\usepackage{babel}
  \providecommand{\acknowledgementname}{Acknowledgement}
  \providecommand{\corollaryname}{Corollary}
  \providecommand{\definitionname}{Definition}
  \providecommand{\lemmaname}{Lemma}
  \providecommand{\propositionname}{Proposition}
  \providecommand{\remarkname}{Remark}
\providecommand{\theoremname}{Theorem}

\begin{document}
\begin{singlespace}

\title{\noindent A Compactness Result for $\mathcal{H}-$holomorphic Curves
in Symplectizations}
\end{singlespace}
\begin{singlespace}

\author{\noindent Alexandru Doicu$^{*}$, Urs Fuchs$^{**}$}
\end{singlespace}
\maketitle
\begin{singlespace}

\lyxaddress{\noindent $^{\star}$Institut für Mathematik, Universität Augsburg.
E-mail: \texttt{alexandru.doicu@math.uni-augsburg.de}}

\lyxaddress{\noindent $^{\star\star}$Mathematisches Institut, Universität Heidelberg.
E-mail: \texttt{ufuchs@mathi.uni-heidelberg.de}}
\end{singlespace}
\begin{abstract}
\begin{singlespace}
\noindent $\mathcal{H}-$holomorphic curves are solutions of a specific
modification of the pseudoholomorphic curve equation in symplectizations
involving a harmonic $1-$form as perturbation term. In this paper
we compactify the moduli space of $\mathcal{H}-$holomorphic curves
with a priori bounds on the harmonic $1-$forms.
\end{singlespace}
\end{abstract}
\begin{singlespace}
\noindent \tableofcontents{}
\end{singlespace}
\begin{singlespace}

\section{\label{sec:Introduction}Introduction}
\end{singlespace}

\begin{singlespace}
\noindent Let $M$ be a closed, connected, $3-$dimensional manifold
and $\alpha$ a $1-$form on $M$ such that $(M,\alpha)$ is a contact
manifold. Further, let $X_{\alpha}$ be the Reeb vector field with
respect to the contact form $\alpha$ on $M$, defined by $\iota_{X_{\alpha}}\alpha\equiv1$
and $\iota_{X_{\alpha}}d\alpha\equiv0$. Denote by $\xi=\ker(\alpha)$
the contact structure and $\pi_{\alpha}:TM\rightarrow\xi$ the canonical
projection along the Reeb vector field $X_{\alpha}$. Denote by $\phi_{\rho}^{\alpha}$
the flow of $X_{\alpha}$, and note that $\phi_{\rho}^{\alpha}$ preserves
the contact structure and the Reeb vector field $X_{\alpha}$. Consider
a $d\alpha-$compatible almost complex structure $J_{\xi}:\xi\rightarrow\xi$,
and let $J$ be the extension of $J_{\xi}$ to a $\mathbb{R}-$invariant
almost complex structure on $\mathbb{R}\times M$ by mapping $1\in T\mathbb{R}$
to $X_{\alpha}$ and $X_{\alpha}$ to $-1\in T\mathbb{R}$. $M$ is
equipped with the metric
\[
g=\alpha\otimes\alpha+d\alpha(\cdot,J_{\xi}\cdot)
\]
while we equip the $\mathbb{R}\times M$ with the metric
\begin{equation}
\overline{g}=dr\otimes dr+g\label{eq:metric-symplectization}
\end{equation}
where $r$ is the coordinate on $\mathbb{R}$. Throughout this paper
we assume that all periodic orbits are non-degenerate. This means
that for every periodic orbit $x$ of period $T$, the linear map
$d\phi_{T}^{\alpha}(x(0)):\xi_{x(0)}\rightarrow\xi_{x(T)}$ does not
contain $1$ in its spectrum. Let $(S,j)$ be a closed Riemann surface
and $\mathcal{P}\subset S$ a finite subset whose elements are called
``punctures''. A proper and non-constant map $u=(a,f):S\backslash\mathcal{P}\rightarrow\mathbb{R}\times M$
is called $\mathcal{H}-$holomorphic if 
\begin{equation}
\begin{array}{cl}
\pi_{\alpha}df\circ j & =J_{\xi}(u)\circ\pi_{\alpha}df\\
f^{*}\alpha\circ j & =da+\gamma\\
E(u;S\backslash\mathcal{P}) & <\infty
\end{array}\label{eq:H-hol}
\end{equation}
for a harmonic $1-$form $\gamma\in\mathcal{H}_{j}^{1}(S)$, i.e.
$d\gamma=d(\gamma\circ j)=0$. The energy
\begin{equation}
E(u;S\backslash\mathcal{P})=\sup_{\varphi\in\mathcal{A}}\int_{S\backslash\mathcal{P}}\varphi'(a)da\circ j\wedge da+\int_{S\backslash\mathcal{P}}f^{*}d\alpha\label{eq:energy}
\end{equation}
sums the contribution of the $\alpha-$energy (first term) and the
$d\alpha-$energy (second term) of $u$ on $S\backslash\mathcal{P}.$
The set $\mathcal{A}$ consists of all smooth maps $\varphi:\mathbb{R}\rightarrow[0,1]$
with $\varphi'(r)\geq0$ for all $r\in\mathbb{R}$. Note that if the
perturbation $1-$form $\gamma$ vanishes, the energy of $u$ is equal
to the Hofer energy defined in \cite{key-1}. This modification of
the pseudoholomorphic curve equation, which was first suggested by
Hofer \cite{key-3}, was used by Abbas et al. \cite{key-4} to prove
the generalized Weinstein conjecture in dimension three. However,
due to a lack of a compactness result of the moduli space of $\mathcal{H}-$holomorphic
curves, the generalized Weinstein conjecture was proved only in the
planar case, i.e. when the leaves of the holomorphic open book decomposition
\cite{key-18} have zero genus. In this paper we describe a compactification
of the moduli space of finite energy $\mathcal{H}-$holomorphic curves
by imposing some additional conditions. These are outlined below.

\noindent The $L^{2}-$norm of the harmonic perturbation $1-$form
$\gamma$ is defined by
\[
\left\Vert \gamma\right\Vert _{L^{2}(S)}^{2}=\int_{S}\gamma\circ j\wedge\gamma.
\]
For an isotopy class $[c]$ which is represented by a smooth loop
$c$ the period and co-period of $\gamma$ over $[c]$ are
\begin{equation}
P_{\gamma}([c])=\int_{c}\gamma\label{eq:periods}
\end{equation}
and
\begin{equation}
S_{\gamma}([c])=\int_{c}\gamma\circ j,\label{eq:co-period}
\end{equation}
respectively. Let $R_{[c]}$ the conformal modulus of $[c]$, as defined
in \cite{key-10}. The \emph{conformal period }and \emph{co-period}
of $\gamma$ over $[c]$ are $\tau_{[c],\gamma}=R_{[c]}P_{\gamma}([c])$
and $\sigma_{[c],\gamma}=R_{[c]}S_{\gamma}([c])$, respectively. The
significance of these two quantities will become apparent later on. 

\noindent The compactness result will be established for finite energy
$\mathcal{H}-$holomorphic curves with harmonic perturbation $1-$forms
having uniformly bounded $L^{2}-$norms and uniformly bounded conformal
periods and co-periods. Specifically, we will consider a sequence
of $\mathcal{H}-$holomorphic curves $u_{n}=(a_{n},f_{n}):(S_{n}\backslash\mathcal{P}_{n},j_{n})\rightarrow\mathbb{R}\times M$
with harmonic perturbations $\gamma_{n}$, satisfying the following
conditions:
\end{singlespace}
\begin{description}
\begin{singlespace}
\item [{A1}] \noindent $(S_{n},j_{n})$ are compact Riemann surfaces of
the same genus and $\mathcal{P}_{n}\subset S_{n}$ is a finite set
of punctures whose cardinality is independent of $n$.
\item [{A2}] \noindent The energy of $u_{n}$, as well as the $L^{2}-$norm
of $\gamma_{n}$ are uniformly bounded by the constants $E_{0}>0$
and $C_{0}>0$, respectively.
\end{singlespace}
\end{description}
\begin{singlespace}
\noindent In \cite{key-15} Bergmann introduced a model for the compactification
of the moduli space of $\mathcal{H}-$holomorphic curves satisfying
the conditions A1 and A2. However, he neglects the phenomena occuring
when conformal periods and conformal coperiods of the sequence of
the harmonic perturbation $1-$forms $\gamma_{n}$ are unbounded.
In this case, the convergence behaviour can be very complicated. For
instance, near a node, the $\mathcal{H}-$holomorphic curve can follow
along a Reeb trajectory which is dense in $M$. To avoid these complications
we make additional assumptions on the conformal period and co-period.
The task is to establish a notion of convergence of such curves as
well as the description of the limit object similar to \cite{key-1}
and \cite{key-10}. Essentially, we will prove the convergence of
a sequence of $\mathcal{H}-$holomorphic maps to a stratified broken
$\mathcal{H}-$holomorphic building. The concept of a stratified broken
$\mathcal{H}-$holomorphic building of a certain level is similar
to that given in \cite{key-1}. Each level consists of a nodal $\mathcal{H}-$holomorphic
curve having the same asymptotic properties at the positive and negative
punctures. However, as compared to \cite{key-1}, there are two differences: 
\end{singlespace}
\begin{enumerate}
\begin{singlespace}
\item The nodes on each level are not just points; in our setting they are
replaced by a finite length Reeb trajectory. 
\item The breaking orbits between the levels have a twist in the sense made
precise in Definition \ref{def:A-mapwhere-} from Section \ref{subsec:Notion-of-convergence}.
\end{singlespace}
\end{enumerate}
\begin{rem}
\begin{singlespace}
\noindent \label{rem:For-a-sequence}For a sequence of punctured Riemann
surfaces $(S_{n},j_{n},\mathcal{P}_{n})$, the Deligne-Mumford convergence
result implies that there exists a punctured nodal Riemann surface
$(S,j,\mathcal{P},\mathcal{D})$ and a sequence of diffeomorphisms
$\varphi_{n}:S^{D,r}\rightarrow S_{n}$, such that $\varphi_{n}^{*}j_{n}$
converges outside certain circles in $C_{\text{loc}}^{\infty}$ to
$j$. Here, $S^{D,r}$ is the surface obtained by blowing up the points
from $\mathcal{D}$ and identifying them via the decoration $r$ (see
Section \ref{chap:Definitions-and-main}). Denote by $\Gamma_{i}^{\text{nod}}$,
for $i=1,...,|\mathcal{D}|/2$, the equivalence classes of the boundary
circles of $S^{\mathcal{D}}$in $S^{\mathcal{D},r}$. Let $\Gamma_{n,i}^{\text{nod}}=(\varphi_{n})_{*}\Gamma_{i}^{\text{nod}}$
for all $n\in\mathbb{N}$ and $i=1,...,|\mathcal{D}|/2$. 
\end{singlespace}
\end{rem}
\begin{singlespace}
\noindent The main result of our analysis is the following
\end{singlespace}
\begin{thm}
\begin{singlespace}
\noindent \label{thm:Let--be-2-1}Let $(S_{n},j_{n},u_{n},\mathcal{P}_{n},\gamma_{n})$
be a sequence of $\mathcal{H}-$holomorphic curves in $\mathbb{R}\times M$
satisfying assumptions A1 and A2. Then there exists a subsequence
that converges to a $\mathcal{H}-$holomorphic curve $(S,j,u,\mathcal{P},\mathcal{D},\gamma)$
in the sense of Definition \ref{def:A-sequence-of} from Section \ref{subsec:Convergence}.
Moreover, if there exists a constant $C>0$ such that for all $n\in\mathbb{N}$
and all $1\leq i\leq|\mathcal{D}|/2$ we have $|\tau_{[\Gamma_{n,i}^{\text{nod}}],\gamma_{n}}|,|\sigma_{[\Gamma_{n,i}^{\text{nod}}],\gamma_{n}}|<C$
then $(S,j,u,\mathcal{P},\mathcal{D},\gamma)$ is a stratified $\mathcal{H}-$holomorphic
building of height $N$ and after going over to a subsequence the
$\mathcal{H}-$holomorphic curves $(S_{n},j_{n},u_{n},\mathcal{P}_{n},\gamma_{n})$
converges to $(S,j,u,\mathcal{P},\mathcal{D},\gamma)$ in the sense
of Definition \ref{def:A-sequence-of-1} from Section \ref{subsec:Convergence}.
\end{singlespace}
\end{thm}
\begin{singlespace}
\noindent For applications we would like to get rid of the a priori
bounds on the $L^{2}-$norm of $\gamma_{n}$. Such bounds are automatically
satisfied if all the leafs of the foliation, given by $\ker(f^{*}\alpha\circ j)$,
are compact.
\end{singlespace}
\begin{singlespace}

\subsection{Outline of the paper}
\end{singlespace}

\begin{singlespace}
\noindent The paper is organized as follows. In Chapter \ref{chap:Definitions-and-main}
we review the basic concepts related to the compactness of $\mathcal{H}-$holomorphic
curves. More precisely, in Section \ref{sec:Harmonic-Deligne-Mumford-converg}
we recall the Deligne-Mumford convergence theorem for stable Riemann
surfaces by following the analysis given in \cite{key-1} and \cite{key-6}.
In Section \ref{sec:Harmonic-pertubed-pseudoholomorp} we provide
the precise definition of $\mathcal{H}-$holomorphic curves. By Proposition
\ref{prop:Let--be}, we recall a result established by Hofer et al.
\cite{key-12} stating that the behaviour of $\mathcal{H}-$holomorphic
curves in a neighbourhood of the punctures is similar to that of usual
pseudoholomorphic curves. This result will enable us to split the
set of punctures into positive and negative punctures as in \cite{key-1}.
In Section \ref{subsec:Notion-of-convergence} we introduce the notion
of convergence and describe the limit object. In particular, the description
of the limit object is given in Section \ref{subsec:Stratified-broken-holomorphic},
Definition \ref{def:A-mapwhere-}, while the notion of convergence
is defined in Section \ref{subsec:Convergence}, Definition \ref{def:A-sequence-of}
and Definition \ref{def:A-sequence-of-1}. The limit object known
as a stratified $\mathcal{H}-$holomorphic building of a certain hight
is similar to the broken pseudoholomorphic building introduced in
\cite{key-1}, \cite{key-6} and \cite{key-10}; the difference is
that we allow two points, lying in the same level, to be connected
by a finite length trajectory of the Reeb vector field. Essentially,
the $\mathcal{H}-$holomorphic curves converge in $C_{\text{loc}}^{\infty}$
away from the punctures and certain loops that degenerate to nodes,
while the projections of the $\mathcal{H}-$holomorphic curves to
$M$ converge in $C^{0}$.

\noindent The proof of the main compactness results on the thick part
with certain points removed, as well as on the thin part and in a
neighbourhood of the removed points, are carried out in Sections \ref{2_1_sec:First_Theorem}
and \ref{sec:The-Thin-Part} of Chapter \ref{chap:Compactness-results},
respectively. For the thick part, we use the Deligne-Mumford convergence
and the thick-thin decomposition to show that the domains converge
in the Deligne-Mumford sense to a punctured nodal Riemann surface.
By using bubbling-off analysis and the results of Appendix \ref{chap:Special-coordinates}
(to generate a sequence of holomorphic coordinates that behaves well
under Deligne-Mumford limit process) we prove, after introducing additional
punctures, that the $\mathcal{H}-$holomorphic curves have uniformly
bounded gradients in the complement of the special circles and certain
marked points. By using the elliptic regularity theorem for pseudoholomorphic
curves and Arzela-Ascoli theorem we show that the $\mathcal{H}-$holomorphic
curves togeher with the harmonic perturbations converge in $C_{\text{loc}}^{\infty}$
on the thick part with certain points removed to a $\mathcal{H}-$holomorphic
curve with harmonic perturbation. This result is similar to the bubbling-off
analysis used in \cite{key-1}. However, in contrast to Lemma 10.7
of \cite{key-1}, we do not change the hyperbolic structure each time
after adding the additional marked point generated by the bubbling-off
analysis.

\noindent The thin part is decomposed into cusps corresponding to
neighbourhoods of punctures and hyperbolic cylinders corresponding
to nodes in the limit. As the perturbation harmonic $1-$forms are
exact in a neighbourhood of the punctures or the points that were
removed in the first part, by means of a change of the $\mathbb{R}-$coordinate,
the $\mathcal{H}-$holomorphic curves are turned into usual pseudoholomorphic
curves on which the classical theory \cite{key-1}, \cite{key-10}
is applicable. The case of hyperbolic cylinders is more interesting
because the difference from the classical SFT compactness result is
evident. Due to a lack of the monotonicity lemma, we cannot expect
the $\mathcal{H}-$holomorphic curves to have uniformly bounded gradients,
and so, to apply the classical SFT convergence theory. To deal with
this problem we decompose the hyperbolic cylinder into a finite uniform
number of smaller cylinders; some of them having conformal modulus
tending to infinity but $d\alpha-$energies strictly smaller than
$\hbar$, and the rest of them having bounded modulus but $d\alpha-$energies
possibly larger than $\hbar$. Here, $\hbar>0$ is defined by

\noindent 
\begin{equation}
\hbar:=\min\{|P_{1}-P_{2}|\mid P_{1},P_{2}\in\mathcal{P}_{\alpha},P_{1}\not=P_{2},P_{1},P_{2}\leq E_{0}\}\label{eq:hbar}
\end{equation}
where $\mathcal{P}_{\alpha}$ is the action spectrum of $\alpha$
as defined in \cite{key-11} and $E_{0}>0$ is the uniform bound on
the energy. We refer to these cylinders as cylinders of types $\infty$
and $b_{1}$, respectively, and note that they appear alternately.
Convergence results are derived for each cylinder type, and then glued
together to obtain a convergence result on the whole hyperbolic cylinder.
As cylinders of type $\infty$ have small $d\alpha-$energies, we
assume by the classical bubbling-off analysis, that the $\mathcal{H}-$holomorphic
curves have uniformly bounded gradients. To turn these maps into pseudoholomorphic
curves, we perform a transformation by pushing them along the Reeb
flow up to some specific time characterized by a uniformly bounded
conformal period. These transformed curves are now pseudoholomorphic
with respect to a domain-dependent almost complex structure on $M$,
which due to the uniform boundedness of the conformal period varies
in a compact set. For this part of our analysis, we use the results
established in \cite{key-23}. In the case of cylinders of type $b_{1}$
we proceed as follows. Relying on a bubbling-off argument, as we did
in the case of the thick part, we assume that the gradient blows up
only in a finite uniform number of points and remains uniformly bounded
on a compact complement of them. In this compact region we use Arzelá-Ascoli
theorem to show that the $\mathcal{H}-$holomorphic curves together
with the harmonic perturbations converge in $C^{\infty}$ to some
$\mathcal{H}-$holomorphic curve. What is then left is the convergence
in a neighbourhood of the finitely many punctures, where the gradient
blows up. Here, a neighbourhood of a puncture is a disk on which the
harmonic perturbation can be made exact and can be encoded in the
$\mathbb{R}-$coordinate of the $\mathcal{H}-$holomorphic curve.
By this procedure we transform the $\mathcal{H}-$holomorphic curve
into a usual pseudoholomorphic curve defined on a disk $D$. By the
$C^{\infty}-$convergence on any compact complement of the punctures,
established before, we assume that the transformed curves converge
on an arbitrary neighbourhood of $\partial D$. Then we use the results
of \cite{key-10}, especially Gromov compactness with free boundary,
to obtain a convergence results for cylinders of type $b_{1}$. This
part of our analysis uses extensively the results established in Appendix
\ref{sec:Pseudoholomorphic-disks-with} and some of \cite{key-23}.

\noindent In Chapter \ref{chap:Discussion-on-conformal} we discuss
the condition imposed on the conformal period and co-period, namely
their uniformly boundedness. The conformal period and co-period can
be seen as a link between the conformal data and the topology on the
Riemann surface, as well as the harmonic perturbation $1-$form. Without
these conditions, the transformation performed in \cite{key-23} cannot
be established. The reason is that the domain dependent almost complex
structure, which was constructed in order to change the $\mathcal{H}-$holomorphic
curve into a usual pseudoholomorphic curve, does not vary in a compact
space, and so, the results established in \cite{key-11} cannot be
applied. By means of a counterexample stated in Proposition \ref{prop:There-exists-a-2}
we show that the condition on the uniform bound of the conformal period
is not always satisfied. It should be pointed out that Bergmann \cite{key-15}
claimed to have established a compactification of the space of $\mathcal{H}-$holomorphic
curves by performing the same transformation as we did in \cite{key-23},
i.e. by pushing the $M-$component of the $\mathcal{H}-$holomorphic
curve by the Reeb flow up to some specific time determined by the
conformal period, and then by assuming that the conformal period can
be universally bounded by a quantity which depends only on the periods
of the harmonic perturbation $1-$form (note that if the $L^{2}-$norm
of a sequence of harmonic $1-$forms is uniformly bounded then their
periods are also uniformly bounded). In this context, Proposition
\ref{prop:There-exists-a-2} contradicts his argument.
\end{singlespace}
\begin{acknowledgement*}
\begin{singlespace}
\noindent We thank Peter Albers and Kai Cieliebak who provided insight
and expertise that greatly assisted the research. U.F. is supported
by the SNF fellowship 155099, a fellowship at Institut Mittag-Leffler
and the GIF Grant 1281.
\end{singlespace}
\end{acknowledgement*}
\begin{singlespace}

\section{\label{chap:Definitions-and-main}Definitions and main results}
\end{singlespace}

\begin{singlespace}
\noindent In this chapter we present the basic concepts related to
the compactness of $\mathcal{H}-$holomorphic curves. In particular,
we provide the Deligne-Mumford compactness in order to describe the
convergence of a sequence of Riemann surfaces, introduce the concept
of a stratified $\mathcal{H}-$holomorphic buildings of height $N$,
which serves as limit object, and discuss the convergence of such
maps.
\end{singlespace}
\begin{singlespace}

\subsection{\label{sec:Harmonic-Deligne-Mumford-converg}Deligne-Mumford convergence}
\end{singlespace}

\begin{singlespace}
\noindent In this section we review the Deligne-Mumford convergence
following the analysis given in \cite{key-1} and \cite{key-6}.

\noindent Consider the surface $(S,j,\mathcal{M}\amalg\mathcal{D})$,
where $(S,j)$ is a closed Riemann surface, and $\mathcal{M}$ and
$\mathcal{D}$ are finite disjoint subsets of $S$. Assume that the
cardinality of $\mathcal{D}$ is even. The points from $\mathcal{M}$
are called \emph{marked points,} while the points from $\mathcal{D}$
are called \emph{nodal points}. The points from $\mathcal{D}$ are
organized in pairs, $\mathcal{D}=\{d'_{1},d''_{1},d'_{2},d''_{2},...,d'_{k},d''_{k}\}$.
A nodal surface\textbf{ $(S,j,\mathcal{M}\amalg\mathcal{D})$} is
said to be \emph{stable} if the stability condition $2g+\mu\geq3$
is satisfied for each component of the surface $S,$ where $\mu=|\mathcal{M}\cup\mathcal{D}|$.
With a nodal surface $(S,j,\mathcal{M}\amalg\mathcal{D})$ we can
associate the following singular surface with double points,
\[
\hat{S}_{\mathcal{D}}=S/\{d'_{i}\sim d''_{i}\mid i=1,...,k\}.
\]
The identified points $d'_{i}\sim d''_{i}$ are called \emph{node}s.
The nodal surface $(S,j,\mathcal{M}\amalg\mathcal{D})$ is said to
be \emph{connected} if the singular surface $\hat{S}_{\mathcal{D}}$
is connected. For each marked point $p\in\mathcal{M}\amalg\mathcal{D}$
of a stable nodal Riemann surface $(S,j,\mathcal{M}\amalg\mathcal{D})$,
we define the surface $S^{p}$ with boundary as the oriented blow-up
of $S$ at the point $p$. Thus $S^{p}$ is the circle compactification
of $S\backslash\{p\}$; it is a compact surface bounded by the circle
$\Gamma_{p}=(T_{p}S\backslash\{0\})/\mathbb{R}_{+}$. The canonical
projection $\pi:S^{p}\rightarrow S$ sends the circle $\Gamma_{p}$
to the point $p$ and the maps $S^{p}\backslash\Gamma_{p}$ diffeomorphically
to $S\backslash\{p\}$. Similarly, given a finite set $\mathcal{M}'=\{p_{1},...,p_{k}\}\subset\mathcal{M}\amalg\mathcal{D}$
of punctures, we consider a blow-up surface $S^{\mathcal{M}'}$ with
$k$ boundary components $\Gamma_{1},...,\Gamma_{k}$. It comes with
the projection $\pi:S^{\mathcal{M}'}\rightarrow S$, which collapses
the boundary circles $\Gamma_{1},...,\Gamma_{k}$ to points $p_{1},...,p_{k}$
and the maps $S^{\mathcal{M}'}\backslash\coprod_{i=1}^{k}\Gamma_{i}$
diffeomorphically to $\dot{S}=S\backslash\mathcal{M}'$. 

\noindent If the nodal surface $(S,j,\mathcal{M}\amalg\mathcal{D})$
is connected its arithmetic genus $g$ is defined as
\[
g=\frac{1}{2}|\mathcal{D}|-b_{0}+\sum_{i=1}^{b_{0}}g_{i}+1,
\]
where $|\mathcal{D}|=2k$ is the cardinality of $\mathcal{D}$, $b_{0}$
is the number of connected components of the surface $S$, and $\sum_{i=1}^{b_{0}}g_{i}$
is the sum of the genera of the connected components of $S$. The
\emph{signature} of a nodal curve $(S,j,\mathcal{M}\amalg\mathcal{D})$
is the pair $(g,\mu)$, where $g$ is the arithmetic genus and $\mu=|\mathcal{M}|$.
A stable nodal Riemann surface $(S,j,\mathcal{M}\amalg\mathcal{D})$
is called \emph{decorated} if for each node there is an orientation
reversing orthogonal map
\begin{equation}
r_{i}:\overline{\Gamma}_{i}=(T_{\overline{d}_{i}}S\backslash\{0\})/\mathbb{R}_{+}\rightarrow\underline{\Gamma}_{i}=(T_{\underline{d}_{i}}S\backslash\{0\})/\mathbb{R}_{+}.\label{eq:compactifying_circles-1}
\end{equation}
For the orthogonal orientation reversing map $r_{i}$, we must have
that $r_{i}(e^{2\pi i\vartheta}p)=e^{-2\pi i\vartheta}r(p)$ for all
$p\in\overline{\Gamma}_{i}$.

\noindent In the following we argue as in \cite{key-1}. Consider
the oriented blow-up $S^{\mathcal{D}}$ at the points of $\mathcal{D}$
as described above. The circles $\overline{\Gamma}_{i}$ and $\underline{\Gamma}_{i}$
defined by (\ref{eq:compactifying_circles-1}) are boundary circles
for the points $d'_{i},d''_{i}\in\mathcal{D}$. The canonical projection
$\pi:S^{\mathcal{D}}\rightarrow S$, collapsing the circles $\overline{\Gamma}_{i}$
and $\underline{\Gamma}_{i}$ to the points $d'_{i}$ and $d''_{i}$,
respectively, induces a conformal structure on $S^{\mathcal{D}}\backslash\coprod_{i=1}^{k}\overline{\Gamma}_{i}\amalg\underline{\Gamma}_{i}$.
The smooth structure of $S^{\mathcal{D}}\backslash\coprod_{i=1}^{k}\overline{\Gamma}_{i}\amalg\underline{\Gamma}_{i}$
extends to $S^{\mathcal{D}}$, while the extended conformal structure
degenerates along the boundary circles $\overline{\Gamma}_{i}$ and
$\underline{\Gamma}_{i}$. Let $(S,j,\mathcal{M}\amalg\mathcal{D},r)$
be a decorated surface, where $r=(r_{1},...,r_{k})$. By means of
the mappings $r_{i}$, $i=1,...,k$, $\overline{\Gamma}_{i}$ and
$\underline{\Gamma}_{i}$ can be glued together to yield a closed
surface $S^{\mathcal{D},r}$. The genus of the surface $S^{\mathcal{D},r}$
is equal to the arithmetic genus of $(S,j,\mathcal{M}\amalg\mathcal{D})$.
There exists a canonical projection $p:S^{\mathcal{D},r}\rightarrow\hat{S}_{\mathcal{D}}$
which projects the circle $\Gamma_{i}=\{\overline{\Gamma}_{i},\underline{\Gamma}_{i}\}$
to the node $d_{i}=\{d'_{i},d''_{i}\}$. The projection $p$ induces
on the surface $S^{\mathcal{D},r}$ a conformal structure in the complement
of the special circles $\Gamma_{i}$; the conformal structure is still
denoted by $j$. The continous extension of $j$ to $S^{\mathcal{D},r}$
degenerates along the special circles $\Gamma_{i}$.

\noindent According to the uniformization theorem, for a stable surface
$(S,j,\mathcal{M}\amalg\mathcal{D})$ there exists a unique complete
hyperbolic metric of constant curvature $-1$ of finite volume, in
the given conformal class $j$ on $\dot{S}=S\backslash(\mathcal{M}\amalg\mathcal{D})$.
This metric is denoted by $h^{j,\mathcal{M}\amalg\mathcal{D}}$. Each
point in $\mathcal{M}\amalg\mathcal{D}$ corresponds to a cusp of
the hyperbolic metric $h^{j,\mathcal{M}\amalg\mathcal{D}}$. Assume
that for a given stable Riemann surface $(S,j,\mathcal{M}\amalg\mathcal{D})$,
the punctured surface $\dot{S}=S\backslash(\mathcal{M}\amalg\mathcal{D})$
is endowed with the uniformizing hyperbolic metric $h^{j,\mathcal{M}\amalg\mathcal{D}}$. 

\noindent Fix $\delta>0$, and denote by
\begin{eqnarray*}
\text{Thick}_{\delta}(S,h^{j,\mathcal{M}\amalg\mathcal{D}}) & = & \left\{ x\in\dot{S}\mid\rho(x)\geq\delta\right\} 
\end{eqnarray*}
and
\begin{eqnarray*}
\text{Thin}_{\delta}(S,h^{j,\mathcal{M}\amalg\mathcal{D}}) & = & \overline{\left\{ x\in\dot{S}\mid\rho(x)<\delta\right\} },
\end{eqnarray*}
the $\delta-$thick and $\delta-$thin parts, respectively, where
$\rho(x)$ is the injectivity radius of the metric $h^{j,\mathcal{M}\amalg\mathcal{D}}$
at the point $x\in\dot{S}$. A fundamental result of hyperbolic geometry
states that there exists a universal constant $\delta_{0}=\sinh^{-1}(1)$
such that for any $\delta<\delta_{0}$, each component $C$ of $\text{Thin}_{\delta}(S,h^{j,\mathcal{M}\amalg\mathcal{D}})$
is conformally equivalent either to a finite cylinder $[-R,R]\times S^{1}$
if the component $C$ is not adjacent to a puncture, or to the punctured
disk $D\backslash\{0\}\cong[0,\infty)\times S^{1}$ if it is adjacent
to a puncture (see, for example, \cite{key-5} and \cite{key-6}).
Each compact component $C$ of the thin part contains a unique closed
geodesic of length $2\rho(C)$ denoted by $\Gamma_{C}$, where $\rho(C)=\inf_{x\in C}\rho(x)$.
When considering the $\delta-$thick$-$thin decompositions we always
assume that $\delta$ is chosen smaller than $\delta_{0}$.

\noindent The uniformization metric $h^{j,\mathcal{M}\amalg\mathcal{D}}$
can be lifted to a metric $\overline{h}^{j,\mathcal{M}\amalg\mathcal{D}}$
on $\dot{S}^{\mathcal{D},r}:=S^{\mathcal{D},r}\backslash\mathcal{M}$.
The lifted metric degenerates along each circle $\Gamma_{i}$ in the
sense that the length of $\Gamma_{i}$ is $0$, and the distance of
$\Gamma_{i}$ to any other point in $\dot{S}^{\mathcal{D},r}$ is
infinite. However, we can still speak about geodesics on $\dot{S}^{\mathcal{D},r}$
which are orthogonal to $\Gamma_{i}$, i.e., two geodesics rays, whose
asymptotic directions at the cusps $d'_{i}$ and $d''_{i}$ are related
via the map $r_{i}$, and which correspond to a compact geodesic interval
in $S^{\mathcal{D},r}$ intersecting orthogonally the circle $\Gamma_{i}$.
It is convenient to regard $\text{Thin}_{\delta}(S,h^{j,\mathcal{M}\amalg\mathcal{D}})$
and $\text{Thick}_{\delta}(S,h^{j,\mathcal{M}\amalg\mathcal{D}})$
as subsets of $\dot{S}^{\mathcal{D},r}$. This interpretation provides
a compactification of the non-compact components of $\text{Thin}_{\delta}(S,h^{j,\mathcal{M}\amalg\mathcal{D}})$
not adjacent to points from $\mathcal{M}$. Any compact component
$C$ of $\text{Thin}_{\delta}(S,h^{j,\mathcal{M}\amalg\mathcal{D}})\subset\dot{S}^{\mathcal{D},r}$
is a compact annulus; it contains either a closed geodesic $\Gamma_{C}$,
or one of the special circles, still denoted by $\Gamma_{C}$, which
projects to a node (as described above).

\noindent Consider a sequence of decorated stable nodal marked Riemann
surfaces $(S_{n},j_{n},\mathcal{M}_{n}\amalg\mathcal{D}_{n},r_{n})$
indexed by $n\in\mathbb{N}$. 
\end{singlespace}
\begin{defn}
\begin{singlespace}
\noindent \label{def:The-sequence-}The sequence $(S_{n},j_{n},\mathcal{M}_{n}\amalg\mathcal{D}_{n},r_{n})$
is said to converge to a decorated stable nodal surface $(S,j,\mathcal{M}\amalg\mathcal{D},r)$
if for sufficiently large $n$, there exists a sequence of diffeomorphisms
$\varphi_{n}:S^{\mathcal{D},r}\rightarrow S_{n}^{\mathcal{D}_{n},r_{n}}$
with $\varphi_{n}(\mathcal{M})=\mathcal{M}_{n}$ such that the following
are satisfied.
\end{singlespace}
\begin{enumerate}
\begin{singlespace}
\item For any $n\geq1$, the images $\varphi_{n}(\Gamma_{i})$ of the special
circles $\Gamma_{i}\subset S^{\mathcal{D},r}$ for $i=1,...,k$, are
special circles or closed geodesics of the metrics $h^{j_{n},\mathcal{M}_{n}\amalg\mathcal{D}_{n}}$
on $\dot{S}^{\mathcal{D}_{n},r_{n}}$. All special circles on $S^{\mathcal{D}_{n},r_{n}}$
are among these images.
\item $h_{n}\rightarrow\overline{h}$ in $C_{\text{loc}}^{\infty}(\dot{S}^{\mathcal{D},r}\backslash\coprod_{i=1}^{k}\Gamma_{i})$,
where $h_{n}:=\varphi_{n}^{*}h^{j_{n},\mathcal{M}_{n}\amalg\mathcal{D}_{n}}$
and $\overline{h}:=\overline{h}^{j,\mathcal{M}\amalg\mathcal{D}}$.
\item Given a component $C$ of $\text{Thin}_{\delta}(S,h^{j,\mathcal{M}\amalg\mathcal{D}})\subset\dot{S}^{\mathcal{D},r}$
containing a special circle $\Gamma_{i}$, and given a point $c_{i}\in\Gamma_{i}$,
let $\delta_{i}^{n}$ be the geodesic arc corresponding to the induceed
metric $h_{n}=\varphi_{n}^{*}h^{j_{n},\mathcal{M}_{n}\amalg\mathcal{D}_{n}}$
for any $n\geq1$, intersecting $\Gamma_{i}$ orthogonally at the
point $c_{i}$, and having the ends in the $\delta-$thick part of
the metric $h_{n}$. Then, in the limit $n\rightarrow\infty$, $(C\cap\delta_{i}^{n})$
converge in $C^{0}$ to a continous geodesic for a metric $\overline{h}$
passing through the point $c_{i}$.
\end{singlespace}
\end{enumerate}
\end{defn}
\begin{rem}
\begin{singlespace}
\noindent In view of the uniformization theorem, one can see that
the second property of Definition \ref{def:The-sequence-} is equivalent
to the condition $\varphi_{n}^{*}j_{n}\rightarrow j$ in $C_{\text{loc}}^{\infty}(\dot{S}^{\mathcal{D},r}\backslash\coprod_{i=1}^{k}\Gamma_{i})$
which in turn, by the removable singularity theorem, is equivalent
to the convergence in $C_{\text{loc}}^{\infty}(S^{\mathcal{D},r}\backslash\coprod_{i=1}^{k}\Gamma_{i})$.
\end{singlespace}
\end{rem}
\begin{singlespace}
\noindent We are now in the position to state the Deligne-Mumford
convergence.
\end{singlespace}
\begin{thm}
\begin{singlespace}
\noindent \label{thm:(Harmonic-Deligne-Mumford)-Every}(Deligne-Mumford)
Any sequence of nodal stable Riemann surfaces $(S_{n},j_{n},\mathcal{M}_{n}\amalg\mathcal{D}_{n},r_{n})$
of signature $(g,\mu)$ has a subsequence which converges in the sense
of Definition \ref{def:The-sequence-} to a decorated nodal stable
Riemann surface $(S,j,\mathcal{M}\amalg\mathcal{D},r)$ of signature
$(g,\mu)$.
\end{singlespace}
\end{thm}
\begin{cor}
\begin{singlespace}
\noindent \label{cor:Any-sequence-of}Any sequence of stable Riemann
surfaces $(S_{n},j_{n},\mathcal{M}_{n})$ of signature $(g,\mu)$
has a subsequence which converges to a decorated nodal stable Riemann
surface $(S,j,\mathcal{M}\amalg\mathcal{D},r)$ of signature $(g,\mu)$.
\end{singlespace}
\end{cor}
\begin{singlespace}

\subsection{\label{sec:Harmonic-pertubed-pseudoholomorp}Asymptotic behaviour
of $\mathcal{H}-$holomorphic curves}
\end{singlespace}

\begin{singlespace}
\noindent To describe the behaviour of a $\mathcal{H}-$holomorphic
curve near the puncture from $\mathcal{P}$ we need some auxiliary
tools. One of these is the lemma about the removal of singularity.
Consider a $\mathcal{H}-$holomorphic curve $(S,j,\mathcal{P},u,\gamma)$,
and assume that the set of punctures $\mathcal{P}\subset S$ is not
empty. For $p\in\mathcal{P}$, consider a neighbourhood $U(p)=U\subset S$,
which is biholomorphic to the standard open unit disk $D\subset\mathbb{C}$,
such that, under this biholomorphism, the point $p$ is mapped to
$0$. 

\noindent First we mention a removable singularity result for a harmonic
$1-$form $\gamma$ defined on the punctured unit disk $D\backslash\{0\}$.
\end{singlespace}
\begin{lem}
\begin{singlespace}
\noindent \label{lem:If--is}If $\gamma$ is a harmonic $1-$form
defined on the punctured disk $D\backslash\{0\}$, and having a bounded
$L^{2}-$norm with respect to the standard complex structure $i$
on $D$, i.e. $\left\Vert \gamma\right\Vert _{L^{2}(D\backslash\{0\})}^{2}<\infty$
then $\gamma$ can be extended accross the puncture.
\end{singlespace}
\end{lem}
\begin{proof}
\begin{singlespace}
\noindent With $z=s+it=(s,t)$ being the coordinates on $D$, we express
$\gamma$ as $\gamma=f(s,t)ds+g(s,t)dt$, where $f,g:D\backslash\{0\}\rightarrow\mathbb{R}$
are harmonic functions. As $\gamma$ is harmonic with respect to the
standard complex structure $i$, $F:=f+ig:D\backslash\{0\}\rightarrow\mathbb{C}$
is a mermorphic function with a bounded $L^{2}-$norm, i.e.,
\[
\int_{D\backslash\{0\}}|F(s,t)|^{2}dsdt=\int_{D\backslash\{0\}}\left(|f(s,t)|^{2}+|g(s,t)|^{2}\right)dsdt<\infty.
\]
Consider the Laurent series of $F$,
\[
F(z)=\sum_{n=-\infty}^{\infty}F_{n}z^{n},
\]
where $F_{n}\in\mathbb{C}$. Since the Laurent series converges in
$C_{\text{loc}}^{0}$ to $F$ and $e^{2\pi in\theta}$ is an orthonormal
system in $L^{2}(S^{1})$, we infer that for every fixed $0<\rho<1$,
\[
\int_{0}^{1}|F(\rho e^{2\pi i\theta})|^{2}d\theta=\sum_{n=-\infty}^{\infty}|F_{n}|^{2}\rho^{2n}.
\]
Consequently, due to Fubini's theorem,
\[
\int_{D\backslash\{0\}}|F(z)|^{2}dsdt=2\pi\int_{(0,1]\times S^{1}}\rho|F(\rho e^{2\pi i\theta})|^{2}d\theta d\rho=2\pi\int_{0}^{1}\sum_{n=-\infty}^{\infty}|F_{n}|^{2}\rho^{2n+1}d\rho.
\]
As the terms in the sum are all non-negative, it follows that
\[
\int_{D\backslash\{0\}}|F(z)|^{2}dsdt\geq2\pi|F_{n}|^{2}\int_{0}^{1}\rho^{2n+1}d\rho
\]
for all $n\in\mathbb{Z}$. However, for $n<0$ and because of
\[
\int_{0}^{1}\rho^{2n+1}d\rho=\infty,
\]
this yields a contradiction to the finiteness of the $L^{2}-$norm
of $F$. Hence $F_{-n}=0$ for all $n\geq1$, and so, $F$ can be
extended to a holomorphic function on $D$. Therefore $\gamma$ can
be extended accross the puncture. 
\end{singlespace}
\end{proof}
\begin{singlespace}
\noindent A removable singularity result for $\mathcal{H}-$holomorphic
curves is the following
\end{singlespace}
\begin{lem}
\begin{singlespace}
\noindent \label{lem:Let--be_Removable_Singularity}Let $(D,i,\{0\},u,\gamma)$
be a $\mathcal{H}-$holomorphic curve defined on $D\backslash\{0\}$
such that the image of $u$ lies in a compact subset of $\mathbb{R}\times M$.
Then $u$ extends continously to a $\mathcal{H}-$holomorphic map
on the whole disk $D$.
\end{singlespace}
\end{lem}
\begin{proof}
\begin{singlespace}
\noindent Since $D$ is contractible and $d\gamma=d(\gamma\circ i)=0$,
the harmonic perturbation $\gamma$ can be written as $\gamma=d\Gamma$,
where $\Gamma:D\rightarrow\mathbb{R}$ is a harmonic function. Hence
$\overline{u}=(\overline{a},\overline{f}):=(a+\Gamma,f)$ is a pseudoholomorphic
curve (unperturbed), which still has the property that its image lies
in a (maybe larger) compact subset of $\mathbb{R}\times M$. Application
of the usual removable singularity theorem (see Lemma 5.5 of \cite{key-1})
finishes the proof of the lemma.
\end{singlespace}
\end{proof}
\begin{singlespace}
\noindent In a neighbourhood of a puncture, the map $a$ is either
bounded or unbounded. In the first case, Lemma \ref{lem:Let--be_Removable_Singularity}
can be used to extend the $\mathcal{H}-$holomorphic curve across
the puncture. In the second case, in which $a:D\backslash\{0\}\rightarrow\mathbb{R}$
is unbounded we have the following
\end{singlespace}
\begin{prop}
\begin{singlespace}
\noindent \label{prop:Let--be}Let $(D,i,\{0\},u,\gamma)$ be a $\mathcal{H}-$holomorphic
curve defined on $D\backslash\{0\}$ such that the image of $u$ is
unbounded in $\mathbb{R}\times M$. Then $u$ is asymptotic to a trivial
cylinder over a periodic orbit of $X_{\alpha}$, i.e. after identifying
$D\backslash\{0\}$ with the half open cylinder $[0,\infty)\times S^{1}$
there exists a periodic orbit $x$ of period $|T|$ of $X_{\alpha}$,
where $T\not=0$ such that 
\[
\lim_{s\rightarrow\infty}f(s,t)=x(Tt)\text{ and }\lim_{s\rightarrow\infty}\frac{a(s,t)}{s}=T\text{ in }C^{\infty}(S^{1})
\]
where $(s,t)$ denote the coordinates on $[0,\infty)\times S^{1}$.
\end{singlespace}
\end{prop}
\begin{proof}
\begin{singlespace}
\noindent As we restrict the curve to the disk, the harmonic perturbation
$\gamma$ is exact, i.e. there exists a harmonic function $\Gamma$
defined on the open unit disk such that $\gamma=d\Gamma$. The new
curve $\overline{u}=(\overline{a},\overline{f})=(a+\Gamma,f)$ is
pseudoholomorphic. Let $\psi:\mathbb{R}_{+}\times S^{1}\rightarrow D\backslash\{0\}$,
$(s,t)\mapsto e^{-2\pi(s+it)}$ be a biholomorphism which maps $D\backslash\{0\}$
to the cylinder $\mathbb{R}_{+}\times S^{1}$. We consider the pseudoholomorphic
curve $\overline{u}$ as beeing defined on the cylinder $\mathbb{R}_{+}\times S^{1}$
with finite energy and having an unbounded image in $\mathbb{R}\times M$.
Since we assumed that $X_{\alpha}$ is of Morse type, we obtain by
Proposition 5.6 of \cite{key-1}, that there exists $T\not=0$ and
a periodic orbit $x$ of $X_{\alpha}$ of period $|T|$ such that
\[
\lim_{s\rightarrow\infty}\overline{f}(s,t)=x(Tt)\text{ and }\lim_{s\rightarrow\infty}\frac{\overline{a}(s,t)}{s}=T\text{ in }C^{\infty}(S^{1}).
\]
By the boundedness of the harmonic function $\Gamma$, we have
\[
\lim_{s\rightarrow\infty}\frac{a(s,t)}{s}=T\text{ in }C^{\infty}(S^{1}),
\]
and the proof is finished.
\end{singlespace}
\end{proof}
\begin{singlespace}
\noindent The puncture $p\in\mathcal{P}$ is called \emph{positive}
or \emph{negative} depending on the sign of the coordinate function
$a$ when approaching the puncture. Keep in mind that the holomorphic
coordinates near the puncture affects only the choice of the origin
on the orbit $x$ of $X_{\alpha}$; the parametrization of the asymptotic
orbits induceed by the holomorphic polar coordinates remains otherwise
the same. Hence, the orientation induced on $x$ by the holomorphic
coordinates coincides with the orientation defined by the vector field
$X_{\alpha}$ if and only if the puncture is positive.

\noindent Let $S^{\mathcal{P}}$ be the oriented blow-up of $S$ at
the punctures $\mathcal{P}=\{p_{1},...,p_{k}\}$ as defined in the
previous section or in Section 4.3 of \cite{key-1}. $S^{\mathcal{P}}$
is a compact surface with boundary circles $\Gamma_{1},...,\Gamma_{k}$.
Noting that each of these circles is endowed with a canonical $S^{1}-$action
and letting $\varphi_{i}:S^{1}\rightarrow\Gamma_{i}$ be (up to a
choice of the base point) the canonical parametrization of the boundary
circle $\Gamma_{i}$, for $i=1,...,k$, we reformulate Proposition
\ref{prop:Let--be} as follows.
\end{singlespace}
\begin{prop}
\begin{singlespace}
\noindent \label{prop:Let--be-1}Let $(S,j,\mathcal{P},u,\gamma)$
be a $\mathcal{H}-$holomorphic map without removable singularities.
The map $f:\dot{S}\rightarrow M$ extends to a continous map $\overline{f}:S^{\mathcal{P}}\rightarrow M$
such that
\begin{equation}
\overline{f}(\varphi_{i}(e^{2\pi it}))=x_{i}(Tt),\label{eq:sign_orientation_periodic_orbit}
\end{equation}
where $x_{i}:S^{1}=\mathbb{R}/\mathbb{Z}\rightarrow M$ is a periodic
orbit of the Reeb vector field $X_{\alpha}$ of period $T$, parametrized
by the vector field $X_{\alpha}$. The sign of $T$ coincides with
the sign of the puncture $p_{i}\in\mathcal{P}$.
\end{singlespace}
\end{prop}
\begin{singlespace}

\subsection{\label{subsec:Notion-of-convergence}Notion of convergence}
\end{singlespace}
\begin{singlespace}

\subsubsection{\label{subsec:Stratified-broken-holomorphic}Stratified $\mathcal{H}-$holomorphic
buildings}
\end{singlespace}

\begin{singlespace}
\noindent In this section we introduce the notion of a stratified
$\mathcal{H}-$holomorphic building. These are the objects which are
needed for the compactification of the moduli space of $\mathcal{H}-$holomorphic
curves. First we define a $\mathcal{H}-$holomorphic building of height
$1$, and then we introduce the general notion of a $\mathcal{H}-$holomorphic
building of height greater than $1$. Let $(S,j)$ be a Riemann surface,
and $\underline{\mathcal{P}}\subset S$ and $\overline{\mathcal{P}}\subset S$
two disjoint unordered finite subsets called the sets of \emph{negative
}and\emph{ positive punctures, }respectively. Let $\underline{\mathcal{P}}=\{\underline{p}_{1},...,\underline{p}_{l}\}$,
$\overline{\mathcal{P}}=\{\overline{p}_{1},...,\overline{p}_{f}\}$
and $\mathcal{P}=\underline{\mathcal{P}}\amalg\overline{\mathcal{P}}$.
The set of nodes, defined by $\mathcal{D}=\{d'_{1},d''_{1},...,d'_{k},d''_{k}\}\subset S$,
is a finite subset of $S$, where the significance of the pair $\{d'_{i},d''_{i}\}$
will be clarified later on. Denote by $S^{\mathcal{P}}$ the blow-up
of the surface $\dot{S}=S\backslash\mathcal{P}$ at the punctures
$\mathcal{P}$. The surface $S^{\mathcal{P}}$ has $|\mathcal{P}|$
boundary components, which due to the splitting of $\mathcal{P}$,
are denoted by $\underline{\Gamma}=\{\underline{\Gamma}_{1},...,\underline{\Gamma}_{l}\}$
and $\overline{\Gamma}=\{\overline{\Gamma}_{1},...,\overline{\Gamma}_{f}\}$.
\end{singlespace}
\begin{defn}
\begin{singlespace}
\noindent \label{def:A-map-}$(S,j,u,\mathcal{P},\mathcal{D},\gamma,\tau,\sigma)$,
where $\tau=\{\tau_{i}\}_{i=1,...,|\mathcal{D}|/2}$, $\sigma=\{\sigma_{i}\}_{i=1,...,|\mathcal{D}|/2}$
and $\tau_{i},\sigma_{i}\in\mathbb{R}$ for all $i=1,...,|\mathcal{D}|/2$
is called a \emph{stratified} \emph{$\mathcal{H}-$holomorphic building
of height $1$} if the following conditions are satisfied.
\end{singlespace}
\begin{enumerate}
\begin{singlespace}
\item $(S,j,u,\mathcal{P},\gamma)$ is a $\mathcal{H}-$holomorphic curve
as defined in (\ref{eq:H-hol}).
\item For each $\{d'_{i},d''_{i}\}\in\mathcal{D}$, $\tau_{i},\sigma_{i}\in\mathbb{R}$
the points $u(d'_{i})$ and $u(d''_{i})$ are connected by the map
$[-1/2,1/2]\rightarrow\mathbb{R}\times M$, $s\mapsto(-2\sigma_{i}s+b,\phi_{-2\tau_{i}s}^{\alpha}(w_{f}))$
for some $b\in\mathbb{R}$ and $w_{f}\in M$ such that $u(d'_{i})=(\sigma_{i}+b,\phi_{\tau_{i}}^{\alpha}(w_{f}))$
and $u(d''_{i})=(-\sigma_{i}+b,\phi_{-\tau_{i}}^{\alpha}(w_{f}))$.
\end{singlespace}
\end{enumerate}
\end{defn}
\begin{rem}
\begin{singlespace}
\noindent The $M-$component $f:\dot{S}\rightarrow M$ of a stratified
$\mathcal{H}-$holomorphic building $u=(a,f):\dot{S}\rightarrow\mathbb{R}\times M$
of height $1$ can be continously extended to $S^{\mathcal{P}}$.
For the extension $\overline{f}:S^{\mathcal{P}}\rightarrow M$, it
is apparent that $\overline{f}|_{\Gamma}$, where $\Gamma=\underline{\Gamma}\amalg\overline{\Gamma}$,
defines parametrizations of Reeb orbits.
\end{singlespace}
\end{rem}
\begin{singlespace}
\noindent 

\end{singlespace}\begin{rem}
\begin{singlespace}
\noindent The energy of a stratified $\mathcal{H}-$holomorphic building
of height $1$ is the sum of the $\alpha-$ and $d\alpha-$energies
of the $\mathcal{H}-$holomorphic curve, as defined in (\ref{eq:energy}).
\end{singlespace}
\end{rem}
\begin{singlespace}
\noindent In a second step we define a stratified $\mathcal{H}-$holomorphic
building of height $N$. Let $(S_{1},j_{1}),...,(S_{N},j_{N})$ be
closed (possibly disconected) Riemann surfaces, and for any $i\in\{1,...,N\}$,
let $\underline{\mathcal{P}}_{i}\subset S_{i}$ and $\overline{\mathcal{P}}_{i}\subset S_{i}$
be the sets of \emph{negative }and\emph{ positive punctures on level
$i$, }respectively.\emph{ }There is a one-to-one correspondence between
the elements $\overline{\mathcal{P}}_{i-1}$ and $\underline{\mathcal{P}}_{i}$
given by a bijective map $\varphi_{i}:\overline{\mathcal{P}}_{i-1}\rightarrow\underline{\mathcal{P}}_{i}$.
The pair $\{\overline{p}_{i-1,j},\underline{p}_{ij}\}$, where $\underline{p}_{ij}=\varphi_{i}(\overline{p}_{i-1,j})$,
is called the \emph{breaking point} between the levels $S_{i-1}$
and $S_{i}$.

\noindent Let $\mathcal{P}=\coprod_{i=1}^{N}\underline{\mathcal{P}}_{i}\amalg\overline{\mathcal{P}}_{i}$
be the set of \emph{punctures,} $\mathcal{P}_{i}=\underline{\mathcal{P}}_{i}\amalg\overline{\mathcal{P}}_{i}$
the set\emph{ }of \emph{punctures at level $i$}, $\mathcal{D}_{i}=\{d'_{i1},d''_{i1},...,d'_{ik_{i}},d''_{ik_{i}}\}$
the set of \emph{nodes at level $i$,} and $\mathcal{D}=\coprod_{i=1}^{N}\mathcal{D}_{i}$
the set of all nodes.

\noindent If $S_{i}^{\mathcal{P}_{i}}$ is the blow-up of $S_{i}$
at the punctures $\mathcal{P}_{i}=\underline{\mathcal{P}}_{i}\amalg\overline{\mathcal{P}}_{i}$,
then accounting of the splitting of the punctures $\mathcal{P}_{i}$,
we denote the boundary components of $S_{i}^{\mathcal{P}_{i}}$ by
$\underline{\Gamma}_{i}$ and $\overline{\Gamma}_{i}$; they correspond
to the negative and positive punctures $\underline{\mathcal{P}}_{i}$
and $\overline{\mathcal{P}}_{i}$, respectively. There is a one-to-one
correspondence between the elements of $\overline{\Gamma}_{i-1}$
and $\underline{\Gamma}_{i}$ given by an orientation reversing diffeomorphism
$\Phi_{i}:\overline{\Gamma}_{i-1}\rightarrow\underline{\Gamma}_{i}$.
The pair $\{\overline{\Gamma}_{i-1,j},\underline{\Gamma}_{ij}\}$,
where $\underline{\Gamma}_{ij}=\Phi_{i}(\overline{\Gamma}_{i-1,j})$,
is called a \emph{breaking orbit} for all $i=2,...,N$. This gives
an identification of the boundary components $\overline{\Gamma}_{i-1}$
from $S_{i-1}^{\mathcal{P}_{i-1}}$ and the boundary components $\underline{\Gamma}_{i}$
from $S_{i}^{\mathcal{P}_{i}}$. Further on, let \[
S^{P,\Phi} := S_{1}^{P_{1}} \cup_{\Phi_{2}} S_{2}^{P_{2}} \cup_{\Phi_{3}} ... \cup_{\Phi_{N}} S_{N}^{P_{N}} := \bigslant{\left ( \coprod_{i=1}^{N} S_{i}^{P_{i}} \right )}{\sim}
\] where $\sim$ is defined by identifying the circles $\overline{\Gamma}_{i-1,j}$
and $\underline{\Gamma}_{ij}$ via the diffeomorphism $\Phi_{i}$
for all $i=2,...,N$ and $j=1,...,|\underline{\mathcal{P}}_{i}|$.
Obviously, $S^{\mathcal{P},\Phi}$ is a compact surface with $|\underline{\mathcal{P}}_{1}|+|\overline{\mathcal{P}}_{N}|$
boundary components. The equivalence class of $\overline{\Gamma}_{i-1,j}$
in $S^{\mathcal{P},\Phi}$, denoted by $\Gamma_{ij}$ for all $i=2,...,N$
and $j=1,...,|\underline{\mathcal{P}}_{i}|$, is called a \emph{special
circle}; the collection of all special circles is denoted by $\Gamma$.

\noindent A tuple $(S,j,\mathcal{P},\mathcal{D})$ with the properties
described above will be called a \emph{broken building of height $N$}.
We are now well prepared to introduce a stratified $\mathcal{H}-$holomorphic
building of height $N$. 
\end{singlespace}
\begin{defn}
\begin{singlespace}
\noindent \label{def:A-mapwhere-}A tuple $(S,j,u,\mathcal{P},\mathcal{D},\gamma,\tau,\sigma)$,
where $\tau=\{\hat{\tau}_{ij_{i}}\mid i=1,...,N\text{ and }j_{i}=1,...,|\mathcal{D}_{i}|/2\}\cup\{\tau_{ij_{i}}\mid i=1,...,N-1\text{ and }j_{i}=1,...,|\overline{\Gamma}_{i}|\}$,
$\sigma=\{\hat{\sigma}_{ij_{i}}\mid i=1,...,N\text{ and }j_{i}=1,...,|\mathcal{D}_{i}|/2\}$
and $(S,j,\mathcal{P},\mathcal{D})$ is a broken building of height
$N$, is called a \emph{stratified} \emph{$\mathcal{H}-$holomorphic
building of height $N$} if the following are satisfied:
\end{singlespace}
\begin{enumerate}
\begin{singlespace}
\item For any $i=1,...,N$, $(S_{i},j_{i},u_{i},\underline{\mathcal{P}}_{i}\amalg\overline{\mathcal{P}}_{i},\mathcal{D}_{i},\gamma_{i},\{\hat{\tau}_{ij_{i}}\mid j_{i}=1,...,|\mathcal{D}_{i}|/2\},\{\hat{\sigma}_{ij_{i}}\mid j_{i}=1,...,|\mathcal{D}_{i}|/2\})$
is a stratified $\mathcal{H}-$holomorphic building of height $1$,
where $u_{i}=u|_{S_{i}\backslash\mathcal{P}_{i}}$, and $j_{i}$ is
the complex structure on $S_{i}$. 
\item For all breaking points $\{\overline{p}_{i-1,j},\underline{p}_{ij}\}$
and $\tau_{ij}\in\tau$, there exist $T_{ij}>0$ such that the $\mathcal{H}-$holomorphic
building of height $1$, $u_{i-1}:\dot{S}_{i-1}\rightarrow\mathbb{R}\times M$
is asymptotic at $\overline{p}_{i-1,j}$ to a trivial cylinder over
the Reeb orbit $x_{ij}$ of period $T_{ij}>0$, and $u_{i}:\dot{S}_{i}\rightarrow\mathbb{R}\times M$
is asymptotic at $\underline{p}_{ij}$ to the trivial cylinder over
the Reeb orbit $x_{ij}(\cdot+\tau_{ij})$ of period $-T_{ij}<0$.
\end{singlespace}
\end{enumerate}
\end{defn}
\begin{rem}
\begin{singlespace}
\noindent The energy of a stratified $\mathcal{H}-$holomorphic buidling
of height $N$ is defined by
\[
E(u)=\max_{1\leq i\leq N}E_{\alpha}(u_{i})+\sum_{i=1}^{N}E_{d\alpha}(u_{i}).
\]
\end{singlespace}
\end{rem}
\begin{singlespace}

\subsubsection{\label{subsec:Generalized-Blow-up}Collar Blow-up}
\end{singlespace}

\begin{singlespace}
\noindent In this section we introduce a version of blow-up similar
to that from \cite{key-1}. Let $(S,j,\mathcal{P},\mathcal{D})$ be
a broken building of height $N$, and consider the setting used in
the previous section. In addition, let $S_{i}^{\mathcal{P}_{i}\cup\mathcal{D}_{i}}$
be the blow-up of $S_{i}$ at the punctures $\mathcal{P}_{i}$ and
nodes $\mathcal{D}_{i}$. To each pair of nodes $\{d'_{ij},d''_{ij}\},$
the corresponding boundary of $S_{i}^{\mathcal{P}_{i}\cup\mathcal{D}_{i}}$
is denoted by $\{\Gamma'_{ij},\Gamma''_{ij}\}$, and for each such
pair of boundary circles, let $r_{ij}:\Gamma'_{ij}\rightarrow\Gamma''_{ij}$
be orientation reversing diffeomorphisms. The diffeomorphisms $r_{ij}$
are used to glue the boundary circles $\Gamma'_{ij}$ and $\Gamma''_{ij}$
together. Consider the surface $\hat{S}:=S^{\mathcal{P}\cup\mathcal{D},\Phi\cup r}$
which is obtained from $S$ by blowing-up the punctures $\mathcal{P}$
and the nodes $\mathcal{D}$, and by using the orientation reversing
diffomorphisms $\Phi$ and $r_{ij}$. $\hat{S}$ is a compact surface
with boundary components given by the sets $\underline{\Gamma}_{1}$
and $\overline{\Gamma}_{N}$. The equivalence class of $\Gamma'_{ij}$
in $\hat{S}$ is denoted by $\Gamma_{ij}^{\text{nod}}$ and is called
\emph{nodal special circles}; the set of all nodal special circles
is denoted by $\Gamma^{\text{nod}}$. The \emph{collar blow-up} $\overline{S}$
is a modification of the usual blow-up $\hat{S}$ defined in \cite{key-1}.
Essentially, we insert the cylinders $[-1/2,1/2]\times S^{1}$ between
the special circles $\overline{\Gamma}_{i-1,j}$ and $\underline{\Gamma}_{ij}$,
and between $\Gamma'_{ij}$ and $\Gamma''_{ij}$. To obtain a surface
with boundary components $\underline{\Gamma}_{1}$ and $\overline{\Gamma}_{N}$
that has the same topology as $\hat{S}$ we modify the orientation
reversing the diffeomorphismsm $\Phi_{ij}$ and $r_{ij}$ as follows: 
\end{singlespace}
\begin{description}
\begin{singlespace}
\item [{B1}] \noindent The orientation reversing diffeomorphisms $\Phi_{ij}$
correspond to two orientation reversing diffeomorphisms $\overline{\Phi}_{ij}:\overline{\Gamma}_{i-1,j}\rightarrow\{-1/2\}\times S^{1}$
and $\underline{\Phi}_{ij}:\{1/2\}\times S^{1}\rightarrow\underline{\Gamma}_{ij}$
for all $i=2,...,N$ and $j=1,...,|\underline{\mathcal{P}}_{i}|$. 
\item [{B2}] \noindent Instead of glueing $\overline{\Gamma}_{i-1,j}$
and $\underline{\Gamma}_{ij}$ via the orientation reversing diffeomorphisms
$\Phi_{ij}$, we glue $\overline{\Gamma}_{i-1,j}$, the cylinder $[-1/2,1/2]\times S^{1}$,
and $\underline{\Gamma}_{ij}$ via the orientation reversing diffeomorphisms
$\overline{\Phi}_{ij}$ and $\underline{\Phi}_{ij}$ (see Figure \ref{fig28}).
\begin{figure}     
	\centering  
	\def\svgscale{1} 	
	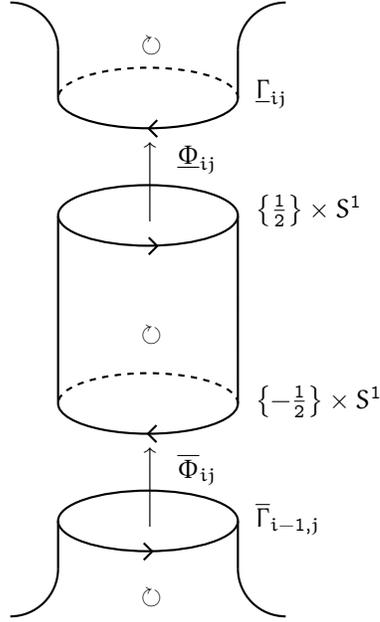  
	\caption{The glueing of $\overline{\Gamma}_{i-1,j}$, the cylinder $[-1/2,1/2]\times S^{1}$ and $\underline{\Gamma}_{ij}$ via the orientation reversing diffeomorphisms $\overline{\Phi}_{ij} : \overline{\Gamma}_{i-1,j} \rightarrow \{-1/2\}\times S^{1}$ and $\underline{\Phi}_{ij} : \{ 1/2 \}\times S^{1} \rightarrow \underline{\Gamma}_{i,j}$.}
	\label{fig28}
\end{figure}
\item [{B3}] \noindent For the nodal special circles $\Gamma'_{ij}$ and
$\Gamma''_{ij}$, we proceed analogously, and denote by $r'_{ij}:\Gamma'_{ij}\rightarrow\{-1/2\}\times S^{1}$
and $r''_{ij}:\{1/2\}\times S^{1}\rightarrow\Gamma''_{ij}$ the orientation
reversing diffeomorphisms that glue $\Gamma'_{ij}$, the cylinder
$[-1/2,1/2]\times S^{1}$ and $\Gamma''_{ij}$ together. 
\end{singlespace}
\end{description}
\begin{singlespace}
\noindent Let $\overline{S}$ be the surface obtained by applying
the above construction to all special and nodal special circles. The
equivalence class of the cylinder $[-1/2,1/2]\times S^{1}$ in $\overline{S}$
corresponding to the special circle $\Gamma_{ij}$ is denoted by $A_{ij}$,
and is called \emph{special cylinder}. The equivalence class of the
cylinder $[-1/2,1/2]\times S^{1}$ in $\overline{S}$ corresponding
to the nodal special circle $\Gamma_{ij}^{\text{nod}}$ is denoted
by $A_{ij}^{\text{nod}}$, and is called \emph{nodal special cylinder}.
The boundary circles of $A_{ij}$ are still denoted by $\overline{\Gamma}_{i-1,j}$
and $\underline{\Gamma}_{ij}$, while the boundary circles of $A_{ij}^{\text{nod}}$
are also still denoted by $\Gamma'_{ij}$ and $\Gamma''_{ij}$. Finally,
the collections of all special and nodal special cylinders are denoted
by $A$ and $A^{\text{nod}}$, respectively. Take notice that there
exists a natural projection between the collar blow-up $\overline{S}$
and the blow-up surface $\hat{S}$, which is defined similarly to
\cite{key-1}, i.e. it maps $\overline{S}\backslash(A\amalg A^{\text{nod}})$
diffeomorphically to $\hat{S}\backslash(\Gamma\amalg\Gamma^{\text{nod}})$
and the annulli $A$ and $A^{\text{nod}}$ are mapped to $\Gamma$
and $\Gamma^{\text{nod}}$. This induces a conformal structure on
$\overline{S}\backslash(A\amalg A^{\text{nod}})$. Let $\tilde{S}$
be the closed surface obtained from $\overline{S}$ by identifying
the boundary components $\underline{\Gamma}_{1}$ and $\overline{\Gamma}_{N}$
to points.

\noindent Having now a stratified $\mathcal{H}-$holomorphic building
$(S,j,u,\mathcal{P},\mathcal{D},\gamma,\tau,\sigma)$ of height $N$,
we define the continous extension $\overline{f}$ of $f$ on the surface
$\overline{S}$ and the continous extension $\overline{a}$ of $a$
on $\overline{S}\backslash A$. The extension $\overline{f}$ may
be defined on the clinders $A_{ij}$ and $A_{ij}^{\text{nod}}$, while
the extension $\overline{a}$ is defined only on $A_{ij}^{\text{nod}}$.
Set
\begin{align*}
\overline{f}(s,t) & =\phi_{-2s\hat{\tau}_{ij}}^{\alpha}(w_{f}),\ \text{ for all }(s,t)\in A_{ij}^{\text{nod}}=[-1/2,1/2]\times S^{1},\\
\overline{f}(s,t) & =\phi_{-\left(s+\frac{1}{2}\right)\tau_{ij}}^{\alpha}(x_{ij}(T_{ij}t)),\ \text{ for all }(s,t)\in A_{ij}=[-1/2,1/2]\times S^{1}
\end{align*}
and
\[
\overline{a}(s,t)=2\hat{\sigma}_{ij}s+b,\ \text{ for all }(s,t)\in A_{ij}^{\text{nod}}=[-1/2,1/2]\times S^{1}
\]
for some $b\in\mathbb{R}$ and $w_{f}\in M$. Here $x_{ij}$ is the
Reeb orbit of period $T_{ij}>0$. 
\end{singlespace}
\begin{singlespace}

\subsubsection{\label{subsec:Convergence}Convergence}
\end{singlespace}

\begin{singlespace}
\noindent In this section we define the notion of convergence using
the notation from the previous two sections.
\end{singlespace}
\begin{defn}
\begin{singlespace}
\noindent \label{def:A-sequence-of}A sequence of $\mathcal{H}-$holomorphic
curves $(S_{n},j_{n},u_{n},\mathcal{P}_{n}'=\underline{\mathcal{P}}_{n}'\amalg\overline{\mathcal{P}}_{n}',\gamma_{n})$
converges in the $C_{\text{loc}}^{\infty}$ sense to a $\mathcal{H}-$holomorphic
curve $(S,j,u,\mathcal{P},\mathcal{D},\gamma)$, if the tuple $(S,j,\mathcal{P},\mathcal{D})$
is a broken building of height $N$ and there exists a sequence of
diffeomorphisms $\varphi_{n}:\tilde{S}\rightarrow S_{n}$, where $\tilde{S}$
is the modified collar blow-up defined in Section \ref{subsec:Generalized-Blow-up},
such that $\varphi_{n}^{-1}(\overline{\mathcal{P}}_{n}')=\overline{\mathcal{P}}_{1}$
and $\varphi_{n}^{-1}(\underline{\mathcal{P}}_{n}')=\underline{\mathcal{P}}_{N}$
and such that the following conditions are satisfied:
\end{singlespace}
\begin{enumerate}
\begin{singlespace}
\item The sequence of complex structures $(\varphi_{n})_{*}j_{n}$ converges
in $C_{\text{loc}}^{\infty}$ on $\tilde{S}\backslash(A\amalg A^{\text{nod}})$
to $j$.
\item The special circles of $(S_{n},j_{n},\mathcal{P}_{n})$ are mapped
by $\varphi_{n}^{-1}$ bijectively onto $\{0\}\times S^{1}$ of $A_{ij}$
or $A_{ij}^{\text{nod}}$. For every special cylinder $A_{ij}$ there
exists an annulus $\overline{A}_{ij}\cong[-1,1]\times S^{1}$ such
that $A_{ij}\subset\overline{A}_{ij}$ and $(\overline{A}_{ij},(\varphi_{n})_{*}j_{n})$
and $(A_{ij},(\varphi_{n})_{*}j_{n})$ are conformally equivalent
to $([-R_{n},R_{n}]\times S^{1},i)$ and $([-R_{n}+h_{n},R_{n}-h_{n}]\times S^{1},i)$,
respectively, where $R_{n},h_{n},R_{n}/h_{n}\rightarrow\infty$ as
$n\rightarrow\infty$, $i$ is the standard complex structure and
the diffeomorphisms are of the form $(s,t)\mapsto(\kappa(s),t)$.
\item The $\mathcal{H}-$holomorphic curves $u_{n}\circ\varphi_{n}:\dot{\tilde{S}}:=\tilde{S}\backslash(\overline{\mathcal{P}}_{1}\amalg\underline{\mathcal{P}}_{N})\rightarrow\mathbb{R}\times M$
together with the harmonic perturbation $(\varphi_{n})^{*}\gamma_{n}$
which are defined on $\tilde{S}$ converge in $C_{\text{loc}}^{\infty}$
on $\dot{\tilde{S}}\backslash\left(A\amalg A^{\text{nod}}\right)$
to the $\mathcal{H}-$holomorphic curve $u$ with harmonic perturbation
$\gamma$. Note that $\dot{\tilde{S}}\backslash(A\amalg A^{\text{nod}})$
may be conformally identified with $S\backslash(\mathcal{P}\amalg\mathcal{D})$.
\end{singlespace}
\end{enumerate}
\end{defn}
\begin{singlespace}
\noindent Next we describe the $C^{0}-$convergence. Let $(S_{n},j_{n},u_{n},\mathcal{P}_{n}',\gamma_{n})$
be a sequence of $\mathcal{H}-$holomorphic curves. For any special
circle $\Gamma_{ij}$, let $\tau_{ij}^{n}\in\mathbb{R}$ and $\sigma_{ij}^{n}\in\mathbb{R}$
be the conformal period of $\varphi_{n}^{*}\gamma_{n}$ on $\Gamma_{ij}$
with respect to the complex structure $\varphi_{n}^{*}j_{n}$, and
the conformal co-period of $\varphi_{n}^{*}\gamma_{n}$ on $\Gamma_{ij}$
with respect to the complex structure $\varphi_{n}^{*}j_{n}$, respectively.
For any nodal special circle $\Gamma_{ij}^{\text{nod}}$ consider
the numbers $\hat{\tau}_{ij}^{n}\in\mathbb{R}$ and $\hat{\sigma}_{ij}^{n}\in\mathbb{R}$,
where $\hat{\tau}_{ij}^{n}$ is the conformal period of $\varphi_{n}^{*}\gamma_{n}$
on $\Gamma_{ij}^{\text{nod}}$ with respect to the complex structure
$\varphi_{n}^{*}j_{n}$, and $\hat{\sigma}_{ij}^{n}$ is the conformal
co-period of $\varphi_{n}^{*}\gamma_{n}$ on $\Gamma_{ij}^{\text{nod}}$
with respect to the complex structure $\varphi_{n}^{*}j_{n}$, respectively.
\end{singlespace}
\begin{rem}
\begin{singlespace}
\noindent For a sequence $(S_{n},j_{n},u_{n},\mathcal{P}_{n}',\gamma_{n})$
of $\mathcal{H}-$holomorphic curves that converges to a $\mathcal{H}-$holomorphic
curve $(S,j,u,\mathcal{P},\mathcal{D},\gamma)$ in the sense of Definition
\ref{def:A-sequence-of}, the quantities $\tau_{ij}^{n}$, $\sigma_{ij}^{n}$,
$\hat{\tau}_{ij}^{n}$ and $\hat{\sigma}_{ij}^{n}$ can be unbounded
(see, e.g, Section \ref{chap:Discussion-on-conformal}). If $\tau_{ij}^{n}$,
$\sigma_{ij}^{n}$, $\hat{\tau}_{ij}^{n}$ and $\hat{\sigma}_{ij}^{n}$
are bounded, then after going over to a further subsequence, and assuming
that there exist the real numbers $\tau_{ij},\sigma_{ij},\hat{\tau}_{ij},\hat{\sigma}_{ij}\in\mathbb{R}$
such that
\begin{align}
\tau_{ij}^{n} & \rightarrow\tau_{ij},\label{eq:con_period}\\
\sigma_{ij}^{n} & \rightarrow\sigma_{ij},\label{eq:con_coperiod}\\
\hat{\tau}_{ij}^{n} & \rightarrow\hat{\tau}_{ij},\label{eq:con_period_nodes}\\
\hat{\sigma}_{ij}^{n} & \rightarrow\hat{\sigma}_{ij}\label{eq:con_period_coperiod}
\end{align}
as $n\rightarrow\infty$, we are able to derive a $C^{0}-$ convergence
result.
\end{singlespace}
\end{rem}
\begin{singlespace}
\noindent The convergence of a sequence of $\mathcal{H}-$holomorphic
curves to a stratified $\mathcal{H}-$holomorphic building of height
$N$ should be understood in the following sense:
\end{singlespace}
\begin{defn}
\begin{singlespace}
\noindent \label{def:A-sequence-of-1}A sequence of $\mathcal{H}-$holomorphic
curves $(S_{n},j_{n},\mathcal{P}_{n}',u_{n},\gamma_{n})$ converges
in the $C^{0}$ sense to a statified $\mathcal{H}-$holomorphic building
$(S,j,u,\mathcal{P},\mathcal{D},\gamma,\tau,\sigma)$ of height $N$
if the following conditions are satisfied.
\end{singlespace}
\begin{enumerate}
\begin{singlespace}
\item The parameters $\tau_{ij}^{n}$, $\sigma_{ij}^{n}$, $\hat{\tau}_{ij}^{n}$
and $\hat{\sigma}_{ij}^{n}$ converge as in (\ref{eq:con_period})-(\ref{eq:con_period_coperiod}).
\item The sequence $(S_{n},j_{n},\mathcal{P}_{n}',u_{n},\gamma_{n})$ converges
to the underlying $\mathcal{H}-$holomorphic curve $(S,j,u,\mathcal{P},\mathcal{D},\gamma)$
in the sense of Definition \ref{def:A-sequence-of} with respect to
a sequence of diffeomorphisms $\varphi_{n}:\tilde{S}\rightarrow S_{n}$.
\item $(S,j,u,\mathcal{P},\mathcal{D},\gamma,\tau,\sigma)$ is a stratified
$\mathcal{H}-$holomorphic building of height $N$ corresponding to
the constants $\tau_{ij}$, $\hat{\tau}_{ij}$ and $\hat{\sigma}_{ij}$,
as in Definition \ref{def:A-mapwhere-}. 
\item The maps $u_{n}\circ\varphi_{n}$ converges in $C_{\text{loc}}^{0}$
on $\overline{S}\backslash A$ to the blow-up map $\overline{u}$
defined on $\dot{\tilde{S}}\backslash A$.
\item The maps $f_{n}\circ\varphi_{n}$ converges in $C^{0}$ on $\overline{S}$
to the blow-up map $\overline{f}$ defined on $\overline{S}$.
\end{singlespace}
\end{enumerate}
\end{defn}
\begin{singlespace}

\section{\label{chap:Compactness-results}Proof of the Compactness Theorem}
\end{singlespace}

\begin{singlespace}
\noindent Let $(S_{n},j_{n},u_{n},\mathcal{P}_{n}',\gamma_{n})$ be
a sequence of $\mathcal{H}-$holomorphic curves satisfying Assumptions
A1 and A2. After introducing an additional finite set of points $\mathcal{M}_{n}$
disjoint from the set of punctures $\mathcal{P}_{n}'$ we assume that
the domains $(S_{n},j_{n},\mathcal{P}_{n}'\amalg\mathcal{M}_{n})$
of the sequence of $\mathcal{H}-$holomorphic curves are stable. This
condition enables us to use the Deligne-Mumford convergence (see Section
\ref{sec:Harmonic-Deligne-Mumford-converg}) which makes it possible
to formulate a convergence result for the domains $(S_{n},j_{n},\mathcal{P}_{n}'\amalg\mathcal{M}_{n})$.
Note that $\mathcal{M}_{n}$ can be choosen in such a way that their
cardinality is independent of the index $n$. As an additional structure,
let $h^{j_{n}}$ be the hyperbolic metric on $\dot{S}_{n}:=S_{n}\backslash(\mathcal{P}_{n}'\amalg\mathcal{M}_{n})$.
By the Deligne-Mumford convergence result (Corollary \ref{cor:Any-sequence-of})
there exists a stable nodal decorated surface $(S,j,\mathcal{P}\amalg\mathcal{M},\mathcal{D},r)$
and a sequence of diffeomorphisms $\varphi_{n}:S^{\mathcal{D},r}\rightarrow S_{n}$,
where $S^{\mathcal{D},r}$ is the closed surface obtained by blowing
up the nodes and glueing pairs of nodal points according to the decoration
$r$ as described in Section \ref{sec:Harmonic-Deligne-Mumford-converg},
such that the following holds: Let $h$ be the hyperbolic metric on
$S\backslash(\mathcal{P}\amalg\mathcal{M}\amalg\mathcal{D})$. The
diffeomorphisms $\varphi_{n}$ map marked points into marked points
and punctures into punctures, i.e. $\varphi_{n}(\mathcal{M})=\mathcal{M}_{n}$
and $\varphi_{n}(\mathcal{P})=\mathcal{P}_{n}'$. Via $\varphi_{n}$
we pull-back the complex structures $j_{n}$ and the hyperbolic metrics
$h^{j_{n}}$, i.e. we define $j^{(n)}:=\varphi_{n}^{*}j_{n}$ on $S^{\mathcal{D},r}$
and $h_{n}:=\varphi_{n}^{*}h^{j_{n}}$ on $\dot{S}^{\mathcal{D},r}:=S^{\mathcal{D},r}\backslash(\mathcal{M}\amalg\mathcal{P})$.
By the Deligne-Mumford convergence, $h_{n}\rightarrow h$ in $C_{\text{loc}}^{\infty}(\dot{S}^{\mathcal{D},r}\backslash\coprod_{j}\Gamma_{j})$
as $n\rightarrow\infty$, where $\Gamma_{j}$ are the special circles
in $S^{\mathcal{D},r}$ (see Section \ref{sec:Harmonic-Deligne-Mumford-converg}
for the definition of special circles). This yields $j^{(n)}\rightarrow j$
in $C_{\text{loc}}^{\infty}(S^{\mathcal{D},r}\backslash\coprod_{j}\Gamma_{j})$
as $n\rightarrow\infty$.

\noindent Consider now the maps $\tilde{u}_{n}=(\tilde{a}_{n},\tilde{f}_{n}):=u_{n}\circ\varphi_{n}:S^{\mathcal{D},r}\backslash\mathcal{P}\rightarrow\mathbb{R}\times M$
and $\tilde{\gamma}_{n}:=\varphi_{n}^{*}\gamma_{n}\in\mathcal{H}_{j^{(n)}}^{1}(S^{\mathcal{D},r})$.
Then $\tilde{u}_{n}$ is a $\mathcal{H}-$holomorphic curve with harmonic
perturbation $\tilde{\gamma}_{n}$; it satisfies the equation
\[
\begin{array}{cl}
\pi_{\alpha}d\tilde{f}_{n}\circ j^{(n)} & =J(\tilde{f}_{n})\circ\pi_{\alpha}d\tilde{f}_{n}\\
(\tilde{f}_{n}^{*}\alpha)\circ j^{(n)} & =d\tilde{a}_{n}+\tilde{\gamma}_{n}
\end{array}\text{ on }S^{\mathcal{D},r}\backslash\mathcal{P}
\]
and has uniformly bounded energies, i.e. for $E_{0}>0$ and all $n\in\mathbb{N}$
we have $E(\tilde{u}_{n};S^{\mathcal{D},r}\backslash\mathcal{P})\leq E_{0}$.
The $L^{2}-$norm of $\tilde{\gamma}_{n}$ over $S^{\mathcal{D},r}$
is equal to the $L^{2}-$norm of $\gamma_{n}$ over $S_{n}$ and it
is apparent that the $L^{2}-$norm of $\tilde{\gamma}_{n}$ is uniformly
bounded by the constant $C_{0}>0$. Hence A1 and A2 are satisfied
for $\tilde{u}_{n}$.

\noindent In the following, we first establish a convergence result
on the thick part, i.e. on $S^{\mathcal{D},r}$ away from special
circles, punctures and certain additional marked points, and then
treat the components from the thin part. With the convergence on the
thick components, the first statement of Theorem \ref{thm:Let--be-2-1}
is proved. 
\end{singlespace}
\begin{singlespace}

\subsection{\label{2_1_sec:First_Theorem}The Thick Part}
\end{singlespace}

\begin{singlespace}
\noindent For the sequence $\tilde{u}_{n}:S^{\mathcal{D},r}\backslash\mathcal{P}\rightarrow\mathbb{R}\times M$
as defined above, we prove the $C_{\text{loc}}^{\infty}-$convergence
in the complement of the special circles and of a finite collection
of points in $\dot{S}^{\mathcal{D},r}:=S^{\mathcal{D},r}\backslash(\mathcal{P}\amalg\mathcal{M})$.
Set $\dot{\tilde{S}}^{\mathcal{D},r}:=\dot{S}^{\mathcal{D},r}\backslash\coprod_{j}\Gamma_{j}$.
To simplify the notation we continue to denote the maps $\tilde{u}_{n}$
by $u_{n}$ and $\tilde{\gamma}_{n}$ by $\gamma_{n}$. The main result
of this section is the following
\end{singlespace}
\begin{thm}
\begin{singlespace}
\noindent \label{thm:There-exists-finitely}There exists a subsequence
of $u_{n}$, still denoted by $u_{n}$, a finite subset $\mathcal{Z}\subset\dot{\tilde{S}}^{\mathcal{D},r}$,
and a $\mathcal{H}-$holomorphic curve $u:\dot{\tilde{S}}^{\mathcal{D},r}\backslash\mathcal{Z}\rightarrow\mathbb{R}\times M$
with harmonic perturbation $\gamma$ defined on $\dot{\tilde{S}}^{\mathcal{D},r}$
with respect to the complex structure $j$ such that $u_{n}\rightarrow u$
and $\gamma_{n}\rightarrow\gamma$ in $C_{\text{loc}}^{\infty}(\dot{\tilde{S}}^{\mathcal{D},r}\backslash\mathcal{Z})$.
\end{singlespace}
\end{thm}
\begin{singlespace}
\noindent Before proving Theorem \ref{thm:There-exists-finitely}
we establish some preliminary results.

\noindent Assume that there exists a point $z^{1}\in\mathcal{K}\subset\dot{\tilde{S}}^{\mathcal{D},r}$,
where $\mathcal{K}$ is compact, and a sequence $z_{n}\in\mathcal{K}$
such that 
\[
z_{n}\rightarrow z^{1}\ \text{ and }\ \left\Vert du_{n}(z_{n})\right\Vert \rightarrow\infty
\]
as $n\rightarrow\infty$. The next lemma describing the convergence
of conformal structures on Riemann surfaces is similar to Lemma 10.7
of \cite{key-1}.
\end{singlespace}
\begin{lem}
\begin{singlespace}
\noindent \label{2_1_lem:There-exists-open-neighbourhood-1}There
exist the open neighbourhoods $U_{n}(z^{1})=U_{n}$ and $U(z^{1})=U$
of $z^{1}$, and the diffeomorphisms 
\[
\psi_{n}:D\rightarrow U_{n},\ \psi:D\rightarrow U
\]
such that 

\end{singlespace}\begin{enumerate}
\begin{singlespace}
\item $\psi_{n}$ are $i-j^{(n)}-$biholomorphisms and $\psi$ is a $i-j-$biholomorphism; 
\item \label{2_1_enu:bounded_grad_of_coordinate_maps_psi_n-1}$\psi_{n}\rightarrow\psi$
in $C_{\text{loc}}^{\infty}(D)$ as $n\rightarrow\infty$ with respect
to the Euclidean metric on $D$ and the hyperbolic metric $h$ on
their images; 
\item $\psi_{n}(0)=z^{1}$ for every $n$ and $\psi(0)=z^{1}$; 
\item $z_{n}\in U_{n}$ for every sufficiently large $n$; 
\item $z^{(n)}:=\psi_{n}^{-1}(z_{n})\rightarrow0$ as $n\rightarrow\infty$. 
\end{singlespace}
\end{enumerate}
\end{lem}
\begin{proof}
\begin{singlespace}
\noindent Lemma \ref{2_1_lem:There-exists-open-neighbourhood} applied
to the compact Riemann surface with boundary $\mathcal{K}$ and the
interior point $z^{1}$, yields the diffeomorphisms $\psi_{n}:D\rightarrow U_{n}$
and $\psi:D\rightarrow U$ for which the first three assertions hold
true. The fourth and fifth assertions are obvious since $z_{n}$ converge
to $z^{1}$. 
\end{singlespace}
\end{proof}
\begin{rem}
\begin{singlespace}
\noindent The coordinate maps $\psi_{n}$ and $\psi$ have uniformly
bounded gradients with respect to the Euclidian metric on $D$ and
the hyperbolic metric $h$ on their images. This follows from the
second assertion of Lemma \ref{2_1_lem:There-exists-open-neighbourhood-1}.
\end{singlespace}
\end{rem}
\begin{singlespace}
\noindent Let $\hbar>0$ be defined by (\ref{eq:hbar}). The next
lemma essentially states that the $d\alpha-$energy concentrates around
the point $z^{1}$ and is at least $\hbar/2>0$. The proof relies
on bubbling-off analysis and proceeds as in Section 5.6 of \cite{key-1}.
\end{singlespace}
\begin{lem}
\begin{singlespace}
\noindent \label{2_1_lem:There-exists-a}For every open neighbourhood
$U(z^{1})=U\subset\dot{\tilde{S}}^{\mathcal{D},r}$ we have 
\[
0<\hbar\leq\lim_{n\rightarrow\infty}E_{d\alpha}(u_{n};U)\leq E_{0}.
\]
In particular, for each open neighbourhood $U$ of $z^{1}$ there
exists an integer $N_{1}\in\mathbb{N}$ such that for all $n\geq N_{1}$
we have
\[
E_{d\alpha}(u_{n};U)\geq\frac{\hbar}{2}.
\]
\end{singlespace}
\end{lem}
\begin{proof}
\begin{singlespace}
\noindent Consider the maps $\hat{u}_{n}:=u_{n}\circ\psi_{n}:D\rightarrow\mathbb{R}\times M$,
where $\psi_{n}$ are the biholomorphisms given by Lemma \ref{2_1_lem:There-exists-open-neighbourhood-1}.
They satisfy the $\mathcal{H}-$holomorphic equations 
\[
\begin{array}{cl}
\pi_{\alpha}d\hat{f}_{n}\circ i & =J(\hat{f}_{n})\circ\pi_{\alpha}d\hat{f}_{n}\\
(\hat{f}_{n}^{*}\alpha)\circ i & =d\hat{a}_{n}+\hat{\gamma}_{n}
\end{array}\text{ on }D,
\]
where $\hat{\gamma}_{n}:=\psi_{n}^{*}\gamma_{n}$ is a harmonic $1-$form
on $D$ with respect to $i$. The energy of $\hat{u}_{n}$ on $D$
is uniformly bounded as $E(\hat{u}_{n};D)\leq E_{0}$, while the $L_{2}-$norm
of the $i-$harmonic $1-$form $\hat{\gamma}_{n}$ is uniformly bounded
on $D$ by the constant $C_{0}$. Furthermore, for $z^{(n)}:=\psi_{n}^{-1}(z_{n})$,
$\left\Vert d\hat{u}_{n}(z^{(n)})\right\Vert \rightarrow\infty$ as
$n\rightarrow\infty$. This can be seen as follows. If $v_{n}\in T_{z^{(n)}}D$
with $\left\Vert v_{n}\right\Vert _{\text{eucl.}}=1$ is such that
\[
\left\Vert du_{n}(z_{n})\frac{d\psi_{n}(z^{(n)})v_{n}}{\left\Vert d\psi_{n}(z^{(n)})v_{n}\right\Vert _{h_{n}}}\right\Vert _{\overline{g}}=\left\Vert du_{n}(z_{n})\right\Vert ,
\]
then, 
\begin{eqnarray*}
\left\Vert d\hat{u}_{n}(z^{(n)})v_{n}\right\Vert _{\tilde{g}} & = & \left\Vert du_{n}(z_{n})\frac{d\psi_{n}(z^{(n)})v_{n}}{\left\Vert d\psi_{n}(z^{(n)})v_{n}\right\Vert _{h_{n}}}\right\Vert _{\overline{g}}\left\Vert d\psi_{n}(z^{(n)})v_{n}\right\Vert _{h_{n}}\\
 & = & \left\Vert du_{n}(z_{n})\right\Vert \left\Vert d\psi_{n}(z^{(n)})v_{n}\right\Vert _{h_{n}}\\
 & \geq & \left\Vert du_{n}(z_{n})\right\Vert \frac{1}{2}\left\Vert d\psi_{n}(z^{(n)})\right\Vert \\
 & \geq & \left\Vert du_{n}(z_{n})\right\Vert \frac{1}{4}\left\Vert d\psi(0)\right\Vert \rightarrow\infty
\end{eqnarray*}
as $n\rightarrow\infty$. The first inequality follows from the $i-j^{(n)}-$holomorphicity
of $\psi_{n}$. Applying Hofer's topological lemma (Lemma 2.39 of
\cite{key-6}), we obtain another sequence $\{\tilde{z}^{(n)}\}_{n\in\mathbb{N}}\subset D$
with $\tilde{z}^{(n)}\rightarrow0$, $R_{n}:=\left\Vert d\hat{u}_{n}(\tilde{z}^{(n)})\right\Vert \rightarrow\infty$,
$\epsilon_{n}\searrow0$, $\epsilon_{n}R_{n}\rightarrow\infty$ as
$n\rightarrow\infty$ and $\left\Vert d\hat{u}_{n}(z)\right\Vert \leq2R_{n}$
for all $z\in D_{\epsilon_{n}}(\tilde{z}^{(n)})$. Doing rescaling
we define the maps 
\[
v_{n}(z)=(b_{n}(z),g_{n}(z)):=\left(\hat{a}_{n}\left(\tilde{z}^{(n)}+\frac{z}{R_{n}}\right)-\hat{a}_{n}(\tilde{z}^{(n)}),\hat{f}_{n}\left(\tilde{z}^{(n)}+\frac{z}{R_{n}}\right)\right)
\]
for all $z\in D_{\epsilon_{n}R_{n}}(0)$. The maps $v_{n}=(b_{n},g_{n}):D_{\epsilon_{n}R_{n}}(0)\rightarrow\mathbb{R}\times M$
satisfy $\left\Vert dv_{n}(0)\right\Vert =1$ and $\left\Vert dv_{n}(z)\right\Vert \leq2$
for all $z\in D_{\epsilon_{n}R_{n}}(0)$, and we have
\begin{align*}
E_{\alpha}(v_{n};D_{\epsilon_{n}R_{n}}(0))=E_{\alpha}(\hat{u}_{n};D_{\epsilon_{n}}(\tilde{z}^{(n)}))\leq E_{\alpha}(\hat{u}_{n};D)
\end{align*}
and 
\begin{align*}
E_{d\alpha}(v_{n};D_{\epsilon_{n}R_{n}}(0))=E_{d\alpha}(\hat{u}_{n};D_{\epsilon_{n}}(\tilde{z}^{(n)}))\leq E_{d\alpha}(\hat{u}_{n};D)
\end{align*}
giving $E(v_{n};D_{\epsilon_{n}R_{n}}(0))\leq E_{0}$. Moreover, $v_{n}$
solves the $\mathcal{H}-$holomorphic equations 
\[
\begin{array}{cl}
\pi_{\alpha}dg_{n}\circ i & =J\circ\pi_{\alpha}dg_{n},\\
(g_{n}^{*}\alpha)\circ i & =db_{n}+\underline{\gamma}_{n},
\end{array}
\]
where $\underline{\gamma}_{n}:=\hat{\gamma}_{n}/R_{n}$. Because $v_{n}$
has a bounded gradient, there exists a smooth map $v:\mathbb{C}\rightarrow\mathbb{R}\times M$
with a bounded energy (by $E_{0}$) such that $v_{n}\rightarrow v$
in $C_{\text{loc}}^{\infty}(\mathbb{C})$ as $n\rightarrow\infty$.
Nevertheless, because $\hat{\gamma}_{n}$ is bounded in $L^{2}-$norm,
$\underline{\gamma}_{n}\rightarrow0$ as $n\rightarrow0$. Thus $v=(b,g):\mathbb{C}\rightarrow\mathbb{R}\times M$
is a pseudoholomorphic plane, i.e. it solves the pseudoholomorphic
curve equation 
\[
\begin{array}{cl}
\pi_{\alpha}dg\circ i & =J\circ\pi_{\alpha}dg,\\
(g^{*}\alpha)\circ i & =db.
\end{array}
\]
We prove now that the $\alpha-$ and $d\alpha-$energies of $v$ are
bounded. Let $R>0$ be arbitrary and for some $\tau_{0}\in\mathcal{A}$
consider
\begin{align*}
\int_{D_{R}(0)}\tau'_{0}(b)db\circ i\wedge db & =\lim_{n\rightarrow\infty}\int_{D_{R}(0)}\tau_{0}'(b_{n})db_{n}\circ i\wedge db_{n}\\
 & =\lim_{n\rightarrow\infty}\int_{D_{R/R_{n}}(\tilde{z}^{(n)})}\tau'_{0}(\hat{a}_{n}-\hat{a}_{n}(\tilde{z}^{(n)}))d\hat{a}_{n}\circ i\wedge d\hat{a}_{n}\\
 & =\lim_{n\rightarrow\infty}\int_{D_{R/R_{n}}(\tilde{z}^{(n)})}\tau'_{n}(\hat{a}_{n})d\hat{a}_{n}\circ i\wedge d\hat{a}_{n}\\
 & \leq\lim_{n\rightarrow\infty}\sup_{\tau\in\mathcal{A}}\int_{D_{R/R_{n}}(\tilde{z}^{(n)})}\tau'(\hat{a}_{n})d\hat{a}_{n}\circ i\wedge d\hat{a}_{n}\\
 & =\lim_{n\rightarrow\infty}E_{\alpha}(\hat{u}_{n};D_{R/R_{n}}(\tilde{z}^{(n)})),
\end{align*}
where $\tau_{n}=\tau_{0}(\cdot-\hat{a}_{n}(\tilde{z}^{(n)}))$ is
a sequence of functions that belong to $\mathcal{A}$. Taking the
supremum of the left-hand side over $\tau\in\mathcal{A}$, we get
\[
E_{\alpha}(v;D_{R}(0))\leq\lim_{n\rightarrow\infty}E_{\alpha}(\hat{u}_{n};D_{R/R_{n}}(\tilde{z}^{(n)})),
\]
while picking some arbitrary $\epsilon>0$, we obtain 
\[
E_{\alpha}(v;D_{R}(0))\leq\lim_{n\rightarrow\infty}E_{\alpha}(\hat{u}_{n};D_{R/R_{n}}(\tilde{z}^{(n)}))\leq\lim_{n\rightarrow\infty}E_{\alpha}(\hat{u}_{n};D_{\epsilon}(0)).
\]
For the $d\alpha-$energy, we proceed analogously: for $R>0$ we have
\[
E_{d\alpha}(v;D_{R}(0))=\lim_{n\rightarrow\infty}\int_{D_{R}(0)}g_{n}^{*}d\alpha=\lim_{n\rightarrow\infty}\int_{D_{R/R_{n}}(\tilde{z}^{(n)})}\hat{f}_{n}^{*}d\alpha,
\]
while picking some arbitrary $\epsilon>0$, we find
\[
E_{d\alpha}(v;D_{R}(0))=\lim_{n\rightarrow\infty}\int_{D_{R/R_{n}}(\tilde{z}^{(n)})}\hat{f}_{n}^{*}d\alpha\leq\lim_{n\rightarrow\infty}\int_{D_{\epsilon}(0)}\hat{f}_{n}^{*}d\alpha\leq\lim_{n\rightarrow\infty}E_{d\alpha}(\hat{u}_{n};D_{\epsilon}(0)).
\]
Because the $\alpha-$ and $d\alpha-$energies are non-negative,
\begin{align*}
E(v;D_{R}(0)) & =E_{\alpha}(v;D_{R}(0))+E_{d\alpha}(v;D_{R}(0))\\
 & \leq\lim_{n\rightarrow\infty}E_{\alpha}(\hat{u}_{n};D_{\epsilon}(0))+\lim_{n\rightarrow\infty}E_{d\alpha}(\hat{u}_{n};D_{\epsilon}(0))\\
 & =\lim_{n\rightarrow\infty}E(\hat{u}_{n};D_{\epsilon}(0))\\
 & \leq E_{0},
\end{align*}
and since $R>0$ was arbitrary, we obtain $E_{d\alpha}(v;\mathbb{C})\leq E_{0}$.
As $v$ is a usual pseudoholomorphic curve, it follows that $E(v;\mathbb{C})=E_{\text{H}}(v;\mathbb{C})$,
where $E_{\text{H}}$ is the Hofer energy defined in \cite{key-1};
thus $E_{\text{H}}(v;\mathbb{C})\leq E_{0}$. Moreover, as $v$ is
non-constant we have by Remark 2.38 of \cite{key-6}, that for any
$\epsilon>0$,\textbf{ }
\[
0<\hbar\leq E_{d\alpha}(v;\mathbb{C})\leq\lim_{n\rightarrow\infty}E_{d\alpha}(\hat{u}_{n};D_{\epsilon}(0))\leq\lim_{n\rightarrow\infty}E_{d\alpha}(u_{n};\psi_{n}(D_{\epsilon}(0))).
\]
Choosing $\epsilon>0$ such that $\psi_{n}(D_{\epsilon}(0))\subset U$
for all $n$, we end up with
\[
0<\hbar\leq\lim_{n\rightarrow\infty}E_{d\alpha}(u_{n};U)\leq E_{0},
\]
and the proof is finished.
\end{singlespace}
\end{proof}
\begin{singlespace}
\noindent The next proposition is proved by contradiction by means
of Lemma \ref{2_1_lem:There-exists-a}.
\end{singlespace}
\begin{prop}
\begin{singlespace}
\noindent \label{2_1_prop:There-exists-a}There exists a subsequence
of $u_{n}$, still denoted by $u_{n}$, and a finite subset $\mathcal{Z}\subset\dot{\tilde{S}}^{\mathcal{D},r}$
such that for every compact subset $\mathcal{K}\subset\dot{\tilde{S}}^{\mathcal{D},r}\backslash\mathcal{Z}$,
there exists a constant \textup{$C_{\mathcal{K}}>0$} such that 
\[
\left\Vert du_{n}(z)\right\Vert :=\sup_{v\in T_{z}S^{\mathcal{D},r},\left\Vert v\right\Vert _{h_{n}}=1}\left\Vert du_{n}(z)v\right\Vert _{\overline{g}}\leq C_{\mathcal{K}}
\]
for all $z\in\mathcal{K}$. 
\end{singlespace}
\end{prop}
\begin{proof}
\begin{singlespace}
\noindent For the sequence $u_{n}$ and any finite subset $\mathcal{Z}\subset\dot{\tilde{S}}^{\mathcal{D},r},$
we define
\begin{align*}
\mathcal{Z}_{\{u_{n}\},\mathcal{Z}} & :=\left\{ z\in\dot{\tilde{S}}^{\mathcal{D},r}\backslash\mathcal{Z}\mid\text{ there exists a subsequence }u_{n_{k}}\text{ of }u_{n}\text{ and a}\right.\\
 & \left.\text{ sequence }z_{k}\in\dot{\tilde{S}}^{\mathcal{D},r}\backslash\mathcal{Z}\text{ such that }z_{k}\rightarrow z\text{ and }\left\Vert du_{n_{k}}(z_{k})\right\Vert \rightarrow\infty\text{ as }k\rightarrow\infty\right\} .
\end{align*}
If $\mathcal{\mathcal{Z}}_{\{u_{n}\},\emptyset}$ is empty then the
assertion is fullfilled for the sequence $u_{n}$ and the finite set
$\mathcal{Z}=\emptyset$. Otherwise, we choose $z^{1}\in\mathcal{\mathcal{Z}}_{\{u_{n}\},\emptyset}$.
In this case, there exists a sequence $z_{n}^{1}\in\dot{\tilde{S}}^{\mathcal{D},r}$
and a subsequence $u_{n}^{1}$ of $u_{n}$ such that $z_{n}^{1}\rightarrow z^{1}$
and $\left\Vert du_{n}^{1}(z_{n}^{1})\right\Vert \rightarrow\infty$.
Consider now the set $\mathcal{\mathcal{Z}}_{\{u_{n}^{1}\},\{z^{1}\}}$.
If $\mathcal{\mathcal{Z}}_{\{u_{n}^{1}\},\{z^{1}\}}$ is empty then
the assertion is fullfilled for the subsequence $u_{n}^{1}$ and the
finite set $\mathcal{Z}=\{z^{1}\}$. Otherwise, we choose an element
$z^{2}\in\mathcal{\mathcal{Z}}_{\{u_{n}^{2}\},\{z^{1}\}}$. In this
case, by definition, there exists a sequence $z_{n}^{2}\in\dot{\tilde{S}}^{\mathcal{D},r}\backslash\{z^{1}\}$
and a subsequence $u_{n}^{2}$ of $u_{n}^{1}$ such that $z_{n}^{2}\rightarrow z^{2}$
and $\left\Vert du_{n}^{2}(z_{n}^{2})\right\Vert \rightarrow\infty$.
Let us show that the set of points $\mathcal{Z}=\{z^{1},z^{2},...\}$
constructed in this way is finite, or more precisely, that $|\mathcal{Z}|\leq2E_{0}/\hbar$.
Assume $|\mathcal{Z}|>2E_{0}/\hbar$ and pick an integer $k>2E_{0}/\hbar$
and pairwise different points $z^{1},...,z^{k}\in\mathcal{Z}$. Let
$U_{1},...,U_{k}\subset\dot{\tilde{S}}^{\mathcal{D},r}$ be some open
pairwise disjoint neighbourhoods of $z^{1},...,z^{k}$. Applying Lemma
\ref{2_1_lem:There-exists-a} inductively, we deduce that there exists
a positive integer $N$ such that for every $n\geq N$, $E_{d\alpha}(u_{n};U_{i})\geq\hbar/2$
for all $i=1,...,k$. Since the $U_{i}$ are disjoint, we obtain
\[
k\frac{\hbar}{2}\leq\sum_{i=1}^{k}E_{d\alpha}(u_{n};U_{i})\leq E_{d\alpha}(u_{n};\dot{\tilde{S}}^{\mathcal{D},r})\leq E_{0}.
\]
Thus $k\leq2E_{0}/\hbar$ which is a contradiction to our assumption.
\end{singlespace}
\end{proof}
\begin{singlespace}
\noindent By means of Proposition \ref{2_1_prop:There-exists-a} we
can prove the convergence of the $\mathcal{H}-$holomorphic maps in
a punctured thick part of the Riemann surface. 
\end{singlespace}
\begin{proof}
\begin{singlespace}
\noindent \emph{(of Theorem \ref{thm:There-exists-finitely})} For
some sufficiently small $k\in\mathbb{N}$ we consider the subsets
$\Omega_{k}:=\text{Thick}_{1/k}(\dot{\tilde{S}}^{\mathcal{D},r},h)\backslash\bigcup_{i=1}^{N}D_{1/k}^{h}(z^{i})$,
where $\mathcal{Z}=\{z^{1},...,z^{N}\}$ is the subset in Proposition
\ref{2_1_prop:There-exists-a} and $D_{1/k}^{h}(z_{i})$ is the open
disk around $z_{i}$ of radius $1/k$ with respect to the metric $h$.
In order to keep the notation simple, the subsequence obtained by
applying Proposition \ref{2_1_prop:There-exists-a} is still denoted
by $u_{n}$. Obviously, $\Omega_{k}$ build an exhaustion by compact
sets of $\dot{\tilde{S}}^{\mathcal{D},r}\backslash\mathcal{Z}$. These
sets are compact surfaces with boundary. By Proposition \ref{2_1_prop:There-exists-a},
the maps $u_{n}$ have uniformly bounded gradients on $\Omega_{1}$.
Thus after a suitable translation of the maps $u_{n}$ in the $\mathbb{R}-$coordinate,
there exists a subsequence $u_{n}^{1}$ of $u_{n}$ that converges
in $C^{\infty}(\Omega_{1})$ to a map $u:\Omega_{1}\rightarrow\mathbb{R}\times M$.
Iteratively, at step $k+1$ there exists a subsequence $u_{n}^{k+1}$
of $u_{n}^{k}$ that converges in $C^{\infty}(\Omega_{k+1})$ to a
map $u:\Omega_{k+1}\rightarrow\mathbb{R}\times M$ which is an extension
from $\Omega_{k}$ to $\Omega_{k+1}$. This procedure allows us to
define a map $u:\dot{\tilde{S}}^{\mathcal{D},r}\backslash\mathcal{Z}\rightarrow\mathbb{R}\times M$.
After passing to some diagonal subsequene $u_{n}^{n}$, the maps $u_{n}^{n}$
converge in $C_{\text{loc}}^{\infty}(\dot{\tilde{S}}^{\mathcal{D},r}\backslash\mathcal{Z})$
to the map $u:\dot{\tilde{S}}^{\mathcal{D},r}\backslash\mathcal{Z}\rightarrow\mathbb{R}\times M$.
Since the $L^{2}-$norms of $\gamma_{n}$ are uniformly bounded on
$S^{\mathcal{D},r}$, they converge in $C_{\text{loc}}^{\infty}(\dot{\tilde{S}}^{\mathcal{D},r})$
to some harmonic $1-$form $\gamma$ with a bounded $L^{2}-$norm
on $\dot{\tilde{S}}^{\mathcal{D},r}$. Hence the map $u$ is a $\mathcal{H}-$holomorphic
curve on $\dot{\tilde{S}}^{\mathcal{D},r}\backslash\mathcal{Z}$ with
harmonic perturbation $\gamma$.
\end{singlespace}
\end{proof}
\begin{singlespace}

\subsection{\label{sec:The-Thin-Part}Convergence on the thin part and around
the points from $\mathcal{Z}$}
\end{singlespace}

\begin{singlespace}
\noindent In this section we investigate the convergence of the $\mathcal{H}-$holomorphic
curves $u_{n}$ on the components of the thin part and in the neighbourhood
of the points from $\mathcal{Z}$ that were constructed in Theorem
\ref{thm:There-exists-finitely}. For a sufficient small $\delta>0$,
the set $\text{Thin}_{\delta}(\dot{S}^{\mathcal{D},r},h_{n})$ can
be decomposed in two types of connected components: (I) the non-compact
components that are called cusps, which are neighbourhoods of punctures
with respect to the hyperbolic metric, and (II) the compact components
called hyperbolic cylinders. Each of these cylinders can be biholomorphically
identified with the standard cylinder $[-R,R]\times S^{1}$ for a
suitable $R>0$. In the Deligne-Mumford limiting process $R$ may
tend to $\infty$ and nodes appear. For more details we refer to Chapter
1 of \cite{key-6}. This section is organized as follows. First, we
analyze the convergence of $u_{n}$ on components that can be indentified
with hyperbolic cylinders, and describe the limit object. Second,
we treat the convergence of $u_{n}$ on components that can be identified
with cusps, and as before, describe the limit object. The convergence
results estabished here can be used to describe the convergence of
$u_{n}$ in a neighbourhood of the points from $\mathcal{Z}$. Third,
we use the description of the convergence of the $\mathcal{H}-$holomorpic
curves $u_{n}$ on the thick part (established in Section \ref{2_1_sec:First_Theorem}),
the thin part, and in the neighbourhood of the points from $\mathcal{Z}$
(established in this section) to define a new surface by glueing the
two parts together. On this surface we decribe the convergence of
$u_{n}$ completely.
\end{singlespace}
\begin{singlespace}

\subsubsection{\label{subsec:Cylinders}Cylinders}
\end{singlespace}

\begin{singlespace}
\noindent We analyze the convergence of $u_{n}$ on compact components
of the thin part which are biholomorphic to hyperbolic cylinders.
When restricted to these cylinders, the curves $u_{n}$ can have a
$d\alpha-$energy larger than the constant $\hbar>0$ defined in (\ref{eq:hbar}).
Since we do not have a version of the monotonicity lemma in the $\mathcal{H}-$holomorphic
case, the classical results on the asymptotic of holomorphic cylinders
from \cite{key-1} and \cite{key-11} are not directly applicable.
To deal with this problem we shift the maps by the Reeb flow to make
them pseudoholomorphic. Actually we proceed as follows. We decompose
the hyperbolic cylinder into a finite uniform number of smaller cylinders;
some of them having conformal modulus tending to infinity but a $d\alpha-$energy
strictly smaller than $\hbar$, and the rest of them having bounded
modulus but a $d\alpha-$energy possibly larger than $\hbar$. We
refer to these cylinders as cylinders of types $\infty$ and $b_{1}$,
respectively. We consider an alternating appearance of these cylinders,
as it can be seen in Figure \ref{fig12}.

\noindent \begin{figure}     
	\centering  
	\def\svgscale{0.8} 	
	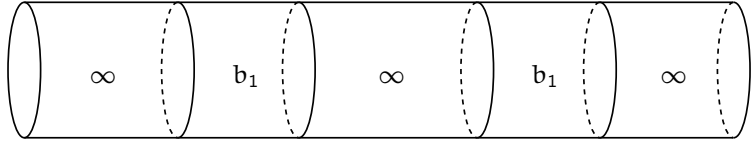  
	\caption{The component of the thin part, which is biholomorphic to a cylinder, is divided in cylinders of types $b_{1}$ and $\infty$ in an alternating order.}
	\label{fig12}
\end{figure}

\noindent The convergence and the description of the limit object
are first treated for cylinders of type $\infty$, and then for cylinders
of type $b_{1}$.

\noindent As cylinders of type $\infty$ have a small $d\alpha-$energy,
we can assume, by the classical bubbling-off analysis, that the maps
$u_{n}$ have uniformly bounded gradients. To make the curves $u_{n}$
pseudoholomorphic, we perform a transformation by pushing them along
the Reeb flow up to some specific time. This procedure is made precise
in \cite{key-23}. As the gradients of these transformed curves still
remain uniformly bounded, we can adapt the results of \cite{key-11}
to formulate a convergence result for the transformed curves (see
\cite{key-23}). Undoing the transformation we obtain a convergence
result for the $\mathcal{H}-$holomorphic curves. 

\noindent In the case of cylinders of type $b_{1}$ we proceed as
follows. Relying on a bubbling-off argument, as we did in the case
of the thick part (see Section \ref{2_1_sec:First_Theorem}), we assume
that the gradients blow up only in a finite uniform number of points
and remain uniformly bounded in a compact complement of them. In this
compact region, the Arzelá-Ascoli theorem shows that the curves $u_{n}$
together with the harmonic perturbations $\gamma_{n}$ converge in
$C^{\infty}$ to some $\mathcal{H}-$holomorphic curve. What is then
left is the convergence in a neighbourhood of the finitely many punctures
where the gradients blow up. Here, a neighbourhood of a puncture is
a disk on which the harmonic perturbation can be made exact and can
be encoded in the $\mathbb{R}-$coordinate of the curve $u_{n}$.
By this procedure we transform the $\mathcal{H}-$holomorphic curve
into a usual pseudoholomorphic curve defined on a disk $D$. By the
$C^{\infty}-$convergence of $u_{n}$ on any compact complement of
the punctures, we assume that the transformed curves converge on an
arbitrary neighbourhood of $\partial D$. This approach, which is
described in detail in Section \ref{subsec:Cylinders-of-type-1},
uses a convergence result established in Appendix \ref{sec:Pseudoholomorphic-disks-with}.
As for cylinders of type $\infty$, we undo the transformation and
derive a convergence result for the $\mathcal{H}-$homolorphic curves
on cylinders of type $b_{1}$. Finally, glueing all cylinders together,
we are led to a convergence result for the entire component which
is biholomorphic to a hyperbolic cylinder from the thin part.

\noindent Let $C_{n}$ be a component of $\text{Thin}_{\delta}(\dot{S}^{\mathcal{D},r},h_{n})$
which is conformal equivalent to the cylinder $[-\sigma_{n}^{\delta},\sigma_{n}^{\delta}]\times S^{1}$
via a map $\rho_{n}:[-\sigma_{n}^{\delta},\sigma_{n}^{\delta}]\times S^{1}\rightarrow C_{n}$.
Observe that from the definition of Deligne-Mumford convergence, $\sigma_{n}^{\delta}\rightarrow\infty$
as $n\rightarrow\infty$. In the following, we drop the fixed, sufficiently
small constant $\delta>0$, and assume that the curves $u_{n}$ are
defined on $[-\sigma_{n},\sigma_{n}]\times S^{1}$. Let $u_{n}=(a_{n},f_{n}):[-\sigma_{n},\sigma_{n}]\times S^{1}\rightarrow\mathbb{R}\times M$
be a sequence of $\mathcal{H}-$holomorphic curves with harmonic perturbations
$\gamma_{n}$, i.e.,
\begin{align*}
\pi_{\alpha}df_{n}\circ i & =J(f_{n})\circ\pi_{\alpha}df_{n},\\
(f_{n}^{*}\alpha)\circ i & =da_{n}+\gamma_{n}
\end{align*}
on $[-\sigma_{n},\sigma_{n}]\times S^{1}$, and let us assume that
the energy of $u_{n}$, as well as the $L^{2}-$norm of $\gamma_{n}$
on the cylinders are uniformly bounded, i.e. for the constants $E_{0},C_{0}>0$
we have $E(u_{n};[-\sigma_{n},\sigma_{n}]\times S^{1})\leq E_{0}$
and $\left\Vert \gamma_{n}\right\Vert _{L^{2}([-\sigma_{n},\sigma_{n}]\times S^{1})}^{2}\leq C_{0}$
for all $n\in\mathbb{N}$.

\noindent Before describing the decomposition of $[-\sigma_{n},\sigma_{n}]\times S^{1}$
into cylinders of types $\infty$ and $b_{1}$ we give a proposition
which states that the $C^{1}-$norm of the harmonic perturbation $\gamma_{n}$
is uniformly bounded. This result will play an essential role in Section
\ref{subsec:Cylinders-of-type-1}. We set $\gamma_{n}=f_{n}ds+g_{n}dt$,
where $f_{n}$ and $g_{n}$ are harmonic functions defined on $[-\sigma_{n},\sigma_{n}]\times S^{1}$
with coordinates $(s,t)$ such that $f_{n}+ig_{n}$ is holomorphic.
By the uniform $L^{2}-$bound of $\gamma_{n}$, we have
\[
\left\Vert \gamma_{n}\right\Vert _{L^{2}([-\sigma_{n},\sigma_{n}]\times S^{1})}^{2}=\int_{[-\sigma_{n},\sigma_{n}]\times S^{1}}\left(f_{n}^{2}+g_{n}^{2}\right)dsdt\leq C_{0}
\]
for all $n\in\mathbb{N}$. As a result, the $L^{2}-$norm of the holomorphic
function $f_{n}+ig_{n}$ is uniformly bounded. Denote this function
by $G_{n}=f_{n}+ig_{n}$.
\end{singlespace}
\begin{prop}
\begin{singlespace}
\noindent \label{prop:There-exists-a-1}For any $\delta>0$ there
exists a constant $C_{\delta}>0$ such that
\[
\left\Vert G_{n}\right\Vert _{C^{1}([-\sigma_{n}+\delta,\sigma_{n}-\delta]\times S^{1})}\leq C_{\delta}
\]
for all $n\in\mathbb{N}$. 
\end{singlespace}
\end{prop}
\begin{proof}
\begin{singlespace}
\noindent First, we prove that the sequence $G_{n}$ is uniformly
bounded in $C^{0}-$norm. As $G_{n}:[-\sigma_{n},\sigma_{n}]\times S^{1}\rightarrow\mathbb{C}$
is holomorphic, $f_{n}=\Re(G_{n})$ and $g_{n}=\Im(G_{n})$ are harmonic
functions defined on $[-\sigma_{n},\sigma_{n}]\times S^{1}$. For
a sufficiently small $\delta>0$ we establish $C^{0}-$bounds for
$f_{n}$ on the subcylinders $[-\sigma_{n}+(\delta/2),\sigma_{n}-(\delta/2)]\times S^{1}$.
By the mean value theorem for harmonic function, we have
\begin{align*}
f_{n}(p) & =\frac{16}{\pi\delta^{2}}\int_{D_{\frac{\delta}{4}}(p)}f_{n}(s,t)dsdt
\end{align*}
for all $p\in[-\sigma_{n}+(\delta/2),\sigma_{n}-(\delta/2)]\times S^{1}$,
where $D_{\delta/4}(p)\subset[-\sigma_{n},\sigma_{n}]\times S^{1}$.
Then Hölder's inequality yields
\[
|f_{n}(p)|\leq\frac{4}{\sqrt{\pi}\delta}\left(\int_{D_{\frac{\delta}{4}}(p)}|f_{n}(s,t)|^{2}dsdt\right)^{\frac{1}{2}}\leq\frac{4}{\sqrt{\pi}\delta}\sqrt{C_{0}}
\]
for all $n\in\mathbb{N}$. As a results, we obtain
\[
\left\Vert f_{n}\right\Vert _{C^{0}\left(\left[-\sigma_{n}+\frac{\delta}{2},\sigma_{n}-\frac{\delta}{2}\right]\times S^{1}\right)}\leq\frac{4}{\sqrt{\pi}\delta}\sqrt{C_{0}},
\]
and note that the same result holds for $g_{n}$. 

\noindent By means of bubbling-off analysis we prove now that the
gradient of $G_{n}$ is uniformly bounded. Assume
\[
\sup_{p\in[-\sigma_{n}+\delta,\sigma_{n}-\delta]\times S^{1}}\left|\nabla G_{n}(p)\right|\rightarrow\infty
\]
as $n\rightarrow\infty$. Let $p_{n}\in[-\sigma_{n}+\delta,\sigma_{n}-\delta]\times S^{1}$
be such that
\[
|\nabla G_{n}(p_{n})|=\sup_{p\in[-\sigma_{n}+\delta,\sigma_{n}-\delta]\times S^{1}}\left|\nabla G_{n}(p)\right|;
\]
then $R_{n}:=|\nabla G_{n}(p_{n})|\rightarrow\infty$ as $n\rightarrow\infty$.
Set $\epsilon_{n}:=R_{n}^{-\frac{1}{2}}\searrow0$ as $n\rightarrow\infty$,
and observe that $\epsilon_{n}R_{n}\rightarrow\infty$ as $n\rightarrow\infty$.
Choose $n_{0}\in\mathbb{N}_{0}$ sufficiently large such that $D_{10\epsilon_{n}}(p_{n})\subset[-\sigma_{n},\sigma_{n}]\times S^{1}$
for all $n\geq n_{0}$. By Hofer's topologial lemma there exist $\epsilon'_{n}\in(0,\epsilon_{n}]$
and $p'_{n}\in[-\sigma_{n},\sigma_{n}]\times S^{1}$ satisfying:

\end{singlespace}\begin{enumerate}
\begin{singlespace}
\item $\epsilon'_{n}R'_{n}\geq\epsilon_{n}R_{n}$; 
\item $p'_{n}\in D_{2\epsilon_{n}}(p_{n})\subset D_{10\epsilon_{n}}(p_{n})$; 
\item $|\nabla G_{n}(p)|\leq2R'_{n}$, for all $p\in D_{\epsilon'_{n}}(p'_{n})\subset D_{10\epsilon_{n}}(p_{n})$, 
\end{singlespace}
\end{enumerate}
\begin{singlespace}
\noindent where $R'_{n}:=|du_{n}(p'_{n})|$. Via rescaling consider
the maps $\tilde{G}_{n}:D_{\epsilon'_{n}R'_{n}}(0)\rightarrow\mathbb{C}$,
defined by 
\[
\tilde{G}_{n}(w):=G_{n}\left(p'_{n}+\frac{w}{R'_{n}}\right)
\]
for $w\in D_{\epsilon'_{n}R'_{n}}(0)$. Observe that $p'_{n}+(w/R_{n}')\in D_{\epsilon'_{n}}(p'_{n})$
for $w\in D_{\epsilon'_{n}R'_{n}}(0)$, and that for $\tilde{G}_{n}$
we have:

\end{singlespace}\begin{enumerate}
\begin{singlespace}
\item $|\nabla\tilde{G}_{n}(0)|=1$; 
\item $|\nabla\tilde{G}_{n}(w)|\leq2$ for $w\in D_{\epsilon'_{n}R'_{n}}(0)$; 
\item $\tilde{G}_{n}$ is holomorphic on $D_{\epsilon_{n}'R_{n}'}(0)$; 
\item $\tilde{G}_{n}$ is uniformly bounded on $[-\sigma_{n}+\delta,\sigma_{n}-\delta]\times S^{1}$
(by Assertion 1).
\end{singlespace}
\end{enumerate}
\begin{singlespace}
\noindent By the usual regularity theory for pseudoholomorphic maps
and Arzelá-Ascoli theorem, $\tilde{G}_{n}$ converge in $C_{\text{loc}}^{\infty}(\mathbb{C})$
to a bounded holomorphic map $\tilde{G}:\mathbb{C}\rightarrow\mathbb{C}$
with $|\nabla\tilde{G}(0)|=1$. By Liouville theorem this map can
be only the constant map, and so, we arrive at a contradiction with
$|\nabla\tilde{G}(0)|=1$.
\end{singlespace}
\end{proof}
\begin{singlespace}
\noindent Thus for $\delta>0$ we can replace the cylinder $[-\sigma_{n}+\delta,\sigma_{n}-\delta]\times S^{1}$
by $[-\sigma_{n},\sigma_{n}]\times S^{1}$ if we consider $\text{Thin}_{\delta}(\dot{S}^{\mathcal{D},r},h_{n})$
for a smaller $\delta>0$. We come now to the decomposition of $[-\sigma_{n},\sigma_{n}]\times S^{1}$
into cylinders of types $\infty$ and $b_{1}$. Consider the parameter-dependent
function with parameter $h\in[-\sigma_{n},\sigma_{n}]$ defined by
\[
F_{n,h}:[h,\sigma_{n}]\rightarrow\mathbb{R},\ s\mapsto\int_{[h,s]\times S^{1}}f_{n}^{*}d\alpha.
\]
As $f_{n}^{*}d\alpha$ is non-negative, $F_{n,h}$ is positive and
monotone. For the constant $\hbar$ defined in (\ref{eq:hbar}), we
set $h_{n}^{(0)}=-\sigma_{n}$, and define
\[
h_{n}^{(m)}:=\sup\left(F_{n,h_{n}^{(m-1)}}^{-1}\left[0,\frac{\hbar}{4}\right]\right).
\]
Since $E_{d\alpha}(u_{n};[-\sigma_{n},\sigma_{n}]\times S^{1})<E_{0}$,
the sequence $\{h_{n}^{(m)}\}_{m\in\mathbb{N}_{0}}$ has to end after
$N_{n}$ steps, where $h_{n}^{(N_{n})}=\sigma_{n}$. On the cylinder
$[h_{n}^{(N_{n}-1)},h_{n}^{(N_{n})}]\times S^{1}$, the $d\alpha-$energy
of $u_{n}$ can be smaller than $\hbar/4$. Obviously, we have $-\sigma_{n}=h_{n}^{(0)}<h_{n}^{(1)}<...<h_{n}^{(m)}<...<h_{n}^{(N_{n})}=\sigma_{n}$
giving $E_{d\alpha}(u_{n};[h_{n}^{(m-1)},h_{n}^{(m)}]\times S^{1})=\hbar/4$
for $m=1,...,N_{n}-1$ and $E_{d\alpha}(u_{n};[h_{n}^{(N_{n}-1)},h_{n}^{(N_{n})}]\times S^{1})\leq\hbar/4$.
Hence the $d\alpha-$energy can be written as
\[
E_{d\alpha}(u_{n};[-\sigma_{n},\sigma_{n}]\times S^{1})=(N_{n}-1)\frac{\hbar}{4}+E_{d\alpha}(u_{n};[h_{n}^{(N_{n}-1)},h_{n}^{(N_{n})}]\times S^{1}),
\]
 which implies the following bound on $N_{n}$: 
\[
0\leq N_{n}\leq\frac{4E_{0}}{\hbar}+1.
\]
After going over to a subsequence, we can further assume that $N_{n}$
is also independent of $n$; for this reason, we set $N_{n}=N$. Thus
the cylinders $[-\sigma_{n},\sigma_{n}]\times S^{1}$ have been decomposed
into $N$ smaller subcylinders $[h_{n}^{(0)},h_{n}^{(1)}]\times S^{1},...,[h_{n}^{(N-1)},h_{n}^{(N)}]\times S^{1}$
on which we have $E_{d\alpha}(u_{n};[h_{n}^{(m-1)},h_{n}^{(m)}]\times S^{1})=\hbar/4$
for $m\in\{1,...,N-1\}$ and $E_{d\alpha}(u_{n};[h_{n}^{(N-1)},h_{n}^{(N)}]\times S^{1})\leq\hbar/4$. 

\noindent A sequence of cylinders $[a_{n},b_{n}]\times S^{1}$, where
$a_{n},b_{n}\in\mathbb{R}$ and $a_{n}<b_{n}$ is called of type $b_{1}$
if $b_{n}-a_{n}$ is bounded from above, and of type $\infty$ if
$b_{n}-a_{n}\rightarrow\infty$ as $n\rightarrow\infty$. This is
illustrated in Figure \ref{fig13}.

\noindent \begin{figure}     
	\centering  
	\def\svgscale{0.65} 	
	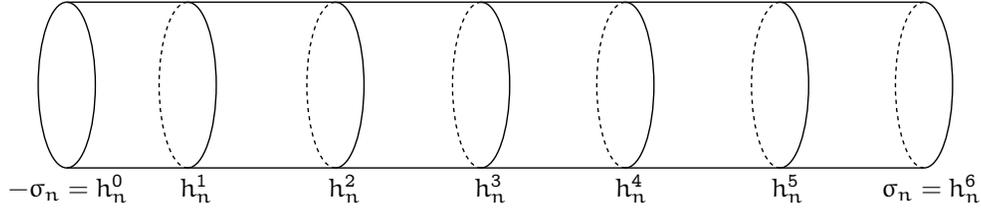  
	\caption{Decomposition of $[-\sigma_{n},\sigma_{n}] \times S^{1}$ into smaller cylinders $[h_{n}^{(m)},h_{n}^{(m+1)}]\times S^{1}$ having $d\alpha-$energy $\hbar / 4$ or less.}
	\label{fig13}
\end{figure}
\end{singlespace}
\begin{lem}
\begin{singlespace}
\noindent \label{lem:Let-=000024=00005Bh_=00007Bn=00007D^=00007B(m-1)=00007D,h_=00007Bn=00007D^=00007B(m)=00007D=00005D}Let
$[h_{n}^{(m-1)},h_{n}^{(m)}]\times S^{1}$ be a cylinder of type $\infty$
and let $h>0$ be chosen small enough such that $h_{n}^{(m)}-h_{n}^{(m-1)}-2h=(h_{n}^{(m)}-h)-(h_{n}^{(m-1)}+h)>0$
for all $n\in\mathbb{N}$. Then there exists a constant $C_{h}>0$
such that 
\[
\left\Vert du_{n}(z)\right\Vert _{C^{0}}=\sup_{\left\Vert v\right\Vert _{\text{eucl}}=1}\left\Vert du_{n}(z)v\right\Vert <C_{h}
\]
for all $z\in[h_{n}^{(m-1)}+h,h_{n}^{(m)}-h]\times S^{1}$ and $n\in\mathbb{N}$.
\end{singlespace}
\end{lem}
\begin{proof}
\begin{singlespace}
\noindent The proof makes use of bubbling-off analysis. Assume that
there exists $h>0$ such that $h_{n}^{(m)}-h_{n}^{(m-1)}-2h>0$ and
\begin{equation}
\sup_{z\in[h_{n}^{(m-1)}+h,h_{n}^{(m)}-h]\times S^{1}}\left\Vert du_{n}(z)\right\Vert _{C^{0}}=\infty.\label{eq:BoundGrad}
\end{equation}
Then there exists a sequence $z_{n}\in(h_{n}^{(m-1)}+h,h_{n}^{(m)}-h)\times S^{1}$
with the property $R_{n}:=\left\Vert du_{n}(z_{n})\right\Vert _{C^{0}}\rightarrow\infty$
as $n\rightarrow\infty$. Let $\epsilon_{n}=R_{n}^{-\frac{1}{2}}\searrow0$
as $n\rightarrow\infty$, and observe that $\epsilon_{n}R_{n}\rightarrow\infty$
as $n\rightarrow\infty$. Choose $n_{0}\in\mathbb{N}$ sufficiently
large such that $D_{10\epsilon_{n}}(z_{n})\subset[h_{n}^{(m-1)},h_{n}^{(m)}]\times S^{1}$
for all $n\geq n_{0}$. By Hofer's topological lemma, there exist
$\epsilon'_{n}\in(0,\epsilon_{n}]$ and $z'_{n}\in[h_{n}^{(m-1)},h_{n}^{(m)}]\times S^{1}$
satisfying: 

\end{singlespace}\begin{enumerate}
\begin{singlespace}
\item $\epsilon'_{n}R'_{n}\geq\epsilon_{n}R_{n}$; 
\item $z'_{n}\in D_{2\epsilon_{n}}(z_{n})\subset D_{10\epsilon_{n}}(z_{n})$; 
\item $\left\Vert du_{n}(z)\right\Vert _{C^{0}}\leq2R'_{n}$, for all $z\in D_{\epsilon'_{n}}(z'_{n})\subset D_{10\epsilon_{n}}(z_{n})$,
\end{singlespace}
\end{enumerate}
\begin{singlespace}
\noindent where $R'_{n}:=\left\Vert du_{n}(z'_{n})\right\Vert _{C^{0}}$.
Applying rescaling consider the map $v_{n}:D_{\epsilon'_{n}R'_{n}}(0)\rightarrow\mathbb{R}\times M$,
defined by 
\[
v_{n}(w)=(b_{n}(w),g_{n}(w)):=u_{n}\left(z'_{n}+\frac{w}{R'_{n}}\right)-a_{n}(z_{n}')
\]
for $w\in D_{\epsilon'_{n}R'_{n}}(0)$. Note that $z'_{n}+(w/R_{n}')\in D_{\epsilon'_{n}}(z'_{n})$
for $w\in D_{\epsilon'_{n}R'_{n}}(0)$, and that for $v_{n}$ we have

\end{singlespace}\begin{enumerate}
\begin{singlespace}
\item $\left\Vert dv_{n}(0)\right\Vert _{C^{0}}=1$; 
\item $\left\Vert dv_{n}(w)\right\Vert _{C^{0}}\leq2$ for $w\in D_{\epsilon'_{n}R'_{n}}(0)$; 
\item $E_{d\alpha}(v_{n};D_{\epsilon'_{n}R'_{n}}(0))\leq\hbar/4$ (straightforward
calculation shows that the $\alpha-$energy is also uniformly bounded); 
\item $v_{n}$ solves 
\begin{align*}
\pi_{\alpha}dg_{n}\circ i & =J\circ\pi_{\alpha}dg_{n},\\
(g_{n}^{*}\alpha)\circ i & =db_{n}+\frac{\gamma_{n}}{R'_{n}}
\end{align*}
on $D_{\epsilon'_{n}R'_{n}}(0)$. 
\end{singlespace}
\end{enumerate}
\begin{singlespace}
\noindent As the gradients of $v_{n}$ are uniformly bounded, $v_{n}$
converge in $C_{\text{loc}}^{\infty}(\mathbb{C})$ to a finite energy
plane $v=(b,g):\mathbb{C}\rightarrow\mathbb{R}\times M$ characterized
by:

\end{singlespace}\begin{enumerate}
\begin{singlespace}
\item $\left\Vert dv(0)\right\Vert _{C^{0}}=1$;
\item $\left\Vert dv(w)\right\Vert _{C^{0}}\leq2$ for $w\in\mathbb{C}$; 
\item $E_{d\alpha}(v;\mathbb{C})\leq\hbar/4$;
\item $v$ is a finite energy holomorphic plane. 
\end{singlespace}
\end{enumerate}
\begin{singlespace}
\noindent Assertion 3 follows from the fact that for an arbitray $R>0$
we have
\[
E_{d\alpha}(v,D_{R}(0))=\lim_{n\rightarrow\infty}E_{d\alpha}(v_{n};D_{R}(0))\leq\lim_{n\rightarrow\infty}E_{d\alpha}(v_{n};D_{\epsilon'_{n}R'_{n}}(0))\leq\frac{\hbar}{4},
\]
while Assertion 4 follows from the fact that $\gamma_{n}$ has a uniformly
bounded $L^{2}-$norm. Note that by employing the above argument,
a bound for the $\alpha-$energy can be also obtained. Now, as $v$
is non-constant, Theorem 31 of \cite{key-2} gives $E_{d\alpha}(v;\mathbb{C})\geq\hbar$,
which is a contradiction to Assertion 3. Thus Assumption (\ref{eq:BoundGrad})
does not hold, and the gradient of $u_{n}$ on cylinders of type $\infty$
is uniformly bounded.
\end{singlespace}
\end{proof}
\begin{singlespace}
\noindent Now we change the above decomposition so that the lengths
of the cylinders of type $b_{1}$ are also bounded by below and describe
the alternating appearance of cylinders of types $\infty$ and $b_{1}$.
This process is necessary, because on the cylinders of type $b_{1}$
whose length tends to zero we cannot analyze the convergence behaviour
of the maps $u_{n}$ and cannot describe their limit object. We proceed
as follows.
\end{singlespace}
\begin{description}
\begin{singlespace}
\item [{\emph{Step$\ $1.}}] \noindent We consider a cylinder $[h_{n}^{(m)},h_{n}^{(m+1)}]\times S^{1}$
of type $\infty$, on which we apply Lemma \ref{lem:Let-=000024=00005Bh_=00007Bn=00007D^=00007B(m-1)=00007D,h_=00007Bn=00007D^=00007B(m)=00007D=00005D}.
When doing this we choose a sufficiently small constant $h>0$, so
that the gradients are uniformly bounded only on $[h_{n}^{(m)}+h,h_{n}^{(m+1)}-h]\times S^{1}$
by the constant $C_{h}>0$, which in turn, is again a cylinder of
type $\infty$. By this procedure, a cylinder $[h_{n}^{(m)},h_{n}^{(m+1)}]\times S^{1}$
of type $\infty$ is decomposed into three smaller cylinders: two
cylinders $[h_{n}^{(m)},h_{n}^{(m)}+h]\times S^{1}$, $[h_{n}^{(m+1)}-h,h_{n}^{(m+1)}]\times S^{1}$
of type $b_{1}$ and one cylinder $[h_{n}^{(m)}+h,h_{n}^{(m+1)}-h]\times S^{1}$
of type $\infty$. The length of these two cylinders of type $b_{1}$
is $h>0$. To any other cylinder of type $\infty$ we apply the same
procedure with a fixed constant $h>0$.
\item [{\emph{Step$\ $2.}}] \noindent We combine all cylinders of type
$b_{1}$, which are next to each other, to form a bigger cylinder
of type $b_{1}$. This can be seen in Figure \ref{fig15}. By this
procedure, we guarantee that in a constellation consisting of three
cylinders that lie next to each other, the type of the middle cylinder
is different to the types of the left and right cylinders. Thus we
got rid of the cylinders of type $b_{1}$ with length tending to zero,
and make sure that the cylinders of types $\infty$ and $b_{1}$ appear
alternately. We aditionally assume that the first and last cylinders
in the decomposition are of type $\infty$, since otherwise, we can
glue the cylinder of type $b_{1}$ to the thick part of the surface
and consider $\text{Thin}_{\delta}(\dot{S}^{\mathcal{D},r},h_{n})$
for a smaller $\delta>0$. By this procedure, we decompose $[-\sigma_{n},\sigma_{n}]\times S^{1}$
into cylinders of types $\infty$ and $b_{1}$, while the first and
last cylinders in the decomposition are of type $\infty$. \begin{figure}     
	\centering  
	\def\svgscale{0.65} 	
	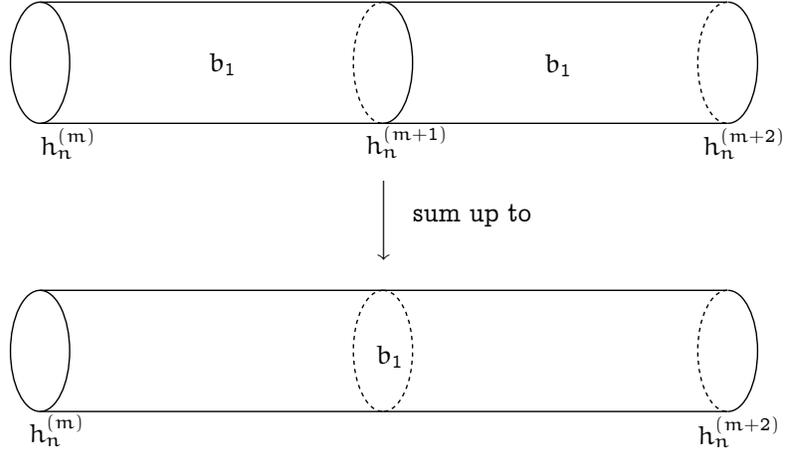  
	\caption{Two cylinders of type $b_{1}$ are combined to form a bigger cylinder of type $b_{1}$.}
	\label{fig15}
\end{figure}
\item [{\emph{Step$\ $3.}}] \noindent For $\tilde{E}_{0}=2(E_{0}+C_{h})$
and in view of the non-degeneracy of the contact manifold $(M,\alpha)$,
let the constant $\hbar_{0}$ be given by
\begin{equation}
\hbar_{0}:=\min\{|T_{1}-T_{2}|\mid T_{1},T_{2}\in\mathcal{\mathcal{P}}_{\alpha},T_{1}\not=T_{2},T_{1},T_{2}\leq\tilde{E}_{0}\}.\label{eq:hbar-0}
\end{equation}
Observe that because of $\tilde{E}_{0}\geq E_{0}$, $\hbar_{0}\leq\hbar$.
If $[h_{n}^{(m-1)},h_{n}^{(m)}]\times S^{1}$ is a cylinder of type
$\infty$ for some $m\in\{1,...,N\}$, we define the constant $\hbar_{0}$
as above and apply Step 1 and Step 2 to decompose this cylinder into
cylinders of types $\infty$ and $b_{1}$, while the first and last
cylinders in the decomposition are of type $\infty$. The cylinders
of type $\infty$ have now a $d\alpha-$energy smaller than $\hbar_{0}/4$.
We apply this procedure to all cylinders of type $\infty$. In summary,
$[-\sigma_{n},\sigma_{n}]\times S^{1}$ is decomposed into cylinders
of type $\infty$ with a $d\alpha-$energy smaller than $\hbar_{0}/4$
and cylinders of type $b_{1}$, with the first and last cylinders
being of type $\infty$.
\item [{\emph{Step$\ $4.}}] \noindent We enlarge the cylinders of type
$b_{1}$ without changing their type. Let $h>0$ be as in Lemma \ref{lem:Let-=000024=00005Bh_=00007Bn=00007D^=00007B(m-1)=00007D,h_=00007Bn=00007D^=00007B(m)=00007D=00005D}
and pick $m\in\{1,...,N\}$ such that $[h_{n}^{(m-1)},h_{n}^{(m)}]\times S^{1}$
is of type $b_{1}$. For $n$ sufficiently large, we replace the cylinder
$[h_{n}^{(m-1)},h_{n}^{(m)}]\times S^{1}$ by the bigger cylinder
$[h_{n}^{(m-1)}-3h,h_{n}^{(m)}+3h]\times S^{1}$, and apply this procedure
to all cylinders of type $b_{1}$. As a result, neighbouring cylinders
will overlap. Essentially, this means that if $[h_{n}^{(m-2)},h_{n}^{(m-1)}]\times S^{1}$
is a cylinder of type $\infty$, which lies to the left of a cylinder
$[h_{n}^{(m-1)}-3h,h_{n}^{(m)}+3h]\times S^{1}$ of type $b_{1}$,
then their intersection is $[h_{n}^{(m-1)}-3h,h_{n}^{(m-1)}]\times S^{1}$.
This can be seen in Figure \ref{fig16}.
\end{singlespace}
\end{description}
\begin{singlespace}
\noindent \begin{figure}     
	\centering  
	\def\svgscale{0.80} 	
	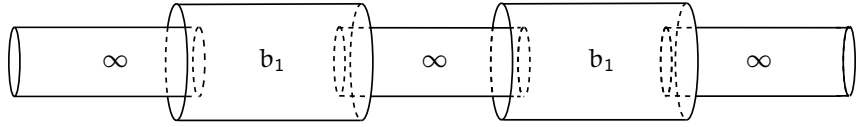  
	\caption{Decomposition of $[-\sigma_{n}, \sigma_{n}]\times S^{1}$ into cylinders of types $\infty$ and $b_{1}$ in an alternating order.}
	\label{fig16}
\end{figure}By the above procedure, the cylinder $[-\sigma_{n},\sigma_{n}]\times S^{1}$
is decomposed into an alternating constellation of cylinders of types
$\infty$ and $b_{1}$. On cylinders of type $\infty$, the $d\alpha-$energy
is smaller than $\hbar_{0}/4$, while on cylinders of type $b_{1}$,
the $d\alpha-$energy can be larger than $\hbar_{0}/4$. By Lemma
\ref{lem:Let-=000024=00005Bh_=00007Bn=00007D^=00007B(m-1)=00007D,h_=00007Bn=00007D^=00007B(m)=00007D=00005D},
the gradients of the $\mathcal{H}-$holomorphic curves on the cylinders
of type $\infty$ are uniformly bounded by the constant $C_{h}>0$
with respect to the Euclidean metric on the domain, and to the metric
described in (\ref{eq:metric-symplectization}) on the target space
$\mathbb{R}\times M$. Finally, the cylinders of types $\infty$ and
$b_{1}$ overlap. 

\noindent We are now well prepared to analyse the convergence of the
$\mathcal{H}-$holomorphic curves on cylinders of types $\infty$
and $b_{1}$. After obtaining separate convergence results, we glue
the limit objects of these cylinders on the overlaps, and obtain a
limit object on the whole cylinder $[-\sigma_{n},\sigma_{n}]\times S^{1}$.
Sections \ref{subsec:Cylinders-of-type} and \ref{subsec:Cylinders-of-type-1}
deal with the convergence and the description of the limit object
on cylinders of types $\infty$ and $b_{1}$, while in Section \ref{subsec:Glueing-cylinders-of}
we carry out the glueing of these two convergence results.
\end{singlespace}
\begin{singlespace}

\subsubsection{\label{subsec:Cylinders-of-type}Cylinders of type $\infty$}
\end{singlespace}

\begin{singlespace}
\noindent We describe the convergence and the limit object of the
sequence of $\mathcal{H}-$holomorphic curves $u_{n}$, defined on
cylinders of type $\infty$. Let $m\in\{1,...,N\}$ be such that $[h_{n}^{(m-1)},h_{n}^{(m)}]\times S^{1}$
is a cylinder of type $\infty$ as described in Section \ref{subsec:Cylinders},
i.e. $h_{n}^{(m)}-h_{n}^{(m-1)}\rightarrow\infty$ as $n\rightarrow\infty$.
Consider the diffeomorphism $\psi_{n}:[-R_{n}^{(m)},R_{n}^{(m)}]\rightarrow[h_{n}^{(m)},h_{n}^{(m+1)}]$
given by $\psi_{n}(s)=s+(h_{n}^{(m)}+h_{n}^{(m+1)})/2$ and the $\mathcal{H}-$holomorphic
maps $u_{n}=(a_{n},f_{n}):[-R_{n}^{(m)},R_{n}^{(m)}]\times S^{1}\rightarrow\mathbb{R}\times M$
with harmonic perturbation $\gamma_{n}$ . For deriving a $C_{\text{loc}}^{\infty}-$convergence
result we consider the following setting:
\end{singlespace}
\begin{description}
\begin{singlespace}
\item [{C1}] \noindent $R_{n}^{(m)}\rightarrow\infty$ as $n\rightarrow\infty$.
\item [{C2}] \noindent $\gamma_{n}$ is a harmonic $1-$form on $[-R_{n}^{(m)},R_{n}^{(m)}]\times S^{1}$
with respect to the standard complex structure $i$, i.e. $d\gamma_{n}=d\gamma_{n}\circ i=0$.
\item [{C3}] \noindent The $d\alpha-$energy of $u_{n}$ is uniformly small,
i.e. $E_{d\alpha}(u_{n};[-R_{n}^{(m)},R_{n}^{(m)}]\times S^{1})\leq\hbar_{0}/2$
for all $n$, where $\hbar_{0}$ is the constant defined in (\ref{eq:hbar-0}).
\item [{C4}] \noindent The energy of $u_{n}$ is uniformly bounded, i.e.
for the constant $E_{0}>0$ we have $E(u_{n};[-R_{n}^{(m)},R_{n}^{(m)}]\times S^{1})\leq E_{0}$
for all $n\in\mathbb{N}$.
\item [{C5}] \noindent The map $u_{n}$ together with the $1-$form $\gamma_{n}$
solve the $\mathcal{H}-$holomorphic curve equation
\begin{align*}
\pi_{\alpha}df_{n}\circ i & =J(f_{n})\circ\pi_{\alpha}df_{n},\\
(f_{n}^{*}\alpha)\circ i & =da_{n}+\gamma_{n}.
\end{align*}
\item [{C6}] \noindent The harmonic $1-$form $\gamma_{n}$ has a uniformly
bounded $L^{2}-$norm, i.e. for the constant $C_{0}>0$ we have $\left\Vert \gamma_{n}\right\Vert _{L^{2}([-R_{n}^{(m)},R_{n}^{(m)}]\times S^{1})}^{2}\leq C_{0}$
for all $n$. 
\item [{C7}] \noindent The map $u_{n}$ has a uniformly bounded gradient
due to Lemma \ref{lem:Let-=000024=00005Bh_=00007Bn=00007D^=00007B(m-1)=00007D,h_=00007Bn=00007D^=00007B(m)=00007D=00005D}
and Step 4 of Section \ref{subsec:Cylinders}, i.e. for the constant
$C_{h}>0$ we have 
\[
\left\Vert du_{n}(z)\right\Vert _{C^{0}}=\sup_{\left\Vert v\right\Vert _{\text{eucl}}=1}\left\Vert du_{n}(z)v\right\Vert <C_{h}
\]
for all $z\in[-R_{n}^{(m)},R_{n}^{(m)}]\times S^{1}$ and all $n\in\mathbb{N}$.
\item [{C8}] \noindent If $P_{n}:=P_{\gamma_{n}}(\{0\}\times S^{1})$ is
the period of $\gamma_{n}$ over the closed curve $\{0\}\times S^{1}$,
as defind in (\ref{eq:periods}), we assume that the sequence $R_{n}P_{n}$
is bounded by the constant $C>0$. Moreover, after going over to some
subsequence, we assume that $R_{n}P_{n}$ converges to some real number
$\tau$.
\item [{C9}] \noindent If $S_{n}:=S_{\gamma_{n}}(\{0\}\times S^{1})$ is
the co-period of $\gamma_{n}$ over the curve $\{0\}\times S^{1}$
as defined in (\ref{eq:co-period}), we assume that $S_{n}R_{n}\rightarrow\sigma$
as $n\rightarrow\infty$.
\end{singlespace}
\end{description}
\begin{rem}
\begin{singlespace}
\noindent \label{rem:The-special-circles}The special circles $\Gamma_{i}^{\text{nod}}$
in Remark \ref{rem:For-a-sequence} are of two types: contractible
and non-contractible. In the contractible case, $\Gamma_{i}^{\text{nod}}$
lies in the isotopy class of $(\rho_{n}\circ\psi_{n})(\{0\}\times S^{1})$
and the conformal periods and co-periods of the harmonic $1-$forms
$\gamma_{n}$ vanish. Hence, conditions C1-C9 are satisfied on the
sequence of degenerating cylinders $[-R_{n}^{(m)},R_{n}^{(m)}]\times S^{1}$.
In the non-contractible case, $\Gamma_{i}^{\text{nod}}$ also lies
in the isotopy class of $(\rho_{n}\circ\psi_{n})(\{0\}\times S^{1})$,
and by the assumptions of Theorem \ref{thm:Let--be-2-1}, conditions
C1-C9 are satisfied.
\end{singlespace}
\end{rem}
\begin{singlespace}
\noindent To simplify notation we drop the index $m$. By Theorem
11 from \cite{key-23} we consider two cases. In Case 1, there exists
a subsequence of $u_{n}$ with vanishing center action, and we use
Theorem 2 from \cite{key-23} to describe the convergence of the $\mathcal{H}-$holomorphic
curves with harmonic perturbations $\gamma_{n}$. In Case 2, each
subsequence of $u_{n}$ has a center action larger than $\hbar_{0}$,
and we use Theorem 4 from \cite{key-23} to describe the convergence. 
\end{singlespace}
\begin{rem}
\begin{singlespace}
\noindent \label{rem:For-every-sequence}For every sequence $h_{n}\in\mathbb{R}_{+}$
with $h_{n}<R_{n}$ and $h_{n},R_{n}/h_{n}\rightarrow\infty$ as $n\rightarrow\infty$,
consider a sequence of diffeomorphisms $\theta_{n}:[-R_{n},R_{n}]\rightarrow[-1,1]$
having the following properties:
\end{singlespace}
\begin{enumerate}
\begin{singlespace}
\item The left and right shifts $\theta_{n}^{\pm}(s):=\theta_{n}(s\pm R_{n})$
defined on $[0,h_{n}]\rightarrow[-1,-1/2]$ and $[-h_{n},0]\rightarrow[1/2,1]$,
respectively, converge in $C_{\text{loc}}^{\infty}$ to the diffeomorphisms
$\theta^{-}:[0,\infty)\rightarrow[-1,-1/2)$ and $\theta^{+}:(-\infty,0]\rightarrow(1/2,1]$,
respectively.
\item On $[-R_{n}+h_{n},R_{n}-h_{n}]$ we define the diffeomorphism $\theta_{n}$
to be linear by requiring
\[
\theta_{n}:\text{Op}([-R_{n}+h_{n},R_{n}-h_{n}])\rightarrow\text{Op}\left(\left[-\frac{1}{2},\frac{1}{2}\right]\right),\ s\mapsto\frac{s}{2(R_{n}-h_{n})},
\]
where $\text{Op}([-R_{n}+h_{n},R_{n}-h_{n}])$ and $\text{Op}([-1/2,1/2])$
are sufficiently small neighbourhoods of the intervals $[-R_{n}+h_{n},R_{n}-h_{n}]$
and $[-1/2,1/2]$, respectively.
\end{singlespace}
\end{enumerate}
\end{rem}
\begin{singlespace}
\noindent Note that the diffeomorphism $\theta_{n}$ give rise to
a diffeomorphism between the cylinders $[-R_{n},R_{n}]\times S^{1}$
and $[-1,1]\times S^{1}$, according to $[-R_{n},R_{n}]\times S^{1}\rightarrow[-1,1]\times S^{1}$,
$(s,t)\mapsto(\theta_{n}(s),t)$. By abuse of notation these diffeomorphisms
will be still denoted by $\theta_{n}$.

\noindent Denote by $u_{n}^{\pm}(s,t):=u_{n}(s\pm R_{n},t)$ the left
and right shifts of the maps $u_{n}$, and by $\gamma_{n}^{\pm}:=\gamma_{n}(s\pm R_{n},t)$
the left and right shifts of the harmonic perturbation, which are
defined on $[0,h_{n}]\times S^{1}$ and $[-h_{n},0]\times S^{1}$,
respectively. In both cases we use the diffeomorphisms $\theta_{n}$
to pull the structures back to the cylinder $[-1,1]\times S^{1}$.
Let $i_{n}:=d\theta_{n}\circ i\circ d\theta_{n}^{-1}$ be the induced
complex structure on $[-1,1]\times S^{1}$. Then $u_{n}\circ\theta_{n}^{-1}:[-1,1]\times S^{1}\rightarrow\mathbb{R}\times M$
is a sequence of $\mathcal{H}-$holomorphic curves with harmonic perturbations
$(\theta_{n}^{-1})^{*}\gamma_{n}$ with respect to the complex structure
$i_{n}$ on $[-1,1]\times S^{1}$ and the cylindrical almost complex
structure $J$ on the target space $\mathbb{R}\times M$. From the
result $\theta_{n}^{-1}(s)=(\theta_{n}^{-})^{-1}(s)-R_{n}$, and the
fact that $\theta_{n}^{-}$ and $\theta_{n}^{+}$ converge in $C_{\text{loc}}^{\infty}$
to $\theta^{-}$ on $[-1,-1/2)$ and $\theta^{+}$ on $(1/2,1]$,
respectively, it follows that the complex structures $i_{n}$ converge
in $C_{\text{loc}}^{\infty}$ to a complex structure $\tilde{i}$
on $[-1,-1/2)\times S^{1}$ and $(1/2,1]\times S^{1}$. First, we
formulate the convergence in the case when there exists a subsequence
of $u_{n}$, still denoted by $u_{n}$, with a vanishing center action.
\end{singlespace}
\begin{thm}
\begin{singlespace}
\noindent \label{thm:There-exists-a}Let $u_{n}$ be a sequence of
$\mathcal{H}-$holomorphic cylinders with harmonic perturbations $\gamma_{n}$
that satisfy C1-C9 and possessing a subsequence having vanishing center
action. Then there exists a subsequence of $u_{n}$, still denoted
by $u_{n}$, the $\mathcal{H}-$holomorphic cylinders $u^{\pm}$ defined
on $(-\infty,0]\times S^{1}$ and $[0,\infty)\times S^{1}$, respectively,
and a point $w=(w_{a},w_{f})\in\mathbb{R}\times M$ such that for
every sequence $h_{n}\in\mathbb{R}_{+}$ and every sequence of diffeomorphisms
$\theta_{n}:[-R_{n},R_{n}]\rightarrow[-1,1]$ constructed as in Remark
\ref{rem:For-every-sequence} the following $C_{\text{loc}}^{\infty}-$
and $C^{0}-$convergence results hold (after a suitable shift of $u_{n}$
in the $\mathbb{R}-$coordinate)

\end{singlespace}\begin{singlespace}
\noindent $C_{\text{loc}}^{\infty}-$convergence:
\end{singlespace}
\begin{enumerate}
\begin{singlespace}
\item For any sequence $s_{n}\in[-R_{n}+h_{n},R_{n}-h_{n}]$ there exists
$\tau_{\{s_{n}\}}\in[-\tau,\tau]$ such that after passing to a subsequence,
the shifted maps $u_{n}(s+s_{n},t)+S_{n}s_{n}$, defined on $[-R_{n}+h_{n}-s_{n},R_{n}-h_{n}-s_{n}]\times S^{1}$,
converge in $C_{\text{loc}}^{\infty}$ to $(w_{a},\phi_{-\tau_{\{s_{n}\}}}^{\alpha}(w_{f}))$.
The shifted harmonic perturbation $1-$forms $\gamma_{n}(s+s_{n},t)$
posses a subsequence converging in $C_{\text{loc}}^{\infty}$ to $0$.
\item The left shifts $u_{n}^{-}(s,t)-R_{n}S_{n}:=u_{n}(s-R_{n},t)-R_{n}S_{n}$,
defined on $[0,h_{n})\times S^{1}$, posses a subsequence that converges
in $C_{\text{loc}}^{\infty}$ to a pseudoholomorphic half cylinder
$u^{-}=(a^{-},f^{-})$, defined on $[0,+\infty)\times S^{1}$. The
curve $u^{-}$ is asymptotic to $(w_{a},\phi_{\tau}^{\alpha}(w_{f}))$.
The left shifted harmonic perturbation $1-$forms $\gamma_{n}^{-}$
converge in $C_{\text{loc}}^{\infty}$ to an exact harmonic $1-$form
$d\Gamma^{-}$, defined on $[0,+\infty)\times S^{1}$. Their asymptotics
are $0$. 
\item The right shifts $u_{n}^{+}(s,t)+R_{n}S_{n}:=u_{n}(s+R_{n},t)+R_{n}S_{n}$,
defined on $(-h_{n},0]\times S^{1}$, posses a subsequence that converges
in $C_{\text{loc}}^{\infty}$ to a pseudoholomorphic half cylinder
$u^{+}=(a^{+},f^{+})$, defined on $(-\infty,0]\times S^{1}$. The
curve $u^{+}$ is asymptotic to $(w_{a},\phi_{-\tau}^{\alpha}(w_{f}))$.
The right shifted harmonic perturbation $1-$forms $\gamma_{n}^{+}$
converge in $C_{\text{loc}}^{\infty}$ to an exact harmonic $1-$form
$d\Gamma^{+}$, defined on $(-\infty,0]\times S^{1}$. Their asymptitics
are $0$.
\end{singlespace}
\end{enumerate}
\begin{singlespace}
\noindent $C^{0}-$convergence: 
\end{singlespace}
\begin{enumerate}
\begin{singlespace}
\item The maps $v_{n}:[-1/2,1/2]\times S^{1}\rightarrow\mathbb{R}\times M$
defined by $v_{n}(s,t)=u_{n}(\theta_{n}^{-1}(s),t)$, converge in
$C^{0}$ to $(-2\sigma s+w_{a},\phi_{-2\tau s}^{\alpha}(w_{f}))$. 
\item The maps $v_{n}^{-}-R_{n}S_{n}:[-1,-1/2]\times S^{1}\rightarrow\mathbb{R}\times M$
defined by $v_{n}^{-}(s,t)=u_{n}((\theta_{n}^{-})^{-1}(s),t)$, converge
in $C^{0}$ to a map $v^{-}:[-1,-1/2]\times S^{1}\rightarrow\mathbb{R}\times M$
such that $v^{-}(s,t)=u^{-}((\theta^{-})^{-1}(s),t)$ and $v^{-}(-1/2,t)=(w_{a},\phi_{\tau}^{\alpha}(w_{f}))$. 
\item The maps $v_{n}^{+}+R_{n}S_{n}:[1/2,1]\times S^{1}\rightarrow\mathbb{R}\times M$
defined by $v_{n}^{+}(s,t)=u_{n}((\theta_{n}^{+})^{-1}(s),t)$, converge
in $C^{0}$ to a map $v^{+}:[1/2,1]\times S^{1}\rightarrow\mathbb{R}\times M$
such that $v^{+}(s,t)=u^{+}((\theta^{+})^{-1}(s),t)$ and $v^{+}(1/2,t)=(w_{a},\phi_{-\tau}^{\alpha}(w_{f}))$. 
\end{singlespace}
\end{enumerate}
\end{thm}
\begin{singlespace}
\noindent An immediate Corollary is
\end{singlespace}
\begin{cor}
\begin{singlespace}
\noindent Unter the same hypothesis of Theorem \ref{thm:There-exists-a}
the following $C_{\text{loc}}^{\infty}-$convergence results hold.
\end{singlespace}
\begin{enumerate}
\begin{singlespace}
\item The maps $v_{n}^{-}-R_{n}S_{n}$ converges in $C_{\text{loc}}^{\infty}$
to $v^{-}$, where $v^{-}$ is asymptotic to $(w_{a},\phi_{\tau}^{\alpha}(w_{f}))$
as $s\rightarrow-1/2$. The harmonic $1-$forms $[(\theta_{n}^{-})^{-1}]^{*}\gamma_{n}^{-}$
with respect to the complex structrue $[(\theta_{n}^{-})^{-1}]^{*}i$
converges in $C_{\text{loc}}^{\infty}$ to a harmonic $1-$form $[(\theta^{-})^{-1}]^{*}d\Gamma^{-}$
with respect to the complex structure $[(\theta^{-})^{-1}]^{*}i$
which is asymptotic to some constant as $s\rightarrow-1/2$.
\item The maps $v_{n}^{+}+R_{n}S_{n}$ converges in $C_{\text{loc}}^{\infty}$
to $v^{+}$, where $v^{+}$ is asymptotic to $(w_{a},\phi_{-\tau}^{\alpha}(w_{f}))$
as $s\rightarrow1/2$. The harmonic $1-$forms $[(\theta_{n}^{+})^{-1}]^{*}\gamma_{n}^{-}$
with respect to the complex structrue $[(\theta_{n}^{+})^{-1}]^{*}i$
converges in $C_{\text{loc}}^{\infty}$ to a harmonic $1-$form $[(\theta^{+})^{-1}]^{*}d\Gamma^{+}$
with respect to the complex structure $[(\theta^{+})^{-1}]^{*}i$
which is asymptotic to some constant as $s\rightarrow1/2$.
\end{singlespace}
\end{enumerate}
\end{cor}
\begin{singlespace}
\noindent Next we formulate the convergence in the case when there
is no subsequence of $u_{n}$ with a vanishing center action. This
result follows from Theorem 4 in \cite{key-23}. 
\end{singlespace}
\begin{thm}
\begin{singlespace}
\noindent \label{thm:There-exist-the}Let $u_{n}$ be a sequence of
$\mathcal{H}-$holomorphic cylinders with harmonic perturbations $\gamma_{n}$
satisfy C1-A9 and possesing no subsequence with vanishing center action.
Then there exist a subsequence of $u_{n}$, still denoted by $u_{n}$,
the $\mathcal{H}-$holomorphic half cylinders $u^{\pm}$ defined on
$(-\infty,0]\times S^{1}$ and $[0,\infty)\times S^{1}$, respectively,
a periodic orbit $x$ of period $T\in\mathbb{R}\backslash\{0\}$,
and the sequences $\overline{r}_{n}^{\pm}\in\mathbb{R}$ with $|\overline{r}_{n}^{+}-\overline{r}_{n}^{n}|\rightarrow\infty$
as $n\rightarrow\infty$ such that for every sequence $h_{n}\in\mathbb{R}_{+}$
and every sequence of diffeomorphisms $\theta_{n}:[-R_{n},R_{n}]\rightarrow[-1,1]$
as in Remark \ref{rem:For-every-sequence}, the following convergence
results hold (after a suitable shift of $u_{n}$ in the $\mathbb{R}-$coordinate).

\end{singlespace}\begin{singlespace}
\noindent $C_{\text{loc}}^{\infty}-$convergence:
\end{singlespace}
\begin{enumerate}
\begin{singlespace}
\item For any sequence $s_{n}\in[-R_{n}+h_{n},R_{n}-h_{n}]$ there exists
$\tau_{\{s_{n}\}}\in[-\tau,\tau]$ such that after passing to a subsequence,
the shifted maps $u_{n}(s+s_{n},t)-s_{n}T-S_{n}s_{n}$, defined on
$[-R_{n}+h_{n}-s_{n},R_{n}-h_{n}-s_{n}]\times S^{1}$, converge in
$C_{\text{loc}}^{\infty}$ to $(Ts+a_{0},\phi_{-\tau_{\{s_{n}\}}}^{\alpha}(x(Tt))=x(Tt+\tau_{\{s_{n}\}}))$.
The shifted harmonic perturbation $1-$forms $\gamma_{n}(s+s_{n},t)$
posses a subsequence converging in $C_{\text{loc}}^{\infty}$ to $0$.
\item The left shifts $u_{n}^{-}(s,t)-R_{n}S_{n}$, defined on $[0,h_{n})\times S^{1}$,
posses a subsequence that converges in $C_{\text{loc}}^{\infty}$
to a $\mathcal{H}-$holomorphic half cylinder $u^{-}=(a^{-},f^{-})$,
defined on $[0,+\infty)\times S^{1}$. The curve $u^{-}$ is asymptotic
to $(Ts+a_{0},\phi_{\tau}^{\alpha}(x(Tt))=x(Tt+\tau))$. The left
shifted harmonic perturbation $1-$forms $\gamma_{n}^{-}$ converge
in $C_{\text{loc}}^{\infty}$ to an exact harmonic $1-$form $d\Gamma^{-}$,
defined on $[0,+\infty)\times S^{1}$. Their asymptotics are $0$.
\item The right shifts $u_{n}^{+}(s,t)+R_{n}S_{n}$, defined on $(-h_{n},0]\times S^{1}$
posses a subsequence that converges in $C_{\text{loc}}^{\infty}$
to a $\mathcal{H}-$holomorphic half cylinder $u^{+}=(a^{+},f^{+})$,
defined on $(-\infty,0]\times S^{1}$. The curve $u^{+}$ is asymptotic
to $(Ts+a_{0},\phi_{-\tau}^{\alpha}(x(Tt))=x(Tt-\tau))$. The right
shifted harmonic perturbation $1-$forms $\gamma_{n}^{+}$ converge
in $C_{\text{loc}}^{\infty}$ to an exact harmonic $1-$form $d\Gamma^{+}$,
defined on $(-\infty,0]\times S^{1}$. Their asymptotics are $0$.
\end{singlespace}
\end{enumerate}
\begin{singlespace}
\noindent $C^{0}-$convergence: 
\end{singlespace}
\begin{enumerate}
\begin{singlespace}
\item The maps $f_{n}\circ\theta_{n}^{-1}:[-1/2,1/2]\times S^{1}\rightarrow M$
converge in $C^{0}$ to $\phi_{-2\tau s}^{\alpha}(x(Tt))=x(Tt-2\tau s)$.
\item The maps $f_{n}^{-}\circ(\theta_{n}^{-})^{-1}:[-1,-1/2]\times S^{1}\rightarrow M$
converge in $C^{0}$ to a map $f^{-}\circ(\theta^{-})^{-1}:[-1,-1/2]\times S^{1}\rightarrow M$
such that $f^{-}((\theta^{-})^{-1}(-1/2),t)=\phi_{\tau}^{\alpha}(x(Tt))=x(Tt+\tau)$.
\item The maps $f_{n}^{+}\circ(\theta_{n}^{+})^{-1}:[1/2,1]\times S^{1}\rightarrow M$
converge in $C^{0}$ to a map $f^{+}\circ(\theta^{+})^{-1}:[1/2,1]\times S^{1}\rightarrow M$
such that $f^{+}((\theta^{+})^{-1}(1/2),t)=\phi_{-\tau}^{\alpha}(x(Tt))=x(Tt-\tau)$.
\item There exist $C>0$, $\rho>0$ and $N\in\mathbb{N}$ such that for
any $R>0$, $a_{n}\circ\theta_{n}^{-1}(s,t)\in[\overline{r}_{n}^{-}+R-C,\overline{r}_{n}^{+}-R+C]$
for all $n\geq N$ and all $(s,t)\in[-\rho,\rho]\times S^{1}$.
\end{singlespace}
\end{enumerate}
\end{thm}
\begin{singlespace}
\noindent An immediate corallary is
\end{singlespace}
\begin{cor}
\begin{singlespace}
\noindent Unter the same hypothesis of Theorem \ref{thm:There-exist-the}
and the notations from Theorem \ref{rem:For-every-sequence} we have
the following $C_{\text{loc}}^{\infty}-$convergence results.
\end{singlespace}
\begin{enumerate}
\begin{singlespace}
\item The maps $v_{n}^{-}-R_{n}S_{n}$ converges in $C_{\text{loc}}^{\infty}$
to $v^{-}$ where $f^{-}((\theta^{-})^{-1}(-1/2),t)=x(Tt+\tau)$.
The harmonic $1-$forms $[(\theta_{n}^{-})^{-1}]^{*}\gamma_{n}^{-}$
with respect to the complex structrue $[(\theta_{n}^{-})^{-1}]^{*}i$
converges in $C_{\text{loc}}^{\infty}$ to a harmonic $1-$form $[(\theta^{-})^{-1}]^{*}d\Gamma^{-}$
with respect to the complex structure $[(\theta^{-})^{-1}]^{*}i$
which is asymptotic to some constant as $s\rightarrow-1/2$.
\item The maps $v_{n}^{+}+R_{n}S_{n}$ converges in $C_{\text{loc}}^{\infty}$
to $v^{+}$ where $f^{+}((\theta^{+})^{-1}(1/2),t)=x(Tt-\tau)$. The
harmonic $1-$forms $[(\theta_{n}^{+})^{-1}]^{*}\gamma_{n}^{-}$ with
respect to the complex structrue $[(\theta_{n}^{+})^{-1}]^{*}i$ converges
in $C_{\text{loc}}^{\infty}$ to a harmonic $1-$form $[(\theta^{+})^{-1}]^{*}d\Gamma^{+}$
with respect to the complex structure $[(\theta^{+})^{-1}]^{*}i$
which is asymptotic to some constant as $s\rightarrow1/2$.
\end{singlespace}
\end{enumerate}
\end{cor}
\begin{singlespace}
\noindent Since $\theta^{-}:[0,\infty)\times S^{1}\rightarrow[-1,-1/2)\times S^{1}$
is a biholomorphism with respect to the standard complex structrue
$i$ on the domain and the pull-back structure $\tilde{i}:=[(\theta^{-})^{-1}]^{*}i$,
we can identify $[-1,-1/2)\times S^{1}$ with the punctured disk equipped
with the standard complex structure, that extends over the puncture.

\noindent We use now Theorems \ref{thm:There-exists-a} and \ref{thm:There-exist-the}
to describe the limit object. 

\noindent In Case 1, the ``limit surface'' in the symplectization
consists of two disks which are connected by a straight line at the
origin. The limit map $u=(a,f):[-1,1]\times S^{1}\rightarrow\mathbb{R}\times M$
with the limit perturbation $1-$form $\gamma$ can be described as
follows (see Figure \ref{fig18}).
\end{singlespace}
\begin{description}
\begin{singlespace}
\item [{D1}] \noindent On $[-1,-1/2)\times S^{1}$, $u$ is a $\mathcal{H}-$holomorphic
curve with harmonic perturbation $\gamma$ such that at the puncture
it is asymptotic to $(\sigma+w_{a},\phi_{\tau}^{\alpha}(w_{f}))$,
while the harmonic perturbation is asymptotic to a constant. 
\item [{D2}] \noindent On $(1/2,1]\times S^{1}$, $u$ is a $\mathcal{H}-$holomorphic
curve with harmonic perturbation $\gamma$ such that at the puncture
it is asymptotic to $(-\sigma+w_{a},\phi_{-\tau}^{\alpha}(w_{f}))$,
while the harmonic\textbf{ }perturbation is asymptotic to a constant.
\item [{D3}] \noindent On the middle part $[-1/2,1/2]\times S^{1}$, $u$
is given by $u(s,t)=(-2\sigma s+w_{a},\phi_{-2\tau s}^{\alpha}(w_{f}))$.
On this part the $1-$form $\gamma$ is not defined.
\end{singlespace}
\end{description}
\begin{singlespace}
\noindent \begin{figure}     
	\centering  
	\def\svgscale{1.3} 	
	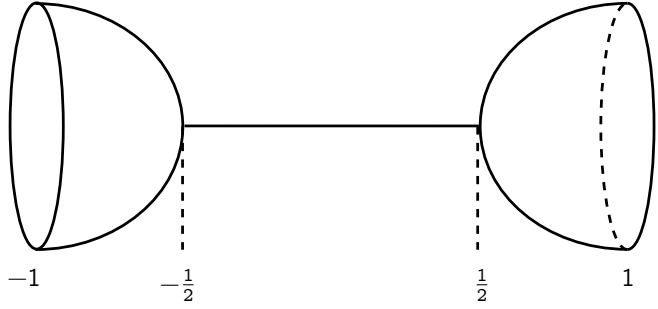  
	\caption{The limit surface consists of two cones connected by a straight line.}
	\label{fig18}
\end{figure}

\noindent In Case 2, the limit surface is the disjoint union of the
cylinders $[-1,-1/2)\times S^{1}$ and $(1/2,1]\times S^{1}$. The
$\mathcal{H}-$holomorphic curve $u=(a,f):([-1,-1/2)\coprod(1/2,1])\times S^{1}\rightarrow\mathbb{R}\times M$
with harmonic perturbation $\gamma$ can be described as follows.
\end{singlespace}
\begin{description}
\begin{singlespace}
\item [{D1'}] \noindent $u$ is asymptotic on $[-1,-1/2)\times S^{1}$
and $(1/2,1]\times S^{1}$ to a trivial cylinder over the Reeb orbit
$x(Tt+\tau)$ or $x(Tt-\tau)$, respectively, while the harmonic perturbation\textbf{
}is asymptotic to a constant. 
\item [{D2'}] \noindent On the middle part $[-1/2,1/2]\times S^{1}$, the
$M-$component $f$ is given by $f(s,t)=x(Tt-2\tau s)$.
\end{singlespace}
\end{description}
\begin{singlespace}

\subsubsection{\label{subsec:Cylinders-of-type-1}Cylinders of type $b_{1}$}
\end{singlespace}

\begin{singlespace}
\noindent We analyze the convergence on cylinders of type $b_{1}$
by using the results of Appendix \ref{sec:Pseudoholomorphic-disks-with}.
Let $m\in\{1,...,N\}$ be such that the cylinders $[h_{n}^{(m-1)}-3h,h_{n}^{(m)}+3h]\times S^{1}$
are of type $b_{1}$. By the construction described in the previous
section and Lemma \ref{lem:Let-=000024=00005Bh_=00007Bn=00007D^=00007B(m-1)=00007D,h_=00007Bn=00007D^=00007B(m)=00007D=00005D},
the $\mathcal{H}-$holomorphic curves have uniform gradient bounds
on the two boundary cylinders $[h_{n}^{(m-1)}-3h,h_{n}^{(m-1)}]\times S^{1}$
and $[h_{n}^{(m)},h_{n}^{(m)}+3h]\times S^{1}$. 

\noindent The convergence analysis is organized as follows. As in
Section \ref{2_1_sec:First_Theorem} we apply bubbling-off analysis
on the cylinder $[h_{n}^{(m-1)},h_{n}^{(m)}]\times S^{1}$ to show
that on any compact set in the complement of a finite number of points
$\mathcal{Z}^{(m)}$ in $[h_{n}^{(m-1)}-3h,h_{n}^{(m)}+3h]\times S^{1}$,
the gradient of $u_{n}$ is uniformly bounded. The points on which
the gradient might blow up are located in $(h_{n}^{(m-1)}-h,h_{n}^{(m)}+h)\times S^{1}$.
Each resulting puncture from $\mathcal{Z}^{(m)}$ lies in a disk $D_{r}$
of radius $r$ smaller than $h/2$. For a smaller radius $r$, we
assume that all disks $D_{r}$ are pairwise disjoint and that their
union lies in $(h_{n}^{(m-1)}-h,h_{n}^{(m)}+h)\times S^{1}$ (see
Figure \ref{fig17}). 

\noindent \begin{figure}     
	\centering  
	\def\svgscale{0.80} 	
	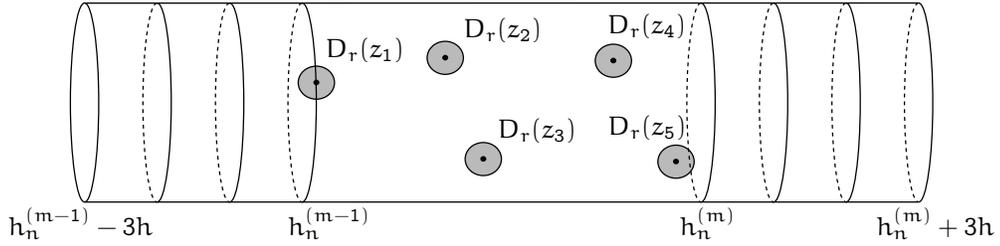  
	\caption{The gradient might blow up on the discs $D_{r}(z_{i})$ contained in $(h_{n}^{(m-1)} - h, h_{n}^{(m)} + h) \times S^{1}$.}
	\label{fig17}
\end{figure}

\noindent Under these assumptions, the $\mathcal{H}-$holomorphic
curves converge in $C^{\infty}$ on the complement of the union of
these disks (centered at the punctures) to a $\mathcal{H}-$holomorphic
curve. What is left to prove is the convergence in each $D_{r}$;
for this we use the results of Appendix \ref{sec:Pseudoholomorphic-disks-with}.
In the final step, we glue the convergence results on the disks to
the rest of the cylinder, and obtain the desired description on the
entire cylinder of type $b_{1}$. 

\noindent Under the biholomorphic map $[h_{n}^{(m-1)}-3h,h_{n}^{(m)}+3h]\times S^{1}\rightarrow[0,H_{n}^{(m)}]\times S^{1}$,$(s,t)\mapsto(s-h_{n}^{(m-1)}+3h,t)$,
where $H_{n}^{(m)}:=h_{n}^{(m)}-h_{n}^{(m-1)}+6h$, assume that the
$\mathcal{H}-$holomorphic curves $u_{n}$ together with the harmonic
perturbations $\gamma_{n}$ are defined on $[0,H_{n}^{(m)}]\times S^{1}$.
By going over to a subsequence, we have $H_{n}^{(m)}\rightarrow H^{(m)}$
as $n\rightarrow\infty$. Consider the translated $\mathcal{H}-$holomorphic
curves $u_{n}-a_{n}(0,0)=(a_{n},f_{n})-a_{n}(0,0):[0,H_{n}^{(m)}]\times S^{1}\rightarrow\mathbb{R}\times M$
with harmonic perturbations $\gamma_{n}$. In order to keep the notation
simple, let the curve $u_{n}-a_{n}(0,0)$ be still denoted by $u_{n}$.
The analysis is performed in the following setting:
\end{singlespace}
\begin{description}
\begin{singlespace}
\item [{E1}] \noindent The maps $u_{n}=(a_{n},f_{n})$ are $\mathcal{H}-$holomorphic
curves with harmonic perturbation $\gamma_{n}$ on $[0,H_{n}^{(m)}]\times S^{1}$
with respect to the standard complex structure $i$ on the domain
and the almost complex structure $J$ on $\xi$.
\item [{E2}] \noindent The maps $u_{n}$ have uniformly bounded energies,
while the harmonic perturbations $\gamma_{n}$ have uniformly bounded
$L^{2}-$norms, i.e., with the constants $E_{0},C_{0}>0$ we have
$E(u_{n};[0,H_{n}^{(m)}]\times S^{1})\leq E_{0}$ and $\left\Vert \gamma_{n}\right\Vert _{L^{2}([0,H_{n}^{(m)}]\times S^{1})}^{2}\leq C_{0}$
for all $n\in\mathbb{N}$.
\item [{E3}] \noindent The maps $u_{n}$ have uniformly bounded gradients
on $[0,3h]\times S^{1}$ and $[H_{n}^{(m)}-3h,H_{n}^{(m)}]\times S^{1}$
with respect to the Euclidean metric on the domain and the cylindrical
metric on the target space $\mathbb{R}\times M$, i.e.
\[
\left\Vert du_{n}(z)\right\Vert =\sup_{\left\Vert v\right\Vert _{\text{eucl.}}=1}\left\Vert du_{n}(z)v\right\Vert _{\overline{g}}<C_{h}
\]
for all $z\in([0,3h]\cup[H_{n}^{(m)}-3h,H_{n}^{(m)}])\times S^{1}$
and $n\in\mathbb{N}$.
\end{singlespace}
\end{description}
\begin{singlespace}
\noindent The next lemma states the existence of a finite set $\mathcal{Z}^{(m)}$
of punctures on which the gradient of $u_{n}$ blows up.
\end{singlespace}
\begin{lem}
\begin{singlespace}
\noindent \label{lem:There-exists-a-1}There exists a finite set of
points $\mathcal{Z}^{(m)}\subset[3h,H_{n}^{(m)}-3h]\times S^{1}$
such that for any compact subset $\mathcal{K}\subset([0,H_{n}^{(m)}]\times S^{1})\backslash\mathcal{Z}^{(m)}$
there exists a constant $C_{\mathcal{K}}>0$ such that
\[
\left\Vert du_{n}(z)\right\Vert =\sup_{\left\Vert v\right\Vert _{\text{eucl.}}=1}\left\Vert du_{n}(z)v\right\Vert _{\overline{g}}<C_{\mathcal{K}}
\]
for all $z\in\mathcal{K}$ and $n\in\mathbb{N}$.
\end{singlespace}
\end{lem}
\begin{proof}
\begin{singlespace}
\noindent The proof relies on the same arguments of bubbling-off analysis,
which have been employed in Theorem \ref{thm:There-exists-finitely}
from Section \ref{2_1_sec:First_Theorem} for the thick part.
\end{singlespace}
\end{proof}
\begin{singlespace}
\noindent Pick some $r>0$ such that $r<h/2$, and let $D_{r}(\mathcal{Z}^{(m)})$
consists of $|\mathcal{Z}^{(m)}|$ pairwise disjoint closed disks
of radius $r>0$, centered at the punctures of $\mathcal{Z}^{(m)}$.
Obviously, $D_{r}(\mathcal{Z}^{(m)})\subset(2h,H_{n}^{(m)}-2h)\times S^{1}$.
Then by Lemma \ref{lem:There-exists-a-1}, $u_{n}$ has a uniformly
bounded gradient on $([0,H_{n}^{(m)}]\times S^{1})\backslash D_{r}(\mathcal{Z}^{(m)})$.
As $([0,H_{n}^{(m)}]\times S^{1})\backslash D_{r}(\mathcal{Z}^{(m)})$
is connected, we assume, after going over to some subsequence, that
$u_{n}|_{([0,H_{n}^{(m)}]\times S^{1})\backslash D_{r}(\mathcal{Z}^{(m)})}$
converge in $C^{\infty}$ to some smooth map $u|_{([0,H^{(m)}]\times S^{1})\backslash D_{r}(\mathcal{Z}^{(m)})}=(a,f)|_{([0,H^{(m)}]\times S^{1})\backslash D_{r}(\mathcal{Z}^{(m)})}$.
Before treating the convergence of the $\mathcal{H}-$holomorphic
curves in a neighbourhood of the punctures of $\mathcal{Z}^{(m)}$,
we establish the convergence of the harmonic perturbations $\gamma_{n}$
on $[0,H_{n}^{(m)}]\times S^{1}$, so that at the end
\end{singlespace}
\begin{itemize}
\begin{singlespace}
\item $u_{n}|_{([0,H_{n}^{(m)}]\times S^{1})\backslash D_{r}(\mathcal{Z}^{(m)})}$
converge in $C^{\infty}$ to a $\mathcal{H}-$holomorphic curve $u|_{([0,H^{(m)}]\times S^{1})\backslash D_{r}(\mathcal{Z}^{(m)})}$,
and
\item the harmonic perturbations $\gamma_{n}$ have uniformly bounded $C^{k}-$norms
on the disks $D_{r}(\mathcal{Z}^{(m)})$ for all $k\in\mathbb{N}_{0}$. 
\end{singlespace}
\end{itemize}
\begin{singlespace}
\noindent The latter result is needed to describe the convergence
of the harmonic perturbations $\gamma_{n}$ on the disks $D_{r}(\mathcal{Z}^{(m)})$.
As in the previous section, we set $\gamma_{n}=f_{n}ds+g_{n}dt$,
where $f_{n}$ and $g_{n}$ are harmonic functions defined on $[0,H_{n}^{(m)}]\times S^{1}$
such that $f_{n}+ig_{n}$ are holomorphic. By the uniform $L^{2}-$bound
of $\gamma_{n}$ it follows that
\[
\left\Vert \gamma_{n}\right\Vert _{L^{2}([0,H_{n}^{(m)}]\times S^{1})}^{2}=\int_{[0,H_{n}^{(m)}]\times S^{1}}\left(f_{n}^{2}+g_{n}^{2}\right)dsdt\leq C_{0}
\]
for all $n\in\mathbb{N}$, and so, that the $L^{2}-$norms of the
holomorphic functions $f_{n}+ig_{n}$ are uniformly bounded. Letting
$G_{n}=f_{n}+ig_{n}$ we state the following 
\end{singlespace}
\begin{prop}
\begin{singlespace}
\noindent \label{prop:There-exists-a-1-1}There exists a subsequence
of $G_{n}$, also denoted by $G_{n}$, that converges in $C^{\infty}$
to some holomorphic map $G$ defined on $[0,H^{(m)}]\times S^{1}$.
Moreover, the harmonic perturbations $\gamma_{n}$ converge in $C^{\infty}$
to a harmonic map $\gamma$. 
\end{singlespace}
\end{prop}
\begin{proof}
\begin{singlespace}
\noindent By Proposition \ref{prop:There-exists-a-1}, $G_{n}$ has
a uniformly bounded $C^{1}-$norm, while by the standard regularity
results from the theory of pseudoholomorphic curves (see, for example,
Section 2.2.3 of \cite{key-6}), the $C^{k}$ derivatives of $G_{n}$
are also uniformly bounded. Hence, in view of Arzelá-Ascoli theorem,
we can extract a subsequence that converges to some holomorphic function
$G$.
\end{singlespace}
\end{proof}
\begin{singlespace}
\noindent Let us analyze the convergence of the $\mathcal{H}-$holomorphic
curves in a neighbourhood of the punctures of $\mathcal{Z}^{(m)}$,
which are given by Lemma \ref{lem:There-exists-a-1}. For $r>0$ as
above and $z\in\mathcal{Z}^{(m)}$, consider the closed disks $D_{r}(z)$
and the $\mathcal{H}-$holomorphic curves $u_{n}=(a_{n},f_{n}):D_{r}(z)\rightarrow\mathbb{R}\times M$
with harmonic perturbations $\gamma_{n}$ that converge in $C^{\infty}$
to some harmonic $1-$form $\gamma$. According to the biholomorphism
$D\rightarrow D_{r}(z)$, $p\mapsto rp+z$, where $D$ is the standard
closed unit disk, regard the $\mathcal{H}-$holomorphic curves $u_{n}$
together with the harmonic perturbations as beeing defined on $D$
instead of $D_{r}(z)$. The following setting is pertinent to our
analysis:
\end{singlespace}
\begin{description}
\begin{singlespace}
\item [{F1}] \noindent The maps $u_{n}=(a_{n},f_{n}):D\rightarrow\mathbb{R}\times M$
are $\mathcal{H}-$holomorphic curves with harmonic perturbations
$\gamma_{n}$ with respect to the standard complex structure $i$
on $D$ and the almost complex structure $J$ on $\xi$.
\item [{F2}] \noindent The maps $u_{n}=(a_{n},f_{n})$ and $\gamma_{n}$
have uniformly bounded energies and $L^{2}-$norms.
\item [{F3}] \noindent For any constant $1>\tau>0$, $u_{n}|_{A_{1,\tau}}=(a_{n},f_{n})|_{A_{1,\tau}}$
converge in $C^{\infty}$ to a $\mathcal{H}-$holomorphic map with
harmonic perturbation $\gamma$, where $A_{1,\tau}=\{z\in D\mid\tau\leq|z|\leq1\}$. 
\end{singlespace}
\end{description}
\begin{singlespace}
\noindent As the domain of definition $D$ is simply connected, we
infer that $\gamma_{n}$ is exact, i.e. it can be written as $\gamma_{n}=d\tilde{\Gamma}_{n}$,
where $\tilde{\Gamma}_{n}:D\rightarrow\mathbb{R}$ is a harmonic function.
By Condition F2, $\tilde{\Gamma}_{n}$ has a uniformly bounded gradient
$\nabla\tilde{\Gamma}_{n}$ in the $L^{2}-$norm, and it is apparent
that the existence of $\tilde{\Gamma}_{n}$ is unique up to addition
by a constant. Let us make some remarks on the choice of $\tilde{\Gamma}_{n}$
and discuss some of its properties. By using the mean value theorem
for harmonic functions as in Proposition \ref{prop:There-exists-a-1}
we conclude (after eventually, shrinking $D$) that the gradient $\nabla\tilde{\Gamma}_{n}$
are uniformly bounded in $C^{0}$. Denote by $z=s+it$ the coordinates
on $D$, and let
\[
K_{n}=\frac{1}{\pi}\int_{D}\tilde{\Gamma}_{n}(s,t)dsdt
\]
be the mean value of $\tilde{\Gamma}_{n}$, so that by the mean value
theorem for harmonic functions, $K_{n}=\tilde{\Gamma}_{n}(0)$. Finally,
define the map $\Gamma_{n}(z):=\tilde{\Gamma}_{n}(z)-\tilde{\Gamma}_{n}(0)$
which obviously satisfies $\gamma_{n}=d\Gamma_{n}$. 
\end{singlespace}
\begin{rem}
\begin{singlespace}
\noindent \label{rem:From-Poincare-inequlity}From Poincaré inequality
it follows that $\left\Vert \Gamma_{n}\right\Vert _{L^{2}(D)}\leq c\left\Vert \nabla\tilde{\Gamma}_{n}\right\Vert _{L^{2}(D)}$
for some constant $c>0$ and so, that $\Gamma_{n}$ is uniformly bounded
in $L^{2}-$norm. Again, by using the mean value theorem for harmonic
functions, we deduce (after maybe shrinking $D$) that $\Gamma_{n}$
has a uniformly bounded $C^{0}-$norm, and consequently, that $\Gamma_{n}$
has a uniformly bounded $C^{1}-$norm. Because $\gamma_{n}=d\Gamma_{n}$
is a harmonic $1-$form, $\partial_{s}\Gamma_{n}+i\partial_{t}\Gamma_{n}$
is a holomorphic function. In this context, by Proposition \ref{prop:There-exists-a-1},
$\Gamma_{n}$ converge in $C^{\infty}$ to a harmonic function $\Gamma:D\rightarrow\mathbb{R}$. 
\end{singlespace}
\end{rem}
\begin{singlespace}
\noindent In the following we transform the $\mathcal{H}-$holomorphic
curves defined on the disk in a usual pseudoholomorphic curve by encoding
the harmonic perturbation $\gamma_{n}=d\Gamma_{n}$ in the $\mathbb{R}-$coordinate
of the $\mathcal{H}-$holomorphic curve $u_{n}$. Specifically, we
define the maps $\overline{u}_{n}=(\overline{a}_{n},\overline{f}_{n})=(a_{n}+\Gamma_{n},f_{n})$
which are obviously pseudoholomorphic. The transformation is usable
if we ensure that the energy bounds are still satisfied. For an ordinary
pseudoholomorphic curve, the sum of the $\alpha-$ and $d\alpha-$energies,
that are both positive, yield the Hofer energy $E_{\text{H}}(\overline{u}_{n};D)$.
A uniform bound on the Hofer energy, which ensures a uniform bound
on the $\alpha-$ and $d\alpha-$energies of $\overline{u}_{n}$,
is
\[
E_{\text{H}}(\overline{u}_{n};D)=\sup_{\varphi\in\mathcal{A}}\int_{D}\overline{u}_{n}^{*}d(\varphi\alpha)=\sup_{\varphi\in\mathcal{A}}\int_{\partial D}\varphi(\overline{a}_{n})\overline{f}_{n}^{*}\alpha\leq\int_{\partial D}|\overline{f}_{n}^{*}\alpha|\leq C_{h}.
\]
Here, the last inequality follows from Condition F3, according to
which, $u_{n}$ converge in $C^{\infty}$ in a fixed neighbourhood
of $\partial D$. Note that the constant $C_{h}$ is guaranteed by
Lemma \ref{lem:There-exists-a-1}.

\noindent In a next step we use the results of Appendix \ref{sec:Pseudoholomorphic-disks-with}
to establish the convergence of the maps $\overline{u}_{n}$ and to
describe their limit object. Then we undo the transformation in the
$\mathbb{R}-$coordinate (more precisely, the encoding of $\gamma_{n}$
in the $\mathbb{R}-$coordinate of the curve $u_{n}$) and give a
convergence result together with a description of the limit object
for $u_{n}$. Before proceeding we state the setting corresponding
to the pseudoholomorphic curves $\overline{u}_{n}$.
\end{singlespace}
\begin{description}
\begin{singlespace}
\item [{G1}] \noindent The maps $\overline{u}_{n}=(\overline{a}_{n},\overline{f}_{n}):D\rightarrow\mathbb{R}\times M$
solve the pseudoholomorphic curve equation
\begin{align*}
\pi_{\alpha}d\overline{f}_{n}\circ i & =J(\overline{f}_{n})\circ\pi_{\alpha}d\overline{f}_{n},\\
\overline{f}_{n}^{*}\alpha\circ i & =d\overline{a}_{n}
\end{align*}
on $D$. 
\item [{G2}] \noindent The maps $\overline{u}_{n}$ have uniformly bounded
energies.
\item [{G3}] \noindent For any $\tau>0$, $\overline{u}_{n}|_{A_{1,\tau}}=(\overline{a}_{n},\overline{f}_{n})|_{A_{1,\tau}}$
converge in $C^{\infty}$ to a pseudoholomorphic map.
\end{singlespace}
\end{description}
\begin{singlespace}
\noindent We consider two cases. In the first case, the $\mathbb{R}-$components
of $\overline{u}_{n}$ are uniformly bounded, while in the second
case they are not. Actually, the first case does not occur. We will
prove this result in the next lemma by using standard bubbling-off
analysis. Let $z_{n}\in D$ be the sequence choosen from the bubbling-off
argument of Lemma \ref{lem:There-exists-a-1}, i.e. for which we have
that
\begin{equation}
\left\Vert d\overline{u}_{n}(z_{n})\right\Vert _{C^{0}}=\sup_{z\in D}\left\Vert d\overline{u}_{n}(z)\right\Vert _{C^{0}}\rightarrow\infty\label{eq:grad_blow_disc}
\end{equation}
as $n\rightarrow\infty$.
\end{singlespace}
\begin{lem}
\begin{singlespace}
\noindent The $\mathbb{R}-$coordinates of the maps $\overline{u}_{n}$
are unbounded on $D$.
\end{singlespace}
\end{lem}
\begin{proof}
\begin{singlespace}
\noindent We prove by contradiction using bubbling-off analysis. Assume
that the $\mathbb{R}-$coordinates of the maps $\overline{u}_{n}$
are uniformly bounded. Employing the same arguments as in the proof
of Lemma \ref{lem:Let-=000024=00005Bh_=00007Bn=00007D^=00007B(m-1)=00007D,h_=00007Bn=00007D^=00007B(m)=00007D=00005D}
for the sequence $R_{n}:=\left\Vert d\overline{u}_{n}(z_{n})\right\Vert _{C^{0}}$,
we find that the maps $v_{n}:D_{\epsilon'_{n}R'_{n}}(0)\rightarrow\mathbb{R}\times M$
converge in $C_{\text{loc}}^{\infty}(\mathbb{C})$ to a non-constant
finite energy holomorphic plane $v$. Note that the boundedness of
$E_{d\alpha}(v;\mathbb{C})$ follows from the fact that for an arbitray
$R>0$ we have 
\begin{align*}
E_{d\alpha}(v,D_{R}(0))=\lim_{n\rightarrow\infty}E_{d\alpha}(v_{n};D_{R}(0))\leq\lim_{n\rightarrow\infty}E_{d\alpha}(v_{n};D_{\epsilon'_{n}R'_{n}}(0))\leq C_{h},
\end{align*}
yielding $E_{d\alpha}(v;\mathbb{C})\leq C_{h}$. As we have assumed
that the $\mathbb{R}-$coordinates of $\overline{u}_{n}$ are uniformly
bounded it follows that the $\mathbb{R}-$coordiantes of $v_{n}$,
and so, of $v$, are also uniformly bounded. By singularity removal,
$v$ can be extended to a pseudoholomorphic sphere. Thus the $d\alpha-$energy
vanishes and by the maximum principle, the function $a$ is constant.
For this reason, $v$ must be constant and we are lead to a contradiction. 
\end{singlespace}
\end{proof}
\begin{singlespace}
\noindent We consider now the second case in which the $\mathbb{R}-$coordinates
of the maps $\overline{u}_{n}$ are unbounded, and make extensively
use of the results of Appendix \ref{sec:Pseudoholomorphic-disks-with}.
By the maximum principle, the function $\overline{a}_{n}$ tends to
$-\infty$, while by Proposition \ref{prop:Upto-passing-to}, the
maps $\overline{u}_{n}=(\overline{a}_{n},f_{n}):(D,i)\rightarrow\mathbb{R}\times M$
converge to a broken holomorphic curve $\overline{u}=(\overline{a},\overline{f}):(Z,j)\rightarrow\mathbb{R}\times M$.
Here, $Z$ is obtained as follows. Let $Z$ be a surface diffeomorphic
to $D$, and let $\Delta=\Delta_{n}\amalg\Delta_{p}\subset Z$ be
a collection of finitely many disjoint loops away from $\partial Z$.
Further on, let $Z\backslash\Delta_{p}=\coprod_{\nu=0}^{N+1}Z^{(\nu)}$
for some $N\in\mathbb{N}$ as described in Appendix \ref{sec:Pseudoholomorphic-disks-with}.
For a loop $\delta\in\Delta_{p}$, there exists $\nu\in\left\{ 0,...,N\right\} $
such that $\delta$ is adjacent to $Z^{(\nu)}$ and $Z^{(\nu+1)}$.
Fix an embedded annuli
\[
A^{\delta,\nu}\cong[-1,1]\times S^{1}\subset Z\backslash\Delta_{n}
\]
such that $\left\{ 0\right\} \times S^{1}=\delta$, $\left\{ -1\right\} \times S^{1}\subset Z^{(\nu)}$,
and $\left\{ 1\right\} \times S^{1}\subset Z^{(\nu+1)}$. In this
context, there exist a sequence of diffeomorphism $\varphi_{n}:D\rightarrow Z$
and a sequence of negative real numbers $\min(a_{n})=r_{n}^{(0)}<r_{n}^{(1)}<...<r_{n}^{(N+1)}=-K-2$,
where $K\in\mathbb{R}$ is the constant determined in Appendix \ref{sec:Pseudoholomorphic-disks-with}
and $r_{n}^{(\nu+1)}-r_{n}^{(\nu)}\rightarrow\infty$ as $n\rightarrow\infty$
such that the following hold:
\end{singlespace}
\begin{description}
\begin{singlespace}
\item [{H1}] \noindent $i_{n}:=(\varphi_{n})_{*}i\rightarrow j$ in $C_{\text{loc}}^{\infty}$
on $Z\backslash\Delta$.
\item [{H2}] \noindent The sequence $\overline{u}_{n}\circ\varphi_{n}^{-1}|_{Z^{(\nu)}}:Z^{(\nu)}\rightarrow\mathbb{R}\times M$
converges in $C_{\text{loc}}^{\infty}$ on $Z^{(\nu)}\backslash\Delta_{n}$
to a punctured nodal pseudoholomorphic curve $\overline{u}^{(\nu)}:(Z^{(\nu)},j)\rightarrow\mathbb{R}\times M$,
and in $C_{\text{loc}}^{0}$ on $Z^{(\nu)}$.
\item [{H3}] \noindent The sequence $\overline{f}_{n}\circ\varphi_{n}^{-1}:Z\rightarrow M$
converges in $C^{0}$ to a map $f:Z\rightarrow M$, whose restriction
to $\Delta_{p}$ parametrizes the Reeb orbits and to $\Delta_{n}$
parametrizes points. 
\item [{H4}] \noindent For any $S>0$ , there exist $\rho>0$ and $\tilde{N}\in\mathbb{N}$
such that $\overline{a}_{n}\circ\varphi_{n}^{-1}(s,t)\in[r_{n}^{(\nu)}+S,r_{n}^{(\nu+1)}-S]$
for all $n\geq\tilde{N}$ and all $(s,t)\in A^{\delta,\nu}$ with
$|s|\leq\rho$.
\end{singlespace}
\end{description}
\begin{singlespace}
\noindent To establish a convergence result for the $\mathcal{H}-$holomorphic
curve $u_{n}$ we undo the tranformation. The maps $u_{n}$ are given
by $u_{n}=\overline{u}_{n}-\Gamma_{n}$, where $\Gamma_{n}:D\rightarrow\text{\ensuremath{\mathbb{R}}}$
is the harmonic function defined in Remark \ref{rem:From-Poincare-inequlity}.
Observe that by Remark \ref{rem:From-Poincare-inequlity}, the $\Gamma_{n}$
converge in $C^{\infty}(D)$ to some harmonic function and are uniformly
bounded in $C^{0}(D)$. Via the above diffeomorphisms $\varphi_{n}:D\rightarrow Z$,
consider the functions $\mathscr{G}_{n}:=\Gamma_{n}\circ\varphi_{n}^{-1}:Z\rightarrow\mathbb{R}$.
Since $\Gamma_{n}$ are harmonic functions with respect to $i$, $\mathscr{G}_{n}$
are harmonic functions on $Z$ with respect to $i_{n}$. Moreover,
their gradients and absolute values are bounded in $L^{2}-$ and $C^{0}-$norms,
respectively, i.e. 
\begin{equation}
\int_{Z}d\mathscr{G}_{n}\circ i_{n}\wedge d\mathscr{G}_{n}\leq C_{0}\label{eq:grad_l2}
\end{equation}
and
\begin{equation}
\left\Vert \mathscr{G}_{n}\right\Vert _{C^{0}(Z)}\leq C_{1}\label{eq:uniform_bdd}
\end{equation}
for some constant $C_{1}>0$ and for all $n\in\mathbb{N}$, respectively.
\end{singlespace}
\begin{lem}
\begin{singlespace}
\noindent \label{lem:For-every-compact}For any compact subset $\mathcal{K}\subset(Z\backslash\Delta)$
there exists a subsequence of $\mathscr{G}_{n}$, also denoted by
$\mathscr{G}_{n}$, such that $\mathscr{G}_{n}\rightarrow\mathscr{G}$
in $C^{\infty}(\mathcal{K})$ as $n\rightarrow\infty$, where $\mathscr{G}$
is a harmonic function defined on a neighbourhood of $\mathcal{K}$.
\end{singlespace}
\end{lem}
\begin{proof}
\begin{singlespace}
\noindent Let $\mathcal{K}\subset(Z\backslash\Delta)$ be a compact
subset. By Lemma \ref{2_1_lem:There-exists-open-neighbourhood} there
exists a finite covering of $\mathcal{K}$ by the charts $\psi_{n}^{(l)}:D\rightarrow U_{n}^{(l)}$
and $\psi^{(l)}:D\rightarrow U^{(l)}$, where $l\in\{1,...,N\}$ and
$N\in\mathbb{N}$. For some $r\in(0,1)$, the following hold: 

\end{singlespace}\begin{enumerate}
\begin{singlespace}
\item $\psi_{n}^{(l)}$ are $i-i_{n}-$biholomorphisms and $\psi^{(l)}$
is an $i-j-$biholomorphism;
\item $\psi_{n}^{(l)}\rightarrow\psi^{(l)}$ in $C_{\text{loc}}^{\infty}(D)$
as $n\rightarrow\infty$;
\item $\mathcal{K}\subset\bigcup_{l=1}^{N}\psi_{n}^{(l)}(D_{r}(0))$ for
all $n\in\mathbb{N}$, and $\mathcal{K}\subset\bigcup_{l=1}^{N}\psi^{(l)}(D_{r}(0))$.
\end{singlespace}
\end{enumerate}
\begin{singlespace}
\noindent Consider the function $\mathscr{G}_{n}^{(l)}:=\mathscr{G}_{n}\circ\psi_{n}^{(l)}:D\rightarrow\mathbb{R}$
for some $l\in\{1,...,N\}$. Because $\psi_{n}^{(l)}$ are $i-i_{n}-$biholomorphisms,
$\mathscr{G}_{n}^{(l)}$ is a harmonic function with respect to $i$.
From (\ref{eq:grad_l2}) and (\ref{eq:uniform_bdd}), $\mathscr{G}_{n}^{(l)}$
satisfies
\[
\int_{D}d\mathscr{G}_{n}^{(l)}\circ i\wedge d\mathscr{G}_{n}^{(l)}\leq C_{0}
\]
and $\left\Vert \mathscr{G}_{n}^{(l)}\right\Vert _{C^{0}(D)}\leq C_{1}$.
Relying on the compactness result for harmonic functions we assume
that $\mathscr{G}_{n}^{(l)}$ converges in $C^{0}(D_{3r/2}(0))$ to
a harmonic function $\mathscr{G}^{(l)}$ defined on $D_{3r/2}(0)$.
By the mean value theorem for harmonic functions, there exists a constant
$c>0$ such that $\left\Vert \nabla\mathscr{G}_{n}^{(l)}\right\Vert _{C^{0}(D_{4r/3}(0))}\leq c$
for all $n\in\mathbb{N}$. Hence $\mathscr{G}_{n}^{(l)}$ is uniformly
bounded in $C^{1}(D_{4r/3}(0))$. Because $d\mathscr{G}_{n}^{(l)}$
defines a harmonic $1-$form, $\partial_{s}\mathscr{G}_{n}^{(l)}+i\partial_{t}\mathscr{G}_{n}^{(l)}$
is a uniformly bounded holomorphic function defined on $D_{4r/3}(0)$,
where $s,t$ are the coordinates on $D_{4r/3}(0)$. By means of the
Cauchy integral formula, all derivatives of $\partial_{s}\mathscr{G}_{n}^{(l)}+i\partial_{t}\mathscr{G}_{n}^{(l)}$
are uniformly bounded on $D_{5r/4}(0)$. From this and the fact that
$\mathscr{G}_{n}^{(l)}$ converges uniformly to $\mathscr{G}^{(l)}$
we deduce that there exists a further subsequence, also denoted by
$\mathscr{G}_{n}^{(l)}$, that converges in $C^{\infty}(D_{6r/5}(0))$
to a harmonic function $\mathscr{G}^{(l)}:D_{6r/5}(0)\rightarrow\mathbb{R}$.
For $n$ sufficiently large, $\psi^{(l)}(\overline{D_{r}(0)})\subset\psi^{(l)}(D_{6r/5}(0))$
and $\psi^{(l)}(\overline{D_{r}(0)})\subset\psi_{n}^{(l)}(D_{6r/5}(0))$.
Hence the harmonic function $\mathscr{G}_{n}=\mathscr{G}_{n}^{(l)}\circ(\psi_{n}^{(l)})^{-1}:\psi^{(l)}(\overline{D_{r}(0)})\rightarrow\mathbb{R}$
converges in $C^{\infty}(\psi^{(l)}(\overline{D_{r}(0)}))$ to a harmonic
function $\tilde{\mathscr{G}}^{(l)}:=\mathscr{G}^{(l)}\circ(\psi^{(l)})^{-1}:\psi^{(l)}(\overline{D_{r}(0)})\rightarrow\mathbb{R}$.
Obviously, if $l,l'\in\{1,...,N\}$ are such that $\psi^{(l)}(D_{r}(0))\cap\psi^{(l')}(D_{r}(0))\not=\emptyset$,
the uniqueness of the limit yields $\tilde{\mathscr{G}}^{(l)}|_{\psi^{(l)}(D_{r}(0))\cap\psi^{(l')}(D_{r}(0))}=\tilde{\mathscr{G}}^{(l')}|_{\psi^{(l)}(D_{r}(0))\cap\psi^{(l')}(D_{r}(0))}$.
Hence all $\tilde{\mathscr{G}}^{(l)}$ glue together to a harmonic
function defined in a neighbourhood of $\mathcal{K}$.
\end{singlespace}
\end{proof}
\begin{singlespace}
\noindent By Lemma \ref{lem:For-every-compact} it is apparent that
after going over to a diagonal subsequence, $\mathscr{G}_{n}$ converges
in $C_{\text{loc}}^{\infty}(Z\backslash\Delta)$ to a harmonic function
$\mathscr{G}:Z\backslash\Delta\rightarrow\mathbb{R}$ with respect
to $j$. This shows that the $\mathcal{H}-$holomorphic curve $u_{n}\circ\varphi_{n}^{-1}|_{Z^{(\nu)}}:Z^{(\nu)}\rightarrow\mathbb{R}\times M$
with harmonic perturbation $d\mathscr{G}_{n}$ converges in $C_{\text{loc}}^{\infty}$
on $Z^{(\nu)}\backslash\Delta_{n}$ to a $\mathcal{H}-$holomorphic
curve $u^{(\nu)}:(Z^{(\nu)},j)\rightarrow\mathbb{R}\times M$ with
harmonic perturbation $d\mathcal{G}$, where $u^{(\nu)}=\overline{u}^{(\nu)}-\mathscr{G}$
for all $\nu$. What is left is the description of the convergence
of the $\mathcal{H}-$holomorphic curves $u_{n}\circ\varphi_{n}^{-1}$
with harmonic perturbation $d\mathscr{G}_{n}$ in a neighbourhood
of the loops from $\Delta_{n}$, i.e. accross the nodes from $\Delta_{n}$.
Observe that, from (\ref{eq:uniform_bdd}), $\mathscr{G}_{n}$ is
uniformly bounded on $Z$ by the constant $C_{1}$ and the $L^{2}-$norm
of $d\mathscr{G}_{n}$ is uniformly bounded by the constant $C_{0}$.
A neighbourhood $C_{n}$ of a loop in $\Delta_{n}$ can be biholomorphically
parametrized as $[-r_{n},r_{n}]\times S^{1}$ by the biholomorphism
$\psi_{n}:[-r_{n},r_{n}]\times S^{1}\rightarrow C_{n}$, where $r_{n}\rightarrow\infty$
as $n\rightarrow\infty$. From the $C^{0}$ bound of $\mathscr{G}_{n}$
on $Z$, the maps $u_{n}\circ\varphi_{n}^{-1}$ are uniformly bounded
in $C^{0}$ on $C_{n}$ (maybe after some shift in the $\mathbb{R}-$coordinaten).
Thus we consider the $\mathcal{H}-$holomorphic cylinder $u_{n}\circ\varphi_{n}^{-1}\circ\psi_{n}$
with harmonic perturbation $\psi_{n}^{*}d\mathscr{G}_{n}$ defined
on $[-r_{n},r_{n}]\times S^{1}$. Note that the energy of $u_{n}\circ\varphi_{n}^{-1}\circ\psi_{n}$
is uniformly bounded by the constant $E_{0}$. As in Section \ref{sec:The-Thin-Part}
we divide the cylinder $[-r_{n},r_{n}]\times S^{1}$ into cylinders
of type $\infty$ with an energy less than $\hbar_{0}/2$ and cylinders
of type $b_{1}$. We apply the result of Section \ref{subsec:Cylinders-of-type}
to cylinders of type $\infty$. Keep in mind that according to Remark
\ref{rem:The-special-circles}), conditions C1-C9 are satisfied. For
cylinders of type $b_{1}$, the maps $u_{n}\circ\varphi_{n}^{-1}\circ\psi_{n}$,
after a specific shift in the $\mathbb{R}-$coordinate, are contained
in a compact subset of $\mathbb{R}\times M$. By the usual bubbling-off
analysis and the maximum principle, these maps together with the harmonic
perturbation converge in $C^{\infty}$ on cylinders of type $b_{1}$.

\noindent We ``glue'' the convergence result for the $\infty$-type
subcylinders of \textbf{$[-r_{n},r_{n}]\times S^{1}$ }introduced
in Section \ref{subsec:Cylinders-of-type} together with the $C^{\infty}-$convergence
result for the cylinders of type $b_{1}$. This process is similar
to that described in Section \ref{subsec:Glueing-cylinders-of}. However,
we are faced with a more simple situation because we can choose the
surface $C^{(m)}$ as the cylinder $[-1-4h,1+4h]\times S^{1}$, where
$h$ is the constant defined in Section \ref{subsec:Cylinders}. By
the method described in Section \ref{subsec:Glueing-cylinders-of},
the diffeomorphism $\varphi_{n}$, the surface $Z$, and the set of
nodes $\Delta_{n}$ having properties H1-H4 are replaced by some modified
versions, which are still denoted by $\varphi_{n}$, $Z$ and $\Delta_{n}$.
More precisely, the loops from $\Delta_{n}$ are the center loops
that correspond to subcylinders of type $\infty$ in the decomposition
of the components from the thick part which are conformal equivalent
to $[-r_{n},r_{n}]\times S^{1}$, where $r_{n}\rightarrow\infty$
as $n\rightarrow\infty$. As in the convergence description (see Section
\ref{subsec:Convergence}), we choose the nodal special cylinders
$A^{\text{nod}}$ around the elements from $\Delta_{n}$.
\end{singlespace}
\begin{rem}
\begin{singlespace}
\noindent Around a puncture from $Z^{(\nu)}$, the $\mathcal{H}-$holomorphic
curve $u_{n}\circ\varphi_{n}$ is asymptotic to a trivial cylinder
over a Reeb orbit (see Section \ref{sec:Harmonic-pertubed-pseudoholomorp}).
This result is a consequence of the uniform $C^{0}-$bound of the
harmonic functions $\mathscr{G}_{n}$.
\end{singlespace}
\end{rem}
\begin{singlespace}
\noindent We are now in the position to formulate the convergence
result for the $\mathcal{H}-$holomorphic curves $u_{n}$ with harmonic
perturbation $\gamma_{n}$ defined on the disk $D$. There exist the
diffeomorphisms $\varphi_{n}:D\rightarrow Z$ such that the following
hold:
\end{singlespace}
\begin{description}
\begin{singlespace}
\item [{I1}] \noindent $i_{n}\rightarrow j$ in $C_{\text{loc}}^{\infty}$
on $Z\backslash\Delta_{p}\amalg A^{\text{nod}}$.
\item [{I2}] \noindent For every special cylinder $A_{ij}$ of $Z$ there
exists an annulus $\overline{A}_{ij}\cong[-1,1]\times S^{1}$ such
that $A_{ij}\subset\overline{A}_{ij}$ and $(\overline{A}_{ij},i_{n})$
and $(A_{ij},i_{n})$ are conformally equivalent to $([-R_{n},R_{n}]\times S^{1},i)$
and $([-R_{n}+h_{n},R_{n}-h_{n}]\times S^{1},i)$, respectively, where
$R_{n}-h_{n},h_{n}\rightarrow\infty$ as $n\rightarrow\infty$, $i$
is the standard complex structure and the diffeomorphisms are of the
form $(s,t)\mapsto(\kappa(s),t)$.
\item [{I3}] \noindent The sequence of $\mathcal{H}-$holomorphic curves
$(D,i,u_{n},\gamma_{n})$ with boundary converges to a stratified
$\mathcal{H}-$holomorphic building $(Z,j,u,\mathcal{P},D,\gamma)$
in the sense of Definition \ref{def:A-sequence-of-1} from Section
\ref{subsec:Convergence}. Moreover, the curves converge in $C^{\infty}$
in a neighbourhood of the boundary $\partial D$.
\end{singlespace}
\end{description}
\begin{singlespace}
\noindent This convergence result can be applied to disks such as
neighbourhoods of all points of $\mathcal{Z}^{(m)}$. To deal with
the entire cylinder of type $b_{1}$, we glue the obtained convergence
result on disks centered at points of $\mathcal{Z}^{(m)}$ to the
complement of\textbf{ }disk neighbourhoods of $\mathcal{Z}^{(m)}$.
During the convergence description of the $\mathcal{H}-$holomorphic
curves $u_{n}$ restricted to disk neighbourhoods of the points of
$\mathcal{Z}^{(m)}$, the diffeomorphism $\varphi_{n}$, describing
the convergence, have the property that in a neighbourhood of $\partial D$
they are independent of $n$ (see Appendix \ref{sec:Pseudoholomorphic-disks-with}).
Coming back to the puncture $z\in\mathcal{Z}^{(m)}$ we focus on the
neighbourhood $D_{r}(z)$. Considering the translation and streching
diffeomorphism $D\rightarrow D_{r}(z)$, $p\mapsto z+rp$, we see
that $\varphi_{n}:D_{r}(z)\backslash D_{r\tau}(z)\hookrightarrow Z$
is independent of $n$; hereafter, we drop the index $n$ and denote
it by $\varphi:D_{r}(z)\backslash D_{r\tau}(z)\hookrightarrow Z$.
This map is used to glue $Z$ and $([0,H^{(m)}]\times S^{1})\backslash D_{r\tau}(z)$
along the collar $D_{r}(z)\backslash D_{r\tau}(z)$. Consider the
surface
\[
C^{(m)}=\left.(([0,H^{(m)}]\times S^{1})\backslash D_{r\tau}(z))\coprod Z\right/\sim,
\]
where $x\sim y$ if and only if $x\in D_{r}(z)\backslash D_{r\tau}(z)$,
$y\in\varphi(D_{r}(z)\backslash D_{r\tau}(z))$ and $\varphi(x)=y$.
This gives rise to the diffeomorphism $\psi_{n}^{(m)}:[0,H_{n}^{(m)}]\times S^{1}\rightarrow C^{(m)}$,
defined by
\[
\psi_{n}^{(m)}(x)=\begin{cases}
x, & x\in C^{(m)}\backslash D_{r}(z)\\
\varphi_{n}(x), & x\in D_{r}(z).
\end{cases}
\]
We are now able to describe the convergence on cylinders of type $b_{1}$.
Let $\Delta_{n}$, $\Delta_{p}$ and $A^{\text{nod}}$ be the collection
of loops from $C^{(m)}$ obtained by the above convergence process
for each point of $\mathcal{Z}^{(m)}$. Take notice that the complex
structure $j^{(m)}$ on $C^{(m)}$ is given by
\[
j^{(m)}(p):=\begin{cases}
i, & p\in C^{(m)}\backslash D_{r}(\mathcal{Z}^{(m)})\\
j, & p\in Z
\end{cases}
\]
and that it is well-defined since $\varphi$ is a biholomorphism.
There exists a sequence of diffeomorphisms $\psi_{n}^{(m)}:[0,H_{n}^{(m)}]\times S^{1}\rightarrow C^{(m)}$
such that the following hold:
\end{singlespace}
\begin{description}
\begin{singlespace}
\item [{J1}] \noindent $(\psi_{n}^{(m)})_{*}i\rightarrow j^{(m)}$ in $C_{\text{loc}}^{\infty}$
on $C^{(m)}\backslash\Delta_{p}\amalg A^{\text{nod}}$.
\item [{J2}] \noindent For every special cylinder $A_{ij}$ of $C^{(m)}$
there exists an annulus $\overline{A}_{ij}\cong[-1,1]\times S^{1}$
such that $A_{ij}\subset\overline{A}_{ij}$ and $(\overline{A}_{ij},i_{n})$
and $(A_{ij},i_{n})$ are conformally equivalent to $([-R_{n},R_{n}]\times S^{1},i)$
and $([-R_{n}+h_{n},R_{n}-h_{n}]\times S^{1},i)$, respectively, where
$R_{n},R_{n}-h_{n}\rightarrow\infty$ as $n\rightarrow\infty$, $i$
is the standard complex structure and the diffeomorphisms are of the
form $(s,t)\mapsto(\kappa(s),t)$.
\item [{J3}] \noindent The $\mathcal{H}-$holomorphic curves $([0,H_{n}^{(m)}]\times S^{1},i,u_{n},\gamma_{n})$
with boundary converges to a stratified broken $\mathcal{H}-$holomorphic
building $(C^{(m)},j,u,\mathcal{P},\mathcal{D},\gamma)$ with boundary
in the sense of Definition \ref{def:A-sequence-of-1} from Section
\ref{subsec:Convergence}. By construction, the curves converge in
$C^{\infty}$ in a neighbourhood of the boundary $\partial D$.
\end{singlespace}
\end{description}
\begin{singlespace}

\subsubsection{\label{subsec:Glueing-cylinders-of}Glueing cylinders of type $\infty$
with cylinders of type $b_{1}$}
\end{singlespace}

\begin{singlespace}
\noindent By a modified version of the diffeomorphisms $\theta_{n}$
we identify the cylinders of type $\infty$ with the cylinder $[-1-4h,1+4h]\times S^{1}$
where $h>0$ is the constant from Lemma \ref{lem:Let-=000024=00005Bh_=00007Bn=00007D^=00007B(m-1)=00007D,h_=00007Bn=00007D^=00007B(m)=00007D=00005D},
so that after the glueing process, we end up with a bigger cylinder
of finite length and a sequence of diffeomorphisms. Let us make this
precedure more precise.

\noindent Let $[h_{n}^{(m-1)},h_{n}^{(m)}]\times S^{1}$ and $[h_{n}^{(m)}-4h,h_{n}^{(m+1)}+4h]\times S^{1}$
be cylinders of types $\infty$ and $b_{1}$, respectively. First
we consider the cylinders $[h_{n}^{(m-1)},h_{n}^{(m)}]\times S^{1}$
of type $\infty$. With the constant $h>0$ defined in Section \ref{subsec:Cylinders},
let $[h_{n}^{(m-1)}+4h,h_{n}^{(m)}-4h]\times S^{1}$ be a subcylinder.
By the uniform gradient bounds of $u_{n}$ on cylinders of type $\infty$,
we conclude that the $\mathcal{H}-$holomorphic curves $u_{n}$ together
with the harmonic perturbations $\gamma_{n}$ converge in $C^{\infty}$
on $[h_{n}^{(m)}-4h,h_{n}^{(m)}]\times S^{1}$ to a $\mathcal{H}-$holomorphic
curve $u$ with harmonic perturbation $\gamma$. For the subcylinders
$[h_{n}^{(m-1)}+4h,h_{n}^{(m)}-4h]\times S^{1}$ we perform the same
analysis as in Theorems \ref{thm:There-exists-a} and \ref{thm:There-exist-the}.
After going over to a subsequence we obtain a sequence of diffeomorphisms
\[
\theta_{n}:[h_{n}^{(m-1)}+4h,h_{n}^{(m)}-4h]\times S^{1}\rightarrow[-1,1]\times S^{1},
\]
so that Theorems \ref{thm:There-exists-a} and \ref{thm:There-exist-the}
hold for the cylinders $[h_{n}^{(m-1)}+4h,h_{n}^{(m)}-4h]\times S^{1}$.
Next we extend the diffeomorphisms $\theta_{n}$, $\theta^{-}$ and
$\theta^{+}$ to $[h_{n}^{(m-1)},h_{n}^{(m)}]\times S^{1}$, $[-4h,\infty)\times S^{1}$
and $(-\infty,4h]\times S^{1}$, respectively, such that 
\begin{align*}
\theta_{n}|_{([h_{n}^{(m-1)},h_{n}^{(m-1)}+3h]\times S^{1})\amalg([h_{n}^{(m)}-3h,h_{n}^{(m)}]\times S^{1})} & =\text{id},\\
\theta^{-}|_{[-4h,-h]\times S^{1}} & =\text{id},\\
\theta^{+}|_{[h,4h]\times S^{1}} & =\text{id},
\end{align*}
$\theta_{n}^{-}\rightarrow\theta^{-}$ in $C^{\infty}([-4h,-h]\times S^{1})$,
and $\theta_{n}^{+}\rightarrow\theta^{+}$in $C^{\infty}([h,4h]\times S^{1})$.

\noindent We consider now the cylinders of type $b_{1}$ and note
that the diffomorphisms 
\[
\psi_{n}:[h_{n}^{(m)}-4h,h_{n}^{(m+1)}+4h]\times S^{1}\rightarrow C^{(m)}
\]
have the property that
\[
\psi_{n}|_{([h_{n}^{(m)}-4h,h_{n}^{(m)}-h]\times S^{1})\amalg([h_{n}^{(m+1)}-3h,h_{n}^{(m+1)}]\times S^{1})}=\text{id}.
\]
In this regard we consider the surface\[
\bigslant{\left ( ([-1-4h,1+4h]\times S^{1}) \amalg C^{(m)} \right )}{\sim}
\]where $x\sim y$ if and only if $x\in[h,4h]\times S^{1}$ and $y\in[h_{n}^{(m)}-4h,h_{n}^{(m)}-h]\times S^{1}$
such that $\theta_{n}(y)=x$.

\noindent By this procedure we glue all cylinders of types $\infty$
and $b_{1}$, and obtain a bigger cylinder $C_{n}$ together with
a sequence of diffeomorphisms $\Phi_{n}:[-\sigma_{n},\sigma_{n}]\times S^{1}\rightarrow C_{n}$,
where $[-\sigma_{n},\sigma_{n}]\times S^{1}$ is the parametrization
of the $\delta-$thin part, i.e. of $\text{Thin}_{\delta}(\dot{S}^{\mathcal{D},r},h_{n})$.
Let $\varphi_{n}:\mathcal{C}_{n}\rightarrow[-\sigma_{n},\sigma_{n}]\times S^{1}$
be the conformal parametrization of the cylindrical component of $\text{Thin}_{\epsilon}(\dot{S}^{\mathcal{D},r},h_{n})$.
Since both ends of $[-\sigma_{n},\sigma_{n}]\times S^{1}$ contain
cylinders of type $\infty$, we infer by the above construction, that
$\Phi_{n}$ is identity near the boundary. Specifically, with the
constant $h>0$ we have
\[
\text{\ensuremath{\Phi}}_{n}|_{([-\sigma_{n},-\sigma_{n}+3h]\times S^{1})\amalg([\sigma_{n}-3h,\sigma_{n}]\times S^{1})}=\text{id}.
\]
Then we consider the surface\[
\bigslant{\left ( \left ( \dot{S}^{D,r} \backslash \varphi_{n}^{-1}([-\sigma_{n} + 3h, \sigma_{n} - 3h] \times S^{1}) \right ) \amalg C_{n} \right )}{\sim}
\]where $x\sim y$ if and only if $x\in\dot{S}^{\mathcal{D},r}\backslash\varphi_{n}^{-1}([-\sigma_{n}+3h,\sigma_{n}-3h]\times S^{1})$
and $y\in C_{n}$ such that $\Phi_{n}\circ\varphi_{n}(x)=y$. 

\noindent In this way we handle all components of $\text{Thin}_{\delta}(\dot{S}^{\mathcal{D},r},h_{n})$
that are conformal equivalent to hyperbolic cylinders.
\end{singlespace}
\begin{singlespace}

\subsubsection{\label{subsec:Punctures}Punctures and elements of $\mathcal{Z}$}
\end{singlespace}

\begin{singlespace}
\noindent We analyze the convergence of $u_{n}$ on components of
the thin part which are biholomorphic to cusps, as well as, in a neighbourhood
of the points from $\mathcal{Z}$. Recall that cusps correspond to
neighbourhoods of punctures. Let $p\in S^{\mathcal{D},r}$ be a puncture
or an element from $\mathcal{Z}$. By Lemma \ref{2_1_lem:There-exists-open-neighbourhood-1}
of Appendix \ref{chap:Special-coordinates}, there exist the open
neighbourhoods $U_{n}$ and $U$ of $p$, and the biholomorphisms
$\psi_{n}:D\rightarrow U_{n}$ and $\psi:D\rightarrow U$ such that
$\psi_{n}$ converge in $C^{\infty}$ to $\psi$. We consider the
sequence of $\mathcal{H}-$holomorphic curves $u_{n}$ with harmonic
perturbations $\gamma_{n}$ restricted to $U_{n}$. By the convergence
of $u_{n}$ on the thick part, for every open neighbourhoods $U$
and $V$ of $p$, such that $V\Subset U$, the $\mathcal{H}-$holomorphic
curves $u_{n}$ together with the harmonic perturbations $\gamma_{n}$
converge in $C^{\infty}$ on $\overline{U\backslash V}$ to some $\mathcal{H}-$holomorphic
curve $u$ with harmonic perturbation $\gamma$. Via the biholomorphisms
$\psi_{n}$ and $\psi$, we consider the $\mathcal{H}-$holomorphic
curves $u_{n}$ and the harmonic perturbations $\gamma_{n}$ as beeing
defined on $D\backslash\{0\}$. Actually, we consider the following
setup: For the sequence of $\mathcal{H}-$holomorphic curves $u_{n}=(a_{n},f_{n}):D\backslash\{0\}\rightarrow\mathbb{R}\times M$
with the harmonic perturbations $\gamma_{n}$ defined on the whole
disk $D$, the following are satisfied:
\end{singlespace}
\begin{description}
\begin{singlespace}
\item [{K1}] \noindent The energy of $u_{n}$ is uniformly bounded, i.e.
with the constant $E_{0}>0$ we have $E(u_{n};D\backslash\{0\})\leq E_{0}$
for all $n\in\mathbb{N}$.
\item [{K2}] \noindent The $L^{2}-$norms of $\gamma_{n}$ are uniformly
bounded, i.e. with the constant $C_{0}>0$ we have $\left\Vert \gamma_{n}\right\Vert _{L^{2}(D\backslash\{0\})}^{2}\leq C_{0}$
for all $n\in\mathbb{N}$.
\item [{K3}] \noindent For every open neighbourhoods $U$ and $V$ of $p$
such that $V\Subset U$, the $\mathcal{H}-$holomorphic curves $u_{n}$
with harmonic perturbations $\gamma_{n}$ converge in $C^{\infty}$
on $\overline{U\backslash V}$ to a $\mathcal{H}-$holomorphic curve
$u$ with harmonic perturbation $\gamma$.
\end{singlespace}
\end{description}
\begin{singlespace}
\noindent We consider two cases. In the first case there exists a
subsequence of $u_{n}$ for which the singularity at $0$ is removable,
i.e. the $\mathbb{R}-$coordinate $a_{n}$ is bounded in a neighbourhood
of $0$, but not necessarily uniformly bounded. In particular, this
case is typically for neighbourhoods of points from $\mathcal{Z}$.
Hence the sequence of $\mathcal{H}-$holomorphic curves $u_{n}$ can
be defined accross the puncture $0$ and we end up with a sequence
of $\mathcal{H}-$holomorphic disks with fixed boundary. To describe
the compactness we use the results of Section \ref{subsec:Cylinders-of-type-1}. 

\noindent In the second case, there exists no subsequence of the $u_{n}$
that has a bounded $\mathbb{R}-$coordinate $a_{n}$ near $0$. Since
$D$ is simply connected, there exists a harmonic function $\tilde{\Gamma}_{n}:D\rightarrow\mathbb{R}$
such that $\gamma_{n}=d\tilde{\Gamma}_{n}$. By the second condition
from above, the gradients $\nabla\tilde{\Gamma}_{n}$ are uniformly
bounded in $L^{2}-$norm by the constant $C_{0}>0$. Denote by
\[
K_{n}=\frac{1}{\pi}\int_{D}\tilde{\Gamma}_{n}(x,y)dxdy
\]
the mean value of $\tilde{\Gamma}_{n}$ on the disk $D$. Furthermore,
define $\Gamma_{n}:=\tilde{\Gamma}_{n}-K_{n}$; $\Gamma_{n}$ is a
harmonic function on the disk with vanishing average and satisfying
$\gamma_{n}=d\Gamma_{n}$, while the gradients $\nabla\Gamma_{n}$
have uniformly bounded $L^{2}-$norms. By Poincaré inequality, the
$L^{2}-$norm of $\Gamma_{n}$ is uniformly bounded, i.e. with the
constant $C_{0}>0$ we have $\left\Vert \Gamma_{n}\right\Vert _{L^{2}(D)}\leq C_{0}$
for all $n\in\mathbb{N}$. Pick $\tau\in(0,1)$ and denote by $D_{\tau}$
the disk around $0$ of radius $\tau$. From the mean value inequality
for harmonic functions, $\Gamma_{n}$ is uniformly bounded in $C^{0}(D_{\tau})$.
Via the biholomorphism $[0,\infty)\times S^{1}\rightarrow D\backslash\{0\}$,
$(s,t)\mapsto e^{-2\pi(s+it)}$ we consider the $\mathcal{H}-$holomorphic
maps $u_{n}$ together with the harmonic perturbations $\gamma_{n}$
as beeing defined on the half open cylinder $[0,\infty)\times S^{1}$.
Specifically we consider the following setup: For the sequence $u_{n}=(a_{n},f_{n}):[0,\infty)\times S^{1}\rightarrow\mathbb{R}\times M$
of $\mathcal{H}-$holomorphic half cylinders with harmonic perturbations
$\gamma_{n}$ the following are satisfied:
\end{singlespace}
\begin{description}
\begin{singlespace}
\item [{L1}] \noindent The energy of $u_{n}$ and the $L^{2}-$norm of
the harmonic pertuabtions $\gamma_{n}$ are uniformly bounded, i.e.
with the constants $E_{0},C_{0}>0$ we have $E(u_{n};[0,\infty)\times S^{1})\leq E_{0}$
and $\left\Vert \gamma_{n}\right\Vert _{L^{2}(D\backslash\{0\})}^{2}\leq C_{0}$
for all $n\in\mathbb{N}$.
\item [{L3}] \noindent The $\mathcal{H}-$holomorphic curves $u_{n}$ converge
in $C_{\text{loc}}^{\infty}$ to a $\mathcal{H}-$holomorphic curve
$u$ with harmonic perturbation $\gamma$.
\item [{L4}] \noindent The harmonic perturbations $\gamma_{n}$ satisfy
$\gamma_{n}=d\Gamma_{n}$, where $\Gamma_{n}:[0,\infty)\times S^{1}\rightarrow\mathbb{R}$
is a harmonic function with a uniformly bounded gradient $\nabla\Gamma_{n}$
in $L^{2}-$norm. Furthermore, $\Gamma_{n}$ is uniformly bounded
in $C^{0}([0,\infty)\times S^{1})$. 
\end{singlespace}
\end{description}
\begin{singlespace}
\noindent By using the decomposition discussed in Section \ref{subsec:Cylinders}
we split the half cylinder into smaller cylinders with $d\alpha-$energies
smaller than $\hbar_{0}/2$. As described in Section \ref{subsec:Cylinders}
we end up with a sequence of finitely many cylinder of types $\infty$
and $b_{1}$, and a half cylinder with a small $d\alpha-$energy.
The appearance of the cylinders of types $b_{1}$ and $\infty$ is
alternating; the decomposition starts with a cylinder of type $\infty$
and ends with a cylinder of type $b_{1}$ followed by the half cylinder
(see Figure \ref{fig19}). 

\noindent \begin{figure}     
	\centering  
	\def\svgscale{1} 	
	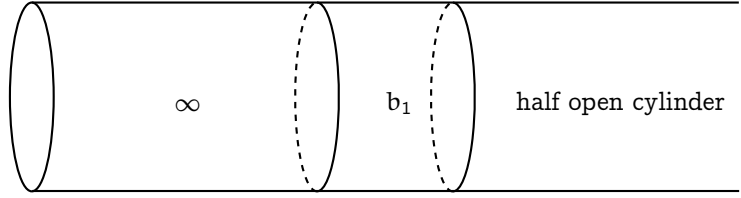  
	\caption{Decomposition of a punctured neighbourhood into cylinders of type $\infty$, $b_{1}$ and a half open cylinder.}
	\label{fig19}
\end{figure}For the cylinders of types $\infty$ and $b_{1}$ we formulate the
convergence results as in Sections \ref{subsec:Cylinders-of-type-1}
and \ref{subsec:Cylinders-of-type}. Since the harmonic $1-$forms
$\gamma_{n}$ are defined over the puncture $p$, the period of the
harmonic perturbation $\gamma_{n}$ over each cylinder (either of
type $\infty$ or type $b_{1}$) is $0$. Hence, the converge properties
of the cylinders of type $\infty$ are the same as in the classical
theory of Hofer (see \cite{key-11}), and we are left with the half
cylinder having a $d\alpha-$energy smaller than $\hbar_{0}/2$. We
have the following setup:
\end{singlespace}
\begin{description}
\begin{singlespace}
\item [{M1}] \noindent $u_{n}=(a_{n},f_{n}):[0,\infty)\times S^{1}\rightarrow\mathbb{R}\times M$
is a $\mathcal{H}-$holomorphic curve with harmonic perturbation $\gamma_{n}$.
\item [{M2}] \noindent The energy of $u_{n}$ and the $L^{2}-$norm of
$\gamma_{n}$ are uniformly bounded by the constants $E_{0}$ and
$C_{0}$, respectively, while the $d\alpha-$energy of $u_{n}$ is
smaller than $\hbar_{0}/2$.
\item [{M3}] \noindent The harmonic perturbations $\gamma_{n}$ satisfy
$\gamma_{n}=d\Gamma_{n}$, where $\Gamma_{n}:[0,\infty)\times S^{1}\rightarrow\mathbb{R}$
is a harmonic function with a uniformly bounded gradient $\nabla\Gamma_{n}$
in $L^{2}-$norm. Furthermore, $\Gamma_{n}$ is uniformly bounded
in $C^{0}([0,\infty)\times S^{1})$. 
\item [{M4}] \noindent The gradients of $u_{n}$ are uniformly bounded,
i.e. there exists a constant $\tilde{C}>0$ such that
\begin{equation}
\left\Vert du_{n}(z)\right\Vert =\sup_{\left\Vert v\right\Vert _{\text{eucl.}}=1}\left\Vert du_{n}(z)(v)\right\Vert _{\overline{g}}\leq\tilde{C}\label{eq:bdd_gradient}
\end{equation}
for all $z\in[0,\infty)\times S^{1}$ and all $n\in\mathbb{N}$.
\end{singlespace}
\end{description}
\begin{singlespace}
\noindent By bubbling-off analysis and in view of the uniformly small
$d\alpha-$energy, Assumption (\ref{eq:bdd_gradient}) is also valid.
Moreover, by the mean value thorem for harmonic functions and the
uniformly boundedness of the $L^{2}-$norms of $\nabla\Gamma_{n}$,
the harmonic perturbation $\gamma_{n}$ is uniformly bounded in $C^{0}$
on $[0,\infty)\times S^{1}$ with respect to the standard Euclidean
metric. We turn the $\mathcal{H}-$holomorphic curve $u_{n}$ with
harmonic perturbation $\gamma_{n}$ into a usual pseudoholomorphic
curve $\overline{u}_{n}$ by setting $\overline{u}_{n}=(\overline{a}_{n},\overline{f}_{n})=(a_{n}+\Gamma_{n},f_{n})$
as in Section \ref{subsec:Cylinders-of-type-1}. In the following
we show that the $\alpha-$ and $d\alpha-$ energies of $\overline{u}_{n}$
are uniformly bounded. As $\overline{f}_{n}=f_{n}$ we have
\[
E_{d\alpha}(\overline{u}_{n};[0,\infty)\times S^{1})=E_{d\alpha}(u_{n};[0,\infty)\times S^{1})\leq\frac{\hbar_{0}}{2}
\]
and therefore the $d\alpha-$energy is uniformly small. By definition
and accounting on the uniform bound on the gradients (\ref{eq:bdd_gradient})
and the uniform $C^{0}-$bound of the harmonic $1-$forms $\gamma_{n}$,
we obtain
\begin{align*}
E_{\alpha}(\overline{u}_{n};[0,\infty)\times S^{1}) & \leq-\sup_{\varphi\in\mathcal{A}}\int_{[0,\infty)\times S^{1}}d(\varphi(\overline{a}_{n})d\overline{a}_{n}\circ i)+E_{d\alpha}(u_{n})\\
 & =-\sup_{\varphi\in\mathcal{A}}\left[\lim_{r\rightarrow\infty}\int_{\{r\}\times S^{1}}\varphi(\overline{a}_{n})d\overline{a}_{n}\circ i-\int_{\{0\}\times S^{1}}\varphi(\overline{a}_{n})d\overline{a}_{n}\circ i\right]+E_{d\alpha}(u_{n})\\
 & \leq2C+\frac{\hbar_{0}}{2}.
\end{align*}
Thus the $\alpha-$energy is uniformly bounded. From the definition
of $\tilde{E}_{0}$ (see Section \ref{subsec:Cylinders}) we have
$E(\overline{u}_{n};[0,\infty)\times S^{1})\leq\tilde{E}_{0}$ for
all $n\in\mathbb{N}$. In this regard, we consider the following setup:
\end{singlespace}
\begin{description}
\begin{singlespace}
\item [{N1}] \noindent $\overline{u}_{n}=(\overline{a}_{n},f_{n}):[0,\infty)\times S^{1}\rightarrow\mathbb{R}\times M$
is a pseudoholomorphic curve.
\item [{N2}] \noindent The energy of $\overline{u}_{n}$ is uniformly bounded,
while the $d\alpha-$energy of $\overline{u}_{n}$ is uniformly smaller
than $\hbar_{0}/2$.
\end{singlespace}
\end{description}
\begin{singlespace}
\noindent Using the diffeomorphism $\theta$ defined above together
with the notation (\ref{eq:C0-conv}), and employing Theorem \ref{thm:Let--be}
of Appendix \ref{chap:Half-cylinders-with} we have the following
\end{singlespace}
\begin{thm}
\begin{singlespace}
\noindent \label{thm:Let--be-1}There exists a subsequence of $\overline{u}_{n}$,
still denoted by $\overline{u}_{n}$ such that the following is satisfied.
\end{singlespace}
\begin{enumerate}
\begin{singlespace}
\item $\overline{u}_{n}$ is asysmptotic to the same Reeb orbit, i.e. there
exists a Reeb orbit $x$ of period $T\not=0$ with $|T|\leq\tilde{E}_{0}$
and a sequence $c_{n}\in S^{1}$ such that
\[
\lim_{s\rightarrow\infty}\overline{f}_{n}(s,t)=x(T(t+c_{n}))\ \text{ and }\ \lim_{s\rightarrow\infty}\frac{\overline{a}_{n}(s,t)}{s}=T
\]
for all $n\in\mathbb{N}$.
\item $\overline{u}_{n}$ converge in $C_{\text{loc}}^{\infty}$ to a pseudoholomorphic
half cylinder $\overline{u}:[0,\infty)\times S^{1}\rightarrow\mathbb{R}\times M$
having a bounded energy and a $d\alpha-$energy smaller than $\hbar_{0}/2$.
Moreover, there exists $c^{*}\in S^{1}$ such that 
\[
\lim_{s\rightarrow\infty}\overline{f}(s,t)=x(T(t+c^{*}))\ \text{ and }\ \lim_{s\rightarrow\infty}\frac{\overline{a}(s,t)}{s}=T.
\]
\item The maps $g_{n}=\overline{f}_{n}\circ\theta^{-1}:[0,1]\times S^{1}\rightarrow M$,
where $g_{n}(1,t)=x(T(t+c_{n}))$ converge in $C^{0}$ to a map $g:[0,1]\times S^{1}\rightarrow M$,
that satisfy $g(1,t)=x(T(t+c^{*}))$, where $x$ is a Reeb orbit of
period $T\not=0$ from part 1.
\end{singlespace}
\end{enumerate}
\end{thm}
\begin{singlespace}
\noindent With this result we are in the position to formulate the
convergence of the sequence of $\mathcal{H}-$holomorphic half cylinders
$u_{n}$ with harmonic perturbations $\gamma_{n}$.
\end{singlespace}
\begin{thm}
\begin{singlespace}
\noindent There exists a subsequence $u_{n}$ still denoted by $u_{n}$
such that the following is satisfied.
\end{singlespace}
\begin{enumerate}
\begin{singlespace}
\item $u_{n}$ is asysmptotic to the same Reeb orbit, i.e. there exists
a Reeb orbit $x$ of period $T\not=0$ with $|T|\leq\tilde{E}_{0}$
and a sequence $c_{n}\in S^{1}$ such that
\[
\lim_{s\rightarrow\infty}f_{n}(s,t)=x(T(t+c_{n}))\ \text{ and }\ \lim_{s\rightarrow\infty}\frac{a_{n}(s,t)}{s}=T
\]
for all $n\in\mathbb{N}$.
\item $u_{n}$ converge in $C_{\text{loc}}^{\infty}$ to a $\mathcal{H}-$holomorphic
half cylinder $u:[0,\infty)\times S^{1}\rightarrow\mathbb{R}\times M$
with harmonic perturbation $\gamma$ having a bounded energy and a
$d\alpha-$energy smaller than $\hbar_{0}/2$. Moreover, there exists
$c^{*}\in S^{1}$ such that 
\[
\lim_{s\rightarrow\infty}f(s,t)=x(T(t+c^{*}))\ \text{ and }\ \lim_{s\rightarrow\infty}\frac{a(s,t)}{s}=T.
\]
\item The maps $g_{n}=f_{n}\circ\theta^{-1}:[0,1]\times S^{1}\rightarrow M$,
where $g_{n}(1,t)=x(T(t+c_{n}))$ converge in $C^{0}$ to a map $g:[0,1]\times S^{1}\rightarrow M$,
and satisfy $g(1,t)=x(T(t+c^{*}))$, where $x$ is a Reeb orbit of
period $T\not=0$ from part 1.
\end{singlespace}
\end{enumerate}
\end{thm}
\begin{proof}
\begin{singlespace}
\noindent Since the $\Gamma_{n}$ are uniformly bounded in $C^{0}-$norm,
the first assertion is obvious. Employing the same arguments as in
\cite{key-23}, i.e. the mean value theorem for harmonic functions
and Cauchy integral formula, we deduce that $\Gamma_{n}$ have uniformly
bounded derivatives, and so, converge in $C_{\text{loc}}^{\infty}$
on $[0,\infty)\times S^{1}$ to a harmonic function $\Gamma:[0,\infty)\times S^{1}\rightarrow\mathbb{R}$
with a gradient bounded in $L^{2}-$norm. Let us show that $\Gamma:[0,\infty)\times S^{1}\rightarrow\mathbb{R}$
is bounded in $C^{0}$. Via the conformal diffomorphism $[0,\infty)\times S^{1}\rightarrow D\backslash\{0\},(s,t)\mapsto e^{-2\pi(s+it)}$
we assume that the harmonic functions $\Gamma_{n}$ and $\Gamma$
are defined on the disk $D$. Then, since the $\Gamma_{n}$ are uniformly
bounded in $C^{0}$ and have gradients with uniformly bounded $L^{2}-$norms,
it follows that $\Gamma_{n}\rightarrow\Gamma$ in $C^{\infty}(\overline{D_{\rho}(0)})$
for some $0<\rho<1$. This shows that $\Gamma$ is uniformly bounded
on $D$ and hence, via the conformal map $[0,\infty)\times S^{1}\rightarrow D\backslash\{0\}$
it is uniformly bounded on $[0,\infty)\times S^{1}$. Thus, the second
assertion is proved, and by means of $\overline{f}_{n}=f_{n}$, the
third assertion is evident.
\end{singlespace}
\end{proof}
\begin{singlespace}
\noindent By cutting a smal piece of finite length from the infinite
half cylinder, we can make the cylinder preceding the infinite half
cylinder to be of type $b_{1}$. Assuming that the infinite half cylinder
is of type $\infty$, we glue all cylinders of types $\infty$ and
$b_{1}$ together (as described in the previous section). Via the
map $[0,1)\times S^{1}\rightarrow D\backslash\{0\},(s,t)\mapsto(1-s)e^{2\pi it}$,
we identify the cylinder $[0,1)\times S^{1}$, which is diffeomorphic
with the infinite half open cylinder, with a punctured disk $D\backslash\{0\}$.
In this way the upper half open cylinder $[0,1)\times S^{1}$ can
be identified with a neighbourhood of a puncture.
\end{singlespace}
\begin{singlespace}

\section{\label{chap:Discussion-on-conformal}Discussion on conformal period}
\end{singlespace}

\begin{singlespace}
\noindent In this section we analyze Condition C8 and C9 of Section
\ref{subsec:Cylinders-of-type} dealing with the boundedness of the
sequence $R_{n}P_{n}$, and which can be regarded as a connection
between the conformal data of the Riemann surface and the harmonic
$1-$forms $\gamma_{n}$. Without this additional condition the convergence
result from \cite{key-23} can not be estabished. The reason is that
the almost complex structure constructed on the contact manifold $M$
might not vary in a compact interval. We show that this condition
is not automatically satisfied by giving a counterexample. It should
be pointed out that this example contradicts Lemma A.2 of \cite{key-15}.
Essentially, we will construct a sequence of harmonic $1-$forms $\gamma_{n}$
on a sequence of stable Riemann surfaces, that degenerate along a
single circle, have uniformly bounded $L^{2}-$norms but unbounded
$P_{n}/\ell_{n}$, where $P_{n}$ denotes the period of $\gamma_{n}$
along the degenerating circle and $\ell_{n}$ its length with respect
to the hyperbolic metric. Observe that the quantity $1/\ell_{n}$
is similar to $R_{n}$ . 

\noindent Let $(S_{n},j_{n},\mathcal{M}_{n})$ be a sequence of stable
Riemann surfaces of genus $g$, where $\mathcal{M}_{n}\subset S_{n}$
are finite sets of marked points with the same cardinality. Choose
a basis $c_{1},...,c_{2g}\in H_{1}(S_{n};\mathbb{Z})$ which is independent
of $n$. This choice is possible because all $S_{n}$ have genus $g$
and are closed (they are topologically the same). By the Deligne-Mumford
convergence,
\[
(S_{n},j_{n},\mathcal{M}_{n})\rightarrow(S,j,\mathcal{M},\mathcal{D},r),
\]
where $(S,j,\mathcal{M},\mathcal{D},r)$ is a decorated nodal Riemann
surface. Again, according to the definition of the Deligne-Mumford
convergence, there exist the diffeomorphisms $\varphi_{n}:S^{\mathcal{D},r}\rightarrow S_{n}$,
such that $j_{n}\rightarrow j$ on $S^{\mathcal{D},r}\backslash\coprod_{j=1}^{l}\Gamma_{j}$
or equivalently, $h_{n}\rightarrow h$ on $\dot{S}^{\mathcal{D},r}\backslash\coprod_{j=1}^{l}\Gamma_{j}$
where $\Gamma_{j}$ are special circles, and $h_{n}$ and $j_{n}$
are the pull-back of the complex structure and the hyperbolic metric
from $S_{n}$ and $\dot{S}_{n}$ via the diffeomorphism $\varphi_{n}$.
Assume that $l=1$, i.e. that there exists only one degenerating geodesic
in the Deligne-Mumford convergence. Denote this geodesic by $\Gamma$.
Furthermore, assume that $\Gamma=c_{1}$ ($\Gamma$ lies in the class
of $c_{1}$). The main result of this section is the following 
\end{singlespace}
\begin{prop}
\begin{singlespace}
\noindent \label{prop:There-exists-a-2}There exists a sequence of
harmonic $1-$forms $\gamma_{n}\in\mathcal{H}_{j_{n}}^{1}(S_{n})$
with uniformly bounded $L^{2}-$norms, periods, and co-periods, but
unbounded conformal periods.
\end{singlespace}
\end{prop}
\begin{proof}
\begin{singlespace}
\noindent Choose a sequence of harmonic $1-$forms $\gamma_{n}\in\mathcal{H}_{j_{n}}^{1}(S^{\mathcal{D},r})$
with vanishing periods except on $\Gamma$ (on all of $c_{i}$ with
$i\not=1$ except on $c_{1}=\Gamma$). By normalization, assume that
$\left\Vert \gamma_{n}\right\Vert _{L^{2}(S^{\mathcal{D},r})}=1$.
The uniform bounds on the $L^{2}-$norms imply that the periods $P_{n}$
of $\gamma_{n}$ over $\Gamma$ converge to $0$. Thus $\gamma_{n}$
converge in $C_{\text{loc}}^{\infty}$ to $\gamma$ on $S^{\mathcal{D},r}\backslash\Gamma$
which can be seen as a harmonic $1-$form on $S$ with vanishing periods.
By Hodge theory, we have $\gamma=0$. For $n$ sufficiently large,
the $L^{2}-$norms of $\gamma_{n}$ concentrate in the collar neighbourhood
around $\Gamma$. Indeed, from 
\[
1=\left\Vert \gamma_{n}\right\Vert _{L^{2}(S^{\mathcal{D},r})}^{2}=\left\Vert \gamma_{n}\right\Vert _{L^{2}(\mathcal{C}_{n})}^{2}+\left\Vert \gamma_{n}\right\Vert _{L^{2}(S^{\mathcal{D},r}\backslash\mathcal{C}_{n})}^{2},
\]
where $\mathcal{C}_{n}$ is the cylindrical component of the $\delta-$thin
part for some sufficiently small but fixed $\delta>0$, it follows
that $S^{\mathcal{D},r}\backslash\mathcal{C}_{n}$ is contained in
a compact subset of $S^{\mathcal{D},r}\backslash\Gamma$, and so,
that $\left\Vert \gamma_{n}\right\Vert _{L^{2}(S^{\mathcal{D},r}\backslash\mathcal{C}_{n})}^{2}$
converge to $0$, and for $n$ sufficiently large we have $\left\Vert \gamma_{n}\right\Vert _{L^{2}(\mathcal{C}_{n})}\leq1$
and $\left\Vert \gamma_{n}\right\Vert _{L^{2}(\mathcal{C}_{n})}\rightarrow1$
as $n\rightarrow\infty$. If $F_{n}$ is the unique holomorphic $1-$form
with $\text{Re}(F_{n})=\gamma_{n}$,
\[
\left\Vert F_{n}\right\Vert _{L^{2}(S^{\mathcal{D},r})}^{2}=\frac{i}{2}\int_{S^{\mathcal{D},r}}F_{n}\wedge\overline{F}_{n}.
\]
The collar $\mathcal{C}_{n}$ is conformaly equivalent to $[-R_{n},R_{n}]\times S^{1}$,
where $R_{n}\sim1/\ell_{n}$ and $\ell_{n}$ is the length of $\Gamma$
with respect to $h_{n}$. On $\mathcal{C}_{n}$ we write $\gamma_{n}=f_{n}ds+g_{n}dt$,
where $f_{n}$ and $g_{n}$ are harmonic functions on the cylinder
$[-R_{n},R_{n}]\times S^{1}$ ($s$ is the coordinate in $[-R_{n},R_{n}]$
and $t$ is the coordinate on $S^{1}$), express the holomorphic $1-$form
$F_{n}$ as $F_{n}=(f_{n}-ig_{n})dz=(f_{n}-ig_{n})(ds+idt)$, and
note that $\left\Vert F_{n}\right\Vert _{L^{2}(\mathcal{C}_{n})}=\left\Vert \gamma_{n}\right\Vert _{L^{2}(\mathcal{C}_{n})}$.
Consider the quantity $|\left\Vert F_{n}\right\Vert _{L^{2}(\mathcal{C}_{n})}-|b_{0}|\left\Vert dz\right\Vert _{L^{2}(\mathcal{C}_{n})}|$,
where $b_{0}=-\tilde{S}_{n}-iP_{n}$ and $\tilde{S}_{n}$ is the co-period
defined by
\begin{eqnarray*}
\tilde{S}_{n} & = & \int_{\Gamma}\gamma_{n}\circ j_{n}=-\int_{\{0\}\times S^{1}}f_{n}(0,t)dt.
\end{eqnarray*}
Recalling that
\begin{eqnarray*}
P_{n} & = & \int_{\Gamma}\gamma_{n}=\int_{\{0\}\times S^{1}}g(0,t)dt,
\end{eqnarray*}
we obtain
\begin{align*}
\left|\left\Vert F_{n}\right\Vert _{L^{2}(\mathcal{C}_{n})}-|b_{0}|\left\Vert dz\right\Vert _{L^{2}(\mathcal{C}_{n})}\right| & =\left\Vert \left[(f_{n}+\tilde{S}_{n})-i(g_{n}-P_{n})\right]dz\right\Vert _{L^{2}(\mathcal{C}_{n})}\\
 & \leq\left\Vert (f_{n}+\tilde{S}_{n})dz\right\Vert _{L^{2}(\mathcal{C}_{n})}+\left\Vert (g_{n}-P_{n})dz\right\Vert _{L^{2}(\mathcal{C}_{n})}.
\end{align*}
Further calculation gives $\left\Vert dz\right\Vert _{L^{2}(\mathcal{C}_{n})}=\sqrt{2R_{n}}$,
$\left\Vert (f_{n}+\tilde{S}_{n})dz\right\Vert _{L^{2}(\mathcal{C}_{n})}=\left\Vert f_{n}+\tilde{S}_{n}\right\Vert _{L^{2}([-R_{n},R_{n}]\times S^{1})}$
and similarly $\left\Vert (g_{n}-P_{n})dz\right\Vert _{L^{2}(\mathcal{C}_{n})}=\left\Vert g_{n}-P_{n}\right\Vert _{L^{2}([-R_{n},R_{n}]\times S^{1})}$.
Application of Lemma \ref{lem:Any-harmonic-functions} yields
\begin{align*}
\left\Vert f_{n}+\tilde{S}_{n}\right\Vert _{L^{2}([-R_{n},R_{n}]\times S^{1})}^{2} & =\int_{-R_{n}}^{R_{n}}\left\Vert f_{n}(s)+\tilde{S}_{n}\right\Vert _{L^{2}(S^{1})}^{2}ds\\
 & \leq\left(36\int_{-R_{n}}^{R_{n}}\rho^{2}(s)ds\right)\max\left\{ \left\Vert f_{n}(\pm R_{n})+\tilde{S}_{n}\right\Vert _{L^{2}(S^{1})}^{2}\right\} 
\end{align*}
and
\[
\left\Vert g_{n}-P_{n}\right\Vert _{L^{2}([-R_{n},R_{n}]\times S^{1})}^{2}\leq\left(36\int_{-R_{n}}^{R_{n}}\rho^{2}(s)ds\right)\max\left\{ \left\Vert g_{n}(\pm R_{n})-P_{n}\right\Vert _{L^{2}(S^{1})}^{2}\right\} .
\]
Using
\[
\int_{-R_{n}}^{R_{n}}\rho^{2}(s)ds=4(1-e^{-4R_{n}})\leq4,
\]
we obtain
\begin{align*}
\left\Vert f_{n}+\tilde{S}_{n}\right\Vert _{L^{2}([-R_{n},R_{n}]\times S^{1})}^{2} & \leq144\max\left\{ \left\Vert f_{n}(\pm R_{n})+\tilde{S}_{n}\right\Vert _{L^{2}(S^{1})}^{2}\right\} ,\\
\left\Vert g_{n}-P_{n}\right\Vert _{L^{2}([-R_{n},R_{n}]\times S^{1})}^{2} & \leq144\max\left\{ \left\Vert g_{n}(\pm R_{n})-P_{n}\right\Vert _{L^{2}(S^{1})}^{2}\right\} .
\end{align*}
Because the harmonic $1-$forms $\gamma_{n}$ converge to $0$ in
$C_{\text{loc}}^{\infty}(S^{\mathcal{D},r}\backslash\Gamma)$, $f_{n}(\pm R_{n})$
, $g_{n}(\pm R_{n})$, $\tilde{S}_{n}$, and $P_{n}$ converge to
zero. Hence
\[
\left|\left\Vert F_{n}\right\Vert _{L^{2}(\mathcal{C}_{n})}-\sqrt{2}|b_{0}|\sqrt{R_{n}}\right|\rightarrow0
\]
as $n\rightarrow\infty$. As $\left\Vert F_{n}\right\Vert _{L^{2}(\mathcal{C}_{n})}$
is almost $1$, there exists the constants $C_{0},C_{1}>0$ such that
\begin{eqnarray*}
C_{0}\frac{1}{\sqrt{2R_{n}}} & \leq & |b_{0}|\leq C_{1}\frac{1}{\sqrt{2R_{n}}}
\end{eqnarray*}
giving
\[
C_{0}\frac{1}{\sqrt{2R_{n}}}\leq\sqrt{P_{n}^{2}+\tilde{S}_{n}^{2}}\leq C_{1}\frac{1}{\sqrt{2R_{n}}},
\]
or equivalently,
\[
C_{0}\sqrt{\frac{R_{n}}{2}}\leq\sqrt{(P_{n}R_{n})^{2}+(\tilde{S}_{n}R_{n})^{2}}\leq C_{1}\sqrt{\frac{R_{n}}{2}.}
\]
These inequalities show that either $P_{n}R_{n}$ or $\tilde{S}_{n}R_{n}$
tend to $\infty$, although $P_{n}$ and $\tilde{S}_{n}$ stay uniformly
bounded ($\gamma_{n}$ have uniformly bounded $L^{2}-$norms). If
$P_{n}R_{n}$ remains uniformly bounded then we replace $\gamma_{n}$
by $\gamma_{n}\circ j_{n}$.
\end{singlespace}
\end{proof}
\begin{lem}
\begin{singlespace}
\noindent \label{lem:Any-harmonic-functions}For any harmonic functions
$f$ and $g$ on the cylinder $[-R,R]\times S^{1}$ such that $\eta=fds+gdt$
is a harmonic $1-$form on $[-R,R]\times S^{1}$ we have
\begin{eqnarray*}
\left\Vert f(s)+\tilde{S}\right\Vert _{L^{2}(S^{1})} & \leq & 6\rho(s)\max\left\{ \left\Vert f(\pm R)+\tilde{S}\right\Vert _{L^{2}(S^{1})}\right\} \\
\left\Vert g(s)-P\right\Vert _{L^{2}(S^{1})} & \leq & 6\rho(s)\max\left\{ \left\Vert g(\pm R)-P\right\Vert _{L^{2}(S^{1})}\right\} 
\end{eqnarray*}
 for all $s\in[-R,R]$. Here $\tilde{S}$ and $P$ are the co-period
and the period of $\eta$, respectively, and $\rho(s)^{2}=8e^{-2R}\cosh(2s)$.
\end{singlespace}
\end{lem}
\begin{proof}
\begin{singlespace}
\noindent Any harmonic $1-$form $\eta$ defined on the cylinder $[-R,R]\times S^{1}$
can be written as $\eta=(-\tilde{S}ds+Pdt)+\tilde{f}(s,t)ds+\tilde{g}(s,t)dt$
where $\tilde{f}$ and $\tilde{g}$ are harmonic functions on $[-R,R]\times S^{1}$
with vanishing average. Note that the average of $f$ corresponds
to the co-period $\tilde{S}$ and the average of $g$ corresponds
to $-P$. To show this, write $\eta$ in the form $\eta=f(s,t)ds+g(s,t)dt$
and compute the averages of $f$ and $g$ as
\[
\frac{1}{2R}\int_{[-R,R]\times S^{1}}f(s,t)ds\wedge dt=\frac{1}{2R}\int_{[-R,R]\times S^{1}}\eta\wedge dt=\int_{\{0\}\times S^{1}}\eta\circ j=-\tilde{S}
\]
and
\[
\frac{1}{2R}\int_{[-R,R]\times S^{1}}g(s,t)ds\wedge dt=\frac{1}{2R}\int_{[-R,R]\times S^{1}}ds\wedge\eta=-\frac{1}{2R}\int_{-R}^{R}\left(\int_{\{s\}\times S^{1}}\eta\right)ds=P,
\]
respectively. Hence the $1-$form $\eta-(-\tilde{S}ds+Pdt)=\tilde{f}(s,t)ds+\tilde{g}(s,t)dt$
has vanishing average twist and vanishing periods. The Fourier series
of $\tilde{f}$ and $\tilde{g}$ in the $t$ variable are
\begin{align*}
\tilde{f}(s,t) & =\frac{a_{0}(s)}{2}+\sum_{k=1}^{\infty}a_{k}(s)\cos(kt)+b_{k}(s)\sin(kt),\\
\tilde{g}(s,t) & =\frac{\alpha_{0}(s)}{2}+\sum_{k=1}^{\infty}\alpha_{k}(s)\cos(kt)+\beta_{k}(s)\sin(kt).
\end{align*}
Since $\tilde{f}$ and $\tilde{g}$ are harmonic, the Fourier expansion
coefficients solve $a_{k}''=k^{2}a_{k}$, $b_{k}''=k^{2}b_{k}$, $\alpha_{k}''=k^{2}\alpha_{k}$
and $\beta_{k}''=k^{2}\beta_{k}$ for $k\in\mathbb{N}_{0}$. The solutions
to these ordinaty differential equations are of the form
\begin{align*}
a_{0}(s) & =c_{0}+sd_{0},\\
a_{k}(s) & =c_{k}\cosh(ks)+d_{k}\sinh(ks),\\
b_{k}(s) & =e_{k}\cosh(ks)+f_{k}\sinh(ks),\\
\alpha_{0}(s) & =\delta_{0}+\epsilon_{0}s,\\
\alpha_{k}(s) & =\delta_{k}\cosh(ks)+\epsilon_{k}\sinh(ks),\\
\beta_{k}(s) & =\eta_{k}\cosh(ks)+\theta_{k}\sinh(ks).
\end{align*}
Since $d\eta=d(\eta\circ j)=0$ we obtain $\partial_{t}\tilde{f}=\partial_{s}\tilde{g}$
and $\partial_{s}\tilde{f}=-\partial_{t}\tilde{g}$, giving $a_{0}(s)=c_{0}$
and $\alpha_{0}(s)=\delta_{0}$. As $\tilde{f}ds+\tilde{g}dt$ has
vanishing co-period and vanishing period, we find $a_{0}(s)=\alpha_{0}(s)=0$,
and the following relations relating the coefficients $a_{k}$, $b_{k}$,
$\alpha_{k}$, and $\beta_{k}$ for $k\in\mathbb{N}$: $\delta_{k}=f_{k}$,
$\epsilon_{k}=e_{k}$, $\eta_{k}=-d_{k}$ and $\theta_{k}=-c_{k}$.
Consequently, $a_{k}$, $b_{k}$, $\alpha_{k}$, and $\beta_{k}$
can be written as
\begin{align*}
a_{k}(s) & =c_{k}\cosh(ks)+d_{k}\sinh(ks),\\
b_{k}(s) & =e_{k}\cosh(ks)+f_{k}\sinh(ks),\\
\alpha_{k}(s) & =f_{k}\cosh(ks)+e_{k}\sinh(ks),\\
\beta_{k}(s) & =-d_{k}\cosh(ks)-c_{k}\sinh(ks).
\end{align*}
Let us express $\tilde{f}$ and $\tilde{g}$ as
\begin{align*}
\tilde{f}(s,t) & =\sum_{k=1}^{\infty}a_{k}(s)\cos(kt)+b_{k}(s)\sin(kt)=\sum_{k\in\mathbb{Z}\backslash\{0\}}F_{k}(s)e^{2\pi ikt},\\
\tilde{g}(s,t) & =\sum_{k=1}^{\infty}\alpha_{k}(s)\cos(kt)+\beta_{k}(s)\sin(kt)=\sum_{k\in\mathbb{Z}\backslash\{0\}}\varGamma_{k}(s)e^{2\pi ikt},
\end{align*}
where $F_{k}=\frac{1}{2}(a_{k}-ib_{k})$, $F_{-k}=\frac{1}{2}(a_{k}+ib_{k})$,
$\varGamma_{k}=\frac{1}{2}(\alpha_{k}-i\beta_{k})$ , and $\varGamma_{-k}=\frac{1}{2}(\alpha_{k}+i\beta_{k})$
for $k\geq1$. From
\[
\frac{\cosh(ks)}{\cosh(kR)}\leq3e^{-R}\cosh(s)\leq3\rho(s)\text{ and }\frac{|\sinh(ks)|}{|\sinh(Rs)|}\leq3e^{-R}\cosh(s)\leq3\rho(s),
\]
where $\rho(s)^{2}=8e^{-2R}\cosh(2s)$, it follows that
\[
\cosh(ks)\leq3\rho(s)\cosh(kR)\text{ and }|\sinh(ks)|\leq3\rho(s)\sinh(Rs).
\]
Define the functions
\[
K(k)=\begin{cases}
+1, & c_{k}\text{ and }d_{k}\text{ have the same parity}\\
-1, & \text{otherwise}
\end{cases}
\]
and
\[
G(k)=\begin{cases}
+1, & e_{k}\text{ and }f_{k}\text{ have the same parity}\\
-1, & \text{otherwise.}
\end{cases}
\]
For $s\in[0,R]\times S^{1}$ we then have
\begin{align*}
\left\Vert \tilde{f}(s)\right\Vert _{L^{2}(S^{1})}^{2} & =\sum_{k\in\mathbb{Z}\backslash\{0\}}|F_{k}(s)|^{2}\\
 & =\frac{1}{2}\sum_{k=1}^{\infty}(c_{k}\cosh(ks)+d_{k}\sinh(ks))^{2}+\frac{1}{2}\sum_{k=1}^{\infty}(e_{k}\cosh(ks)+f_{k}\sinh(ks))^{2}\\
 & =\frac{1}{2}\sum_{k=1,K(k)=1}^{\infty}(c_{k}\cosh(ks)+d_{k}\sinh(ks))^{2}+\frac{1}{2}\sum_{k=1,K(k)=-1}^{\infty}(c_{k}\cosh(ks)+d_{k}\sinh(ks))^{2}\\
 & +\frac{1}{2}\sum_{k=1,G(k)=1}^{\infty}(e_{k}\cosh(ks)+f_{k}\sinh(ks))^{2}+\frac{1}{2}\sum_{k=1,G(k)=-1}^{\infty}(e_{k}\cosh(ks)+f_{k}\sinh(ks))^{2}\\
 & =\frac{1}{2}\sum_{k=1,K(k)=1}^{\infty}c_{k}^{2}\cosh^{2}(ks)+d_{k}^{2}\sinh^{2}(ks)+2c_{k}d_{k}\cosh(ks)\sinh(ks)\\
 & +\frac{1}{2}\sum_{k=1,K(k)=-1}^{\infty}c_{k}^{2}\cosh^{2}(ks)+d_{k}^{2}\sinh^{2}(ks)-(-2c_{k}d_{k})\cosh(ks)\sinh(ks)\\
 & +\frac{1}{2}\sum_{k=1,G(k)=1}^{\infty}e_{k}^{2}\cosh^{2}(ks)+f_{k}^{2}\sinh^{2}(ks)+2e_{k}f_{k}\cosh(ks)\sinh(ks)\\
 & +\frac{1}{2}\sum_{k=1,G(k)=-1}^{\infty}e_{k}^{2}\cosh^{2}(ks)+f_{k}^{2}\sinh^{2}(ks)-(-2e_{k}f_{k})\cosh(ks)\sinh(ks)\\
 & \leq\frac{1}{2}9\rho(s)^{2}\sum_{k=1,K(k)=1}^{\infty}c_{k}^{2}\cosh^{2}(kR)+d_{k}^{2}\sinh^{2}(kR)+2c_{k}d_{k}\cosh(kR)\sinh(kR)\\
 & +\frac{1}{2}9\rho(s)^{2}\sum_{k=1,K(k)=-1}^{\infty}c_{k}^{2}\cosh^{2}(kR)+d_{k}^{2}\sinh^{2}(kR)+2c_{k}d_{k}\cosh(kR)\sinh(-kR)\\
 & +\frac{1}{2}9\rho(s)^{2}\sum_{k=1,G(k)=1}^{\infty}e_{k}^{2}\cosh^{2}(kR)+f_{k}^{2}\sinh^{2}(kR)+2e_{k}f_{k}\cosh(kR)\sinh(kR)\\
 & +\frac{1}{2}9\rho(s)^{2}\sum_{k=1,G(k)=-1}^{\infty}e_{k}^{2}\cosh^{2}(Rk)+f_{k}^{2}\sinh^{2}(Rk)+2e_{k}f_{k}\cosh(Rk)\sinh(-kR)\\
 & =\frac{9}{2}\rho(s)^{2}\sum_{k=1,K(k)=1}^{\infty}(c_{k}\cosh(kR)+d_{k}\sinh(kR))^{2}+\frac{9}{2}\rho(s)^{2}\sum_{k=1,K(k)=-1}^{\infty}(c_{k}\cosh(-kR)+d_{k}\sinh(-kR))^{2}\\
 & +\frac{9}{2}\rho(s)^{2}\sum_{k=1,G(k)=1}^{\infty}(e_{k}\cosh(kR)+f_{k}\sinh(kR))^{2}+\frac{9}{2}\rho(s)^{2}\sum_{k=1,G(k)=-1}^{\infty}(e_{k}\cosh(-kR)+f_{k}\sinh(-kR))^{2}
\end{align*}
\begin{align*}
 & \leq\frac{9}{2}\rho(s)^{2}\sum_{k=1}^{\infty}(c_{k}\cosh(kR)+d_{k}\sinh(kR))^{2}+\frac{9}{2}\rho(s)^{2}\sum_{k=1}^{\infty}(c_{k}\cosh(-kR)+d_{k}\sinh(-kR))^{2}\\
 & +\frac{9}{2}\rho(s)^{2}\sum_{k=1}^{\infty}(e_{k}\cosh(kR)+f_{k}\sinh(kR))^{2}+\frac{9}{2}\rho(s)^{2}\sum_{k=1}^{\infty}(e_{k}\cosh(-kR)+f_{k}\sinh(-kR))^{2}\\
 & =\frac{9}{2}\rho(s)^{2}\sum_{k=1}^{\infty}a_{k}(R)^{2}+\frac{9}{2}\rho(s)^{2}\sum_{k=1}^{\infty}a_{k}(-R)^{2}+\frac{9}{2}\rho(s)^{2}\sum_{k=1}^{\infty}b_{k}(R)^{2}+\frac{9}{2}\rho(s)^{2}\sum_{k=1}^{\infty}b_{k}(-R)^{2}\\
 & =9\rho(s)^{2}\left(\left\Vert \tilde{f}(R)\right\Vert _{L^{2}(S^{1})}^{2}+\left\Vert \tilde{f}(-R)\right\Vert _{L^{2}(S^{1})}^{2}\right).
\end{align*}
The same inequality holds for negative $s$, and a similar estimate
can be derived for the harmonic function $\tilde{g}$. Thus
\begin{align*}
\left\Vert \tilde{f}(s)\right\Vert _{L^{2}(S^{1})}^{2} & \leq9\rho(s)^{2}\left(\left\Vert \tilde{f}(R)\right\Vert _{L^{2}(S^{1})}^{2}+\left\Vert \tilde{f}(-R)\right\Vert _{L^{2}(S^{1})}^{2}\right),\\
\left\Vert \tilde{g}(s)\right\Vert _{L^{2}(S^{1})}^{2} & \leq9\rho(s)^{2}\left(\left\Vert \tilde{g}(R)\right\Vert _{L^{2}(S^{1})}^{2}+\left\Vert \tilde{g}(-R)\right\Vert _{L^{2}(S^{1})}^{2}\right),
\end{align*}
and from $\tilde{f}(s,t):=f(s,t)+\tilde{S}$ and $\tilde{g}(s,t):=g(s,t)-P$,
we end up with
\begin{align*}
\left\Vert f(s)+\tilde{S}\right\Vert _{L^{2}(S^{1})}^{2} & \leq9\rho(s)^{2}\left(\left\Vert f(R)+\tilde{S}\right\Vert _{L^{2}(S^{1})}^{2}+\left\Vert f(-R)+\tilde{S}\right\Vert _{L^{2}(S^{1})}^{2}\right)\\
 & \leq18\rho(s)^{2}\max\left\{ \left\Vert f(R)+\tilde{S}\right\Vert _{L^{2}(S^{1})}^{2},\left\Vert f(-R)+\tilde{S}\right\Vert _{L^{2}(S^{1})}^{2}\right\} ,\\
\left\Vert g(s)-P\right\Vert _{L^{2}(S^{1})}^{2} & \leq9\rho(s)^{2}\left(\left\Vert g(R)-P\right\Vert _{L^{2}(S^{1})}^{2}+\left\Vert g(-R)-P\right\Vert _{L^{2}(S^{1})}^{2}\right)\\
 & \leq18\rho(s)^{2}\max\left\{ \left\Vert g(R)-P\right\Vert _{L^{2}(S^{1})}^{2},\left\Vert g(-R)-P\right\Vert _{L^{2}(S^{1})}^{2}\right\} .
\end{align*}
\end{singlespace}
\end{proof}
\begin{rem}
\begin{singlespace}
\noindent In \cite{key-15}, a notion of convergence for $\mathcal{H}-$holomorphic
curves is derived by using a result (Lemma A.2) which states that
the conformal co-period of a harmonic $1-$form on a Riemann surface
can be universally controlled by its periods. Proposition \ref{prop:There-exists-a-2}
gives a counterexample to this statement.
\end{singlespace}
\end{rem}

\appendix
\begin{singlespace}

\section{\label{sec:Pseudoholomorphic-disks-with}Holomorphic disks with fixed
boundary}
\end{singlespace}

\begin{singlespace}
\noindent This appendix is devoted to the description of the convergence
of pseudoholomorphic disks with fixed boundaries in symplectization,
as well as, of their limit object. The results are used for proving
the convergence of a cylinder of ''finite length'', i.e. of type
$b_{1}$ as discussed in Section \ref{subsec:Cylinders-of-type-1}.

\noindent Let $u_{n}=(a_{n},f_{n}):D\rightarrow\mathbb{R}\times M$
be a sequence of pseudoholomorphic curves in the symplectization $\mathbb{R}\times M$
of the contact manifold $(M,\alpha)$, and being defined on the open
unit disk $D$ with respect to the standard complex structure $i$
and the cylindrical almost complex structure $J$ on $\mathbb{R}\times M$.
For any $\tau>0$ we assume that there exists a subsequence of $u_{n}$,
also denoted by $u_{n}$, such that 
\begin{equation}
u_{n}\rightarrow u\label{eq:gradient_on_boundary}
\end{equation}
as $n\rightarrow\infty$ in $C^{\infty}(D\backslash\overline{D_{\tau}(0)})$.
Furthermore, we assume that the Hofer energy $E_{\textrm{H}}(u_{n};D)$
of $u_{n}$ is uniformly bounded. In the following we analyze the
convergence of $u_{n}$. 

\noindent The functions $a_{n}$ can be supposed to be not uniformly
bounded. If this is not the case, we may deduce using standard bubbling-off
analysis that the gradients of $u_{n}$ are uniformly bounded on all
of $D$, which in turn, implies that $u_{n}$ converge in $C^{\infty}(D)$
to a pseudoholomorphic disk with finite Hofer energy.

\noindent To describe the convergence and the limit object we use
the arguments from \cite{key-10} and \cite{key-14}. However, we
drop the details and explain only the strategy and mention the convergence
result. As we have assumed that the $\mathbb{R}-$coordinates of $u_{n}$
are unbounded, the maximum principle for subharmonic functions gives
$a_{n}\rightarrow-\infty$. By (\ref{eq:gradient_on_boundary}) we
have the $C^{\infty}-$convergence of $u_{n}$ on an arbitrary neighbourhood
of $\partial D$, and by a specific choice of this neighbourhood,
we assume that the $\mathbb{R}-$components of $u_{n}$, when restricted
to this neighbourhood, do not leave a fixed interval $[-K,K]$ for
some $K\in\mathbb{R}$ with $K>0$. Thus from level $-K-2$ we start
with the decomposition of $a_{n}^{-1}((-\infty,-K-2])$ into cylindrical,
essential and one ``bottom'' boundary components. This decomposition
which is identical to the decomposition done in \cite{key-10} and
\cite{key-14} is illustrated in Figure \ref{fig6}. From \cite{key-10}
and \cite{key-14} we know that there are at most $N_{0}\in\mathbb{N}$
cylindrical components.

\noindent \begin{figure}     
	\centering  
	\def\svgscale{1} 	
	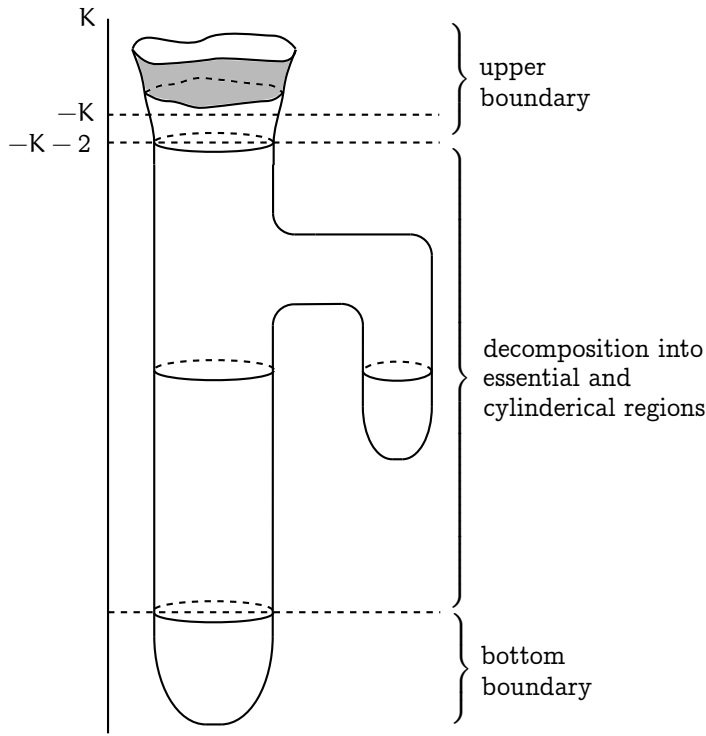  
	\caption{Decomposition of $a_{n}^{-1}((-\infty, -K-2])$}
	\label{fig6}
\end{figure}

\noindent In addition to the above decomposition, we add one more
boundary components, namely the ``upper'' boundary component. This
surface has two types of boundaries. The first one is the boundary
$\partial D$ which lies in a specific neighbourhood such that its
image under $u_{n}$ belongs to $[-K,K]\times M$. The second one
is the boundary which connects certain cylindrical components. For
the cylindrical, essential and bottom boundary components we use the
results established in \cite{key-10} and \cite{key-14} to describe
the convergence and the limit object. For the ``upper'' boundary
component we use Theorem 3.2 of \cite{key-10}, also known as ``Gromov
compactness with free boundary''. Here the choice of the neighbourhood
of $\partial D$, on which the $\mathbb{R}-$components of $u_{n}$
lie in $[-K,K]$, plays an essential role. The existence of a special
parametrization of a neighbourhood of\textbf{ $\partial D$ }will
enable us to apply ``Gromov compactness with free boundary'' in
the analysis of the convergence property of the upper boundary component.
Essentially, the application of ``Gromov compactness with free boundary'',
requires that the properties (A4) and (A5) under Definition 3.1 of
\cite{key-10} are satisfied. The following considerations ensure
these conditions: Choose $L_{0}'\geq1$ as in Remark 3.3 after Theorem
3.2 of \cite{key-10}. More precisely, $L_{0}'$ depends only on the
genus $g$ of the surface, the number of boundary components $m$,
the number of marked points $q$, the uniform bound $C$ on the area
of the considered pseudoholomorphic curves, the constant $\epsilon_{0}$
from Remark II.4.3 of \cite{key-5}, and the constant $C_{\text{ML}}$
from Lemma 3.17 of \cite{key-10} (the classical monotonicity lemma).
For this $L_{0}'$ we write $L_{0}'(g,m,q,C,\epsilon_{0},C_{ML})$.
Further on, choose $L_{0}$ as 
\begin{align*}
L_{0} & :=\max\left\{ L_{0}'(0,1,2,C,\epsilon_{0},C_{\text{ML}}),L_{0}'(0,2,1,C,\epsilon_{0},C_{\text{ML}}),\right.\\
 & \left.L_{0}'(0,3,0,C,\epsilon_{0},C_{\text{ML}}),...,L_{0}'(0,2N_{0},0,C,\epsilon_{0},C_{\text{ML}})\right\} .
\end{align*}
Note that when determining the constant $L_{0}'$ in the first two
cases, we introduce one and two artificial punctures, i.e. $q=2$
or $q=1$, in order to make our surface stable. Set $\tau_{0}=e^{-10\pi L_{0}}$
and choose $\tau<\tau_{0}$. In view of (\ref{eq:gradient_on_boundary}),
assume that there exists a constant $K>0$ such that $u_{n}(D\backslash\overline{D_{\tau}(0)})\subset[-K,K]\times M$
for all $n\in\mathbb{N}$. Hence the boundary is fixed in the symplectization.
The boundary region can be conformaly parametrized as follows. Consider
the map $\beta_{\partial D,0}:[0,5L_{0}]\times S^{1}\rightarrow D\backslash D_{\tau_{0}}(0)$,
$(s,t)\mapsto e^{-2\pi(s+it)}$. This map is obviously a conformal
parametrization of the boundary region. Let now $L=-\ln(\tau)/10\pi$.
Obviously, $L\geq L_{0}$ and the map $\beta_{\partial D}:[0,5L]\times S^{1}\rightarrow D\backslash D_{\tau}(0)$,
$(s,t)\mapsto e^{-2\pi(s+it)}$ is a conformal parametrization of
a neighbourhood of the boundary circle $\partial D$. Fix this boundary.
This conformal parametrization is obviously independent of $n$ and
will be used in conjunction with ``Gromov compactness with free boundary''.
Finally, glue the upper boundary component to the rest of the surface,
and obtain the resulting limit surface together with the convergence
description. 

\noindent To formulate the convergence result we introduce some notations.
Let $Z$ be an oriented surface diffeomorphis to the standard unit
disk $D$ and $\Delta=\Delta_{n}\amalg\Delta_{p}\subset Z$ a collection
of finitely many disjoint simple loops divided into two disjoint sets.
Denote by $Z_{\Delta_{n}}$ the surface obtained by collapsing the
curves in $\Delta_{n}$ to points. Write
\[
Z^{*}:=Z_{\Delta_{n}}\backslash\Delta_{p}=:Z^{(0)}\amalg\coprod_{\nu=1}^{N}Z^{(\nu)}\amalg Z^{(N+1)}
\]
as a disjoint union of components $Z^{(\nu)}$. Here $Z^{(0)}$ is
the bottom boundary component which is the disjoint union of finetly
many disks, while $Z^{(N+1)}$ is the upper boundary component whose
boundary is of two types. One type is the boundary of the disk $D$
and the other boundary components are certain loops from $\Delta_{p}$.
Let $j$ be a conformal structure on $Z\backslash\Delta$ such that
$(Z\backslash\Delta,j)$ is a punctured Riemann surface together with
an identification of distinct pairs of punctures given by the elements
of $\Delta$. This shows that $Z^{*}$ has the structure of a nodal
punctured Riemann surface with a remaining identification of punctures
given by the loops $\{\delta^{i}\}_{i\in I}=\Delta_{p}$, for some
index set $I$. A broken pseudoholomorphic curve (with $N+2$ levels)
is a map $F=(F^{(0)},F^{(1)},...,F^{(N+1)}):(Z^{*},j)\rightarrow X$,
where $X=\coprod_{\nu=0}^{N+1}(\mathbb{R}\times M)$ such that $F^{(\nu)}:(Z^{(\nu)},j)\rightarrow\mathbb{R}\times M$
is a punctured pseudoholomorphic curve with the additional property
that $F$ extends to a continous map $\overline{F}:Z\rightarrow\overline{X}$.
Here $\overline{X}$ is obtained as follows. The negative end of the
compactification of $\mathbb{R}\times M$ of the $\nu-$th copy is
glued to the positive end of the compactification of $\mathbb{R}\times M$
of the copy $\nu+1$. This procedure is done for $\nu=0,...,N$. For
a loop $\delta\in\Delta_{p}$, there exists $\nu\in\{0,...,N\}$ such
that $\delta$ is adjacent to $Z^{(\nu)}$ and $Z^{(\nu+1)}$. Fix
an embedded annulli $A^{\delta,\nu}\cong[-1,1]\times S^{1}\subset Z\backslash\Delta_{n}$
such that $\{0\}\times S^{1}=\delta$, $\{-1\}\times S^{1}\subset Z^{(\nu)}$
and $\{1\}\times S^{1}\subset Z^{(\nu+1)}$.

\noindent In this context, we state a convergence result which has
been established in \cite{key-10} and \cite{key-14}.
\end{singlespace}
\begin{prop}
\begin{singlespace}
\noindent \label{prop:Upto-passing-to}The sequence of pseudoholomorphic
disks $u_{n}=(a_{n},f_{n}):(D,i)\rightarrow\mathbb{R}\times M$ satisfying
(\ref{eq:gradient_on_boundary}) and having a uniformly bounded Hofer
energy has a subsequence that converges to a broken pseudoholomorphic
curve $u=(a,f):(Z,j)\rightarrow\mathbb{R}\times M$ with $N+2$ levels
in the following sense: There exists a sequence of diffomorphisms
$\varphi_{n}:D\rightarrow Z$ and a sequence of negative real numbers
$\min(a_{n})=r_{n}^{(0)}<r_{n}^{(1)}<...<r_{n}^{(N+1)}=-K-2$ with
$K\in\mathbb{R}$ and $r_{n}^{(\nu+1)}-r_{n}^{(\nu)}\rightarrow\infty$
as $n\rightarrow\infty$ such that the following hold:
\end{singlespace}
\begin{enumerate}
\begin{singlespace}
\item $Z$ with the circles $\Delta$ collapsed to points is a nodal Riemann
surface (in the sense of the above discussion, but with boundary).
$i_{n}:=(\varphi_{n})_{*}i\rightarrow j$ in $C_{\text{loc}}^{\infty}$
on $Z\backslash\Delta$. For every $i\in I$, the annulus $(A^{i},(\varphi_{n})_{*}i)$
is conformally equivalent to a standard annulus $[-R_{n},R_{n}]\times S^{1}$
by a diffeomorphism of the form $(s,t)\mapsto(\kappa(s),t)$ with
$R_{n}\rightarrow\infty$ as $n\rightarrow\infty$. 
\item The sequence $u_{n}\circ\varphi_{n}^{-1}|_{Z^{(\nu)}}:Z^{(\nu)}\rightarrow\mathbb{R}\times M$
converges in $C_{\text{loc}}^{\infty}$ on $Z^{(\nu)}\backslash\Delta_{n}$
to a punctured nodal pseudoholomorphic curve $u^{(\nu)}:(Z^{(\nu)},j)\rightarrow\mathbb{R}\times M$,
and in $C_{\text{loc}}^{0}$ on $Z^{(\nu)}$. 
\item The sequence $f_{n}\circ\varphi_{n}^{-1}:Z\rightarrow M$ converges
in $C^{0}$ to a map $f:Z\rightarrow M$, whose restriction to $\Delta_{p}$
parametrizes the Reeb orbits and to $\Delta_{n}$ parametrizes points. 
\item For any $S>0$ , there exist $\rho>0$ and $K\in\mathbb{N}$ such
that $a_{n}\circ\varphi_{n}^{-1}(s,t)\in[r_{n}^{(\nu)}+S,r_{n}^{(\nu+1)}-S]$
for all $n\geq K$ and all $(s,t)\in A^{\delta,\nu}$ with $|s|\leq\rho$.
\item The diffeomorphisms $\varphi_{n}\circ\beta_{\partial D}:[0,5L]\times S^{1}\rightarrow Z$
are independent of $n$.
\end{singlespace}
\end{enumerate}
\end{prop}
\begin{singlespace}

\section{\label{chap:Half-cylinders-with}Half cylinders with small energy}
\end{singlespace}

\begin{singlespace}
\noindent This appendix is devoted to the description of the convergence
of a sequence of pseudoholomorphic half cylinders $u_{n}=(a_{n},f_{n}):[0,\infty)\times S^{1}\rightarrow\mathbb{R}\times M$
with uniformly bounded $\alpha-$ and $d\alpha-$energies. More precisely,
we assume that there exists a constant $E_{0}>0$ such that $E(u_{n};[0,\infty)\times S^{1})\leq E_{0}$
and 
\begin{equation}
E_{d\alpha}(u_{n};[0,\infty)\times S^{1})\leq\frac{\hbar}{2},\label{eq:energybdd3}
\end{equation}
where $\hbar>0$ is defined as in Section \ref{subsec:Cylinders}.
Since the $d\alpha-$energy is smaller than $\hbar/2$ it follows,
from the usual bubbling-off analysis, that the gradients of $u_{n}$
are uniformly bounded with respect to the standard Euclidian metric
on the cylinder $[0,\infty)\times S^{1}$ and the induced cylindrical
metric on $\mathbb{R}\times M$. To analyze the convergence of such
a sequence we use the results of Appendix \ref{sec:Pseudoholomorphic-disks-with}
and \cite{key-23}. As before we split the analysis of the convergence
in two parts, namely the $C_{\text{loc}}^{\infty}-$ and the $C^{0}-$convergence.
Before stating the convergence results we need some auxiliary results
similar to those from \cite{key-23}. We begin with a remark on the
asymptotic of a pseudoholomorphic half cylinder.
\end{singlespace}
\begin{rem}
\begin{singlespace}
\noindent \label{rem:Let--be}Let $u=(a,f):[0,\infty)\times S^{1}\rightarrow\mathbb{R}\times M$
be a pseudoholomorphic half cylinder with $E(u;[0,\infty)\times S^{1})\leq E_{0}$
and $E_{d\alpha}(u;[0,\infty)\times S^{1})\leq\hbar/2$. To describe
the behaviour of $u$ as $s\rightarrow\infty$, we first assume that
$u$ has a bounded image in $\mathbb{R}\times M$. Consider the conformal
transformation $h:[0,\infty)\times S^{1}\rightarrow D\backslash\{0\}$,
$(s,t)\mapsto e^{-2\pi(s+it)}$. Then $u\circ h^{-1}=(a\circ h^{-1},f\circ h^{-1})$
is a pseudoholomorphic punctured disk satisfying the same assumption
as $u$ does. By the removal of singularity, $u\circ h^{-1}$ can
be defined on the whole disk $D$. In this case we use the results
from Appendix \ref{sec:Pseudoholomorphic-disks-with} to describe
the convergence. If $u$ has an unbounded image in $\mathbb{R}\times M$,
then due to Proposition 5.6 from \cite{key-1}, there exists $T\not=0$
and a periodic orbit $x$ of $X_{\alpha}$ such that $x$ is of period
$|T|$ and
\[
\lim_{s\rightarrow\infty}f(s,t)=x(Tt)\ \text{ and }\ \lim_{s\rightarrow\infty}\frac{a(s,t)}{s}=T\ \text{ in }C^{\infty}(S^{1}).
\]
\end{singlespace}
\end{rem}
\begin{singlespace}
\noindent To analyze the convergence of the sequence of pseudoholomorphic
half cylinders $u_{n}=(a_{n},f_{n}):[0,\infty)\times S^{1}\rightarrow\mathbb{R}\times M$
we destinguish two cases. 

\noindent In the first case each element of a subsequence of $u_{n}$,
still denoted by $u_{n}$, has a bounded image in the symplectization
$\mathbb{R}\times M$. By Remark \ref{rem:Let--be} we consider the
sequence of pseudoholomorphic disks $u_{n}\circ h^{-1}:D\rightarrow\mathbb{R}\times M$
having uniformly bounded energies and small $d\alpha-$energies. After
applying bubbling-off analysis and accounting on the uniform energy
bounds as well as on the small $d\alpha-$energies, we obtain a subsequence
having uniform gradient bounds with respect to the Euclidian metric
on the domains and the induced metric on $\mathbb{R}\times M$. After
a specific shift in the $\mathbb{R}-$coordinate, $u_{n}\circ h^{-1}$
converge in $C^{\infty}$ to a pseudoholomorphic disk $u:D\rightarrow\mathbb{R}\times M$. 

\noindent In the second case each element of a subsequence of $u_{n}$,
still denoted by $u_{n}$, has an unbounded image in $\mathbb{R}\times M$.
In the following we assume that after a specific shift in the $\mathbb{R}-$coordinate,
$a_{n}(0,0)=0$. Before describing the convergence of $u_{n}$, we
prove an asymptotic result for punctures which is similar to that
given in \cite{key-1}.
\end{singlespace}
\begin{prop}
\begin{singlespace}
\noindent \label{prop:After-going-over}After going over to a subsequence
the pseudoholomorphic half cylinders $u_{n}$ are asysmptotic to the
same Reeb orbit, i.e. there exists a Reeb orbit $x$ and $T\not=0$
with $|T|\leq C$ and a sequence $c_{n}\in S^{1}$ such that
\[
\lim_{s\rightarrow\infty}f_{n}(s,t)=x(T(t{\color{red}{\normalcolor +c_{n}}}))\ \text{ and }\ \lim_{s\rightarrow\infty}\frac{a_{n}(s,t)}{s}=T.
\]
Moreover, $u_{n}\rightarrow u$ in $C_{\text{loc}}^{\infty}$, where
$u$ is a pseudoholomorphic half cylinder $u:[0,\infty)\times S^{1}\rightarrow\mathbb{R}\times M$
which is asymptotic to the same Reeb orbit $x(T(t{\color{red}{\normalcolor +c^{*}}}))$
of period $T$ as above. Here, $c^{*}\in S^{1}$ and $c_{n}\rightarrow c^{*}$
as $n\rightarrow\infty$.
\end{singlespace}
\end{prop}
\begin{proof}
\begin{singlespace}
\noindent Let the sequence $u_{n}$ be asymptotic to some Reeb orbit.
More precisely, for all $n\in\mathbb{N}$ there exist $T_{n}\not=0$
and a periodic orbit $x_{n}$ of period $|T_{n}|$ such that
\[
\lim_{s\rightarrow\infty}f_{n}(s,t)=x_{n}(T_{n}t)\ \text{ and }\ \lim_{s\rightarrow\infty}\frac{a_{n}(s,t)}{s}=T_{n}
\]
in $C^{\infty}(S^{1})$. For simplicity, choose a subsequence of $T_{n}$,
also denoted by $T_{n}$, which is always positive (positive puncture).
Since we are in the non-degenerate case and $T_{n}\leq E_{0}$, assume,
after going to some subsequence, that $T_{n}=\overline{T}>0$ and
$x_{n}(\overline{T}t)=\overline{x}(\overline{T}(t{\color{red}{\normalcolor +c_{n}}}))$,
where $c_{n}\in S^{1}$ for all $n$. Thus after going over to some
subsequence we may assume that $c_{n}\rightarrow c^{*}\in S^{1}$.
From the uniform boundedness of the gradients of $u_{n}$, the elliptic
regularity, and Arzelá-Ascoli theorem, we have $u_{n}\rightarrow u:[0,\infty)\times S^{1}\rightarrow\mathbb{R}\times M$
in $C_{\text{loc}}^{\infty}$. Here $u$ is a pseuhoholomorphic half
cylinder with bounded energy and a small $d\alpha-$energy which is
asymptotic to some periodic orbit with period $\underline{T}$ or
a point; both being denoted by $\underline{x}$ . Choose the sequences
$\underline{N}_{n},\overline{N}_{n}\overset{n\rightarrow\infty}{\longrightarrow}\infty$
and $\underline{N}_{n}<\overline{N}_{n}$ such that after going over
to a subsequence we have
\[
\lim_{n\rightarrow\infty}f_{n}(\underline{N}_{n},t)=\underline{x}(\underline{T}t)\ \text{ and }\ \lim_{n\rightarrow\infty}f_{n}(\overline{N}_{n},t)={\normalcolor \overline{x}({\normalcolor \overline{T}(t{\color{red}{\normalcolor +c^{*}}})})}\ \text{ in }C^{\infty}(S^{1}),
\]
and consider the maps 
\[
v_{n}=u_{n}|_{[\underline{N}_{n},\overline{N}_{n}]\times S^{1}}.
\]
which have by construction $d\alpha-$energy tending to $0$. Performing
the same analysis as in \cite{key-1} we conclude that $\underline{x}=\overline{x}$
and $\underline{T}=\overline{T}$.
\end{singlespace}
\end{proof}
\begin{singlespace}
\noindent To describe the $C^{0}-$convergence of $u_{n}$ we use
the results established in \cite{key-23}. In view of Proposition
\ref{prop:After-going-over}, choose a sequence $R_{n}>0$ such that
$R_{n}\rightarrow\infty$ and $a_{n}(R_{n},t)-TR_{n}\rightarrow0$
as $n\rightarrow\infty$. Consider the shifted maps $\overline{u}_{n}(s,t):=u_{n}(s+R_{n},t)-TR_{n}$
for $(s,t)\in[-R_{n},R_{n}]\times S^{1}$. These are pseudoholomorphic
cylinders with uniformly bounded $\alpha-$ and $d\alpha-$energies
and a $d\alpha-$energy smaller than $\hbar/2$. Recall that these
pseudoholomorphic cylinders are a special case of the $\mathcal{H}-$holomorphic
cylinders described in \cite{key-23}. We distinguish two cases corresponding
to subsequences with vanishing and non-vanshing center actions. In
latter case, the cater action is greater than $\hbar>0$. By Proposition
\ref{prop:After-going-over}, the first case does not appear and we
are left with the case in which $A(\overline{u}_{n})\geq\hbar$. By
Corollary 25 of \cite{key-23}, for every $\epsilon>0$ there exists
$h>0$ such that for all $n\in\mathbb{N}$ and $R_{n}>h$, $\text{dist}_{\overline{g}_{0}}(\overline{f}_{n}(s,t),x(Tt+c_{n}))<\epsilon$
and $|\overline{a}_{n}(s,t)-Ts-a_{0}|<\epsilon$ for all $(s,t)\in[-R_{n}+h,R_{n}-h]\times S^{1}$.
On the other hand, we have the following result: For every $\epsilon>0$
there exists $h>0$ such that for all $n\in\mathbb{N}$ and $R_{n}>h$,
$\text{dist}_{\overline{g}_{0}}(f_{n}(s,t),x(Tt+c_{n}))<\epsilon$
and $|a_{n}(s,t)-Ts-a_{0}|<\epsilon$ for all $(s,t)\in[h,2R_{n}-h]\times S^{1}$.
As $R_{n}$ can be choosen arbitrary large the following equivalent
statement readily follows: 
\end{singlespace}
\begin{cor}
\begin{singlespace}
\noindent \label{cor:For-every-}For every $\epsilon>0$ there exist
$h>0$ and $N\in\mathbb{N}$ such that for all $n\geq N$ , $\text{dist}_{\overline{g}_{0}}(f_{n}(s,t),x(Tt+c_{n}))<\epsilon$
and $|a_{n}(s,t)-Ts-a_{0}|<\epsilon$ for all $(s,t)\in[h,\infty)\times S^{1}$. 
\end{singlespace}
\end{cor}
\begin{singlespace}
\noindent Consider the diffeomorphism $\theta:[0,\infty)\times S^{1}\rightarrow\mathbb{R}\times M$
and the maps
\begin{equation}
g_{n}:=f_{n}\circ\theta^{-1}:[0,1)\times S^{1}\rightarrow M,\label{eq:C0-conv}
\end{equation}
which by Proposition \ref{prop:After-going-over} converge in $C_{\text{loc}}^{\infty}$
to a map $g:=f\circ\theta^{-1}:[0,1)\times S^{1}\rightarrow M$. By
Corollary \ref{cor:For-every-}, the maps $g_{n}$ and $g$ can be
continously extended to $[0,1]\times S^{1}$ by $g_{n}(1,t)=g(1,t)=x(Tt{\color{red}{\normalcolor +c_{n}}})$
for all $n\in\mathbb{N}$ and all $t\in S^{1}$. Hence due to Corollary
\ref{cor:For-every-}, $g_{n}$ converge in $C^{0}$ to $g$. As a
consequence, we formulate the following compactness property of the
sequence of pseudoholomorphic half cylinders $u_{n}:[0,\infty)\times S^{1}\rightarrow\mathbb{R}\times M$
with uniformly bounded energies and $d\alpha-$energies less than
$\hbar/2$:
\end{singlespace}
\begin{thm}
\begin{singlespace}
\noindent \label{thm:Let--be}Let $u_{n}$ be a sequence of pseudoholomorphic
curves having uniformly bounded energy by $E_{0}$ and satisfying
condition (\ref{eq:energybdd3}). Then there exists a subsequence
of $u_{n}$, still denoted by $u_{n}$, such that the following is
satisfied.
\end{singlespace}
\begin{enumerate}
\begin{singlespace}
\item $u_{n}$ is asysmptotic to the same Reeb orbit, i.e. there exists
a Reeb orbit $x$ and $T\not=0$ with $|T|\leq C$ and a sequence
$c_{n}\in S^{1}$ such that
\[
\lim_{s\rightarrow\infty}f_{n}(s,t)=x(Tt{\color{red}{\normalcolor +c_{n}}})\ \text{ and }\ \lim_{s\rightarrow\infty}\frac{a_{n}(s,t)}{s}=T.
\]
for all $n\in\mathbb{N}$. 
\item $u_{n}$ converge in $C_{\text{loc}}^{\infty}$ to a pseudoholomorphic
half cylinder $u:[0,\infty)\times S^{1}\rightarrow\mathbb{R}\times M$
having uniformly bounded energy by the constant $E_{0}$ and satisfying
condition (\ref{eq:energybdd3}).
\item The maps $g_{n}:[0,1]\times S^{1}\rightarrow M$ converge in $C^{0}$
to a map $g:[0,1]\times S^{1}\rightarrow M$ and satisfy $g(1,t)=x(T(t{\color{red}{\normalcolor +c^{*}}}))$,
where $x$ is a Reeb orbit of period $T\not=0$.
\end{singlespace}
\end{enumerate}
\end{thm}
\begin{singlespace}

\section{\label{chap:Special-coordinates}Special coordinates}
\end{singlespace}

\begin{singlespace}
\noindent Let $S$ be a compact surface with boundary, and let $j_{n}$
and $j$ be complex structures on $S$ for all $n\in\mathbb{N}$.
Additionally, let $h_{n}$ and $h$ be the hyperbolic structures on
$S$ with respect to $j_{n}$ and $j$, respectively. Assume that
$j_{n}\rightarrow j$ and $h_{n}\rightarrow h$ in $C^{\infty}(S)$.
In this appendix we construct a sequence of biholomorphic coordinates
around some point in $S$ with respect to the complex structure $j_{n}$
that converges in a certain sense to the biholomorphic coordinates
with respect to $j$. This result is used in Section \ref{chap:Compactness-results}
for proving the convergence on the thick part.
\end{singlespace}
\begin{lem}
\begin{singlespace}
\noindent \label{2_1_lem:There-exists-open-neighbourhood}For each
$z\in\mbox{int}(S)$ there exist the open neighbourhoods $U_{n}(z)=U_{n}$
and $U(z)=U$ of $z$ and the diffeomorphisms 
\begin{eqnarray*}
\psi_{n}:D_{1}(0) & \rightarrow & U_{n},\\
\psi:D_{1}(0) & \rightarrow & U
\end{eqnarray*}
such that 

\end{singlespace}\begin{enumerate}
\begin{singlespace}
\item $\psi_{n}$ are $i-j_{n}-$biholomorphisms and $\psi$ is a $i-j-$biholomorphism; 
\item \label{2_1_enu:bounded_grad_of_coordinate_maps_psi_n}$\psi_{n}\rightarrow\psi$
in $C_{\text{loc}}^{\infty}(D_{1}(0))$ as $n\rightarrow\infty$ with
respect to the Euclidian metric on $D_{1}(0)$ and $h$ on $S$; 
\item $\psi_{n}(0)=z$ for every $n$ and $\psi(0)=z$.
\end{singlespace}
\end{enumerate}
\end{lem}
\begin{proof}
\begin{singlespace}
\noindent Around $z\in\text{int}(S)$, choose the $i-j-$holomorphic
coordinates $c:D_{2}(0)\rightarrow U$ such that $U\subset\mbox{int}(S)$
and $c(0)=z$, and consider the complex structures $j^{(n)}:=c^{*}j_{n}$.
Since $j_{n}\rightarrow j$ as $n\rightarrow\infty$ in $C^{\infty}$,
$j^{(n)}\rightarrow i$ in $C_{\text{loc}}^{\infty}(D_{2}(0))$ as
$n\rightarrow\infty$. Let $d_{n}^{\mathbb{C}}$ be the operator defined
by $d_{n}^{\mathbb{C}}f=df\circ j^{(n)}$ and let $d^{\mathbb{C}}$
be the operator defined by $d^{\mathbb{C}}f=df\circ i$. Denote by
$p_{x}:\mathbb{R}^{2}\rightarrow\mathbb{R}$, $(x,y)\mapsto x$ the
projection onto the first coordinate. Consider the problem of finding
a smooth fumction $f:\overline{D_{1}(0)}\rightarrow\mathbb{R}$ such
that 
\begin{equation}
\begin{array}{cl}
dd_{n}^{\mathbb{C}}f & =0\text{ on }D_{1}(0),\\
f & =p_{x}\text{ on }\partial D_{1}(0)
\end{array}\label{2_1_proof:eq:1}
\end{equation}
for all $n$ and 
\begin{equation}
\begin{array}{cl}
dd^{\mathbb{C}}f & =0\text{ on }D_{1}(0),\\
f & =p_{x}\text{ on }\partial D_{1}(0).
\end{array}\label{2_1_proof:eq:2}
\end{equation}
As the second problem translates into 
\begin{equation}
\begin{array}{cl}
\Delta f & =0\text{ on }D_{1}(0),\\
f & =p_{x}\text{ on }\partial D_{1}(0),
\end{array}\label{2_1_proof:eq:3}
\end{equation}
where $\Delta$ is the standard Laplace operator in $\mathbb{\ensuremath{R}}$,
the unique solution is $f(x,y)=x$ for all $(x,y)\in\overline{D_{1}(0)}$.
To see the uniqueness observe that the difference of $f$ with any
other solution of (\ref{2_1_proof:eq:3}) solves $\Delta u=0$ with
$u|_{\partial D_{1}(0)}=0$. Thus from the maximum principle for harmonic
functions we deduce that $u\equiv0$, and so, that (\ref{2_1_proof:eq:3})
has the unique solution $f$. In coordinates representation, $j^{(n)}$
can be written as
\[
j^{(n)}=\left(\begin{array}{cc}
j_{11}^{(n)} & j_{12}^{(n)}\\
j_{21}^{(n)} & j_{22}^{(n)}
\end{array}\right)
\]
and take notice that $j^{(n)}\rightarrow i$ in $C^{\infty}$ on $D_{1}(0)$
as $n\rightarrow\infty$. The solutions of (\ref{2_1_proof:eq:1})
are equivalent to the solutions of
\begin{equation}
\begin{array}{cl}
dd_{n}^{\mathbb{C}}\tilde{f} & =t_{n}\text{ on }D_{1}(0),\\
\tilde{f} & =0\text{ on }\partial D_{1}(0),
\end{array}\label{eq:modified_laplace}
\end{equation}
where $t_{n}=-dd_{n}^{\mathbb{C}}p_{x}$. Hence $dd_{n}^{\mathbb{C}}$
is an elliptic and coercive operator, and thus by Proposition 5.10
from \cite{key-16}, the problem (\ref{eq:modified_laplace}) has
a uniquely weak solution $\tilde{f}_{n}\in W^{1,2}(D_{1}(0))$ for
all $n$. From regularity theorem, the solutions $\tilde{f}_{n}$
are smooth for all $n$. Thus $f_{n}:=\tilde{f}_{n}+p_{x}$ is the
smooth unique solution of (\ref{2_1_proof:eq:1}). Let us show that
$f_{n}\rightarrow f$ in $C_{\text{loc}}^{\infty}(D_{1}(0))$ as $n\rightarrow\infty$.
For $u_{n}:=f_{n}-f$ we have 
\[
\begin{array}{cl}
dd_{n}^{\mathbb{C}}u_{n} & =g_{n}\text{ on }D_{1}(0),\\
u_{n} & =0\text{ on }\partial D_{1}(0).
\end{array}
\]
Here, $g_{n}\in C^{\infty}(D_{1}(0))$ is defined by $g_{n}:=dd_{n}^{\mathbb{C}}f$,
and because of $j^{(n)}\rightarrow i$ in $C^{\infty}(D_{1}(0))$
as $n\rightarrow\infty$, $g_{n}$ converges to $0$ in $C_{\text{loc}}^{\infty}(D_{1}(0))$
as $n\rightarrow\infty$. For every $m\in\mathbb{N}_{0}$ we consider
the bounded operator $dd_{n}^{\mathbb{C}}:W_{\partial}^{2+m,2}(D_{1}(0),\mathbb{R})\rightarrow W^{m,2}(D_{1}(0),\mathbb{R})$,
where $W_{\partial}^{2+m,2}(D_{1}(0),\mathbb{R})$ consists of maps
from $W^{2+m,2}(D_{1}(0),\mathbb{R})$ that vanish at the boundary.
By Proposition 5.10 together with Propositons 5.18 and 5.19 of \cite{key-16}
we deduce that the operator $dd_{n}^{\mathbb{C}}$ is bounded invertible;
hence $u_{n}=(dd_{n}^{\mathbb{C}})^{-1}g_{n}$. Since $dd_{n}^{\mathbb{C}}\rightarrow\Delta$
in operator norm, $(dd_{n}^{\mathbb{C}})^{-1}$ is a uniformly bounded
family, and so, $\left\Vert u_{n}\right\Vert _{W^{m+2,2}}\rightarrow0$
as $n\rightarrow\infty$. Further on, as $m\in\mathbb{N}_{0}$ was
arbitrary, the Sobolev embedding theorem yields $u_{n}\rightarrow0$
in $C_{\text{loc}}^{\infty}(D_{1}(0))$ as $n\rightarrow\infty$.
Thus we have constructed a unique sequence of solutions $\{f_{n}:\overline{D_{1}(0)}\rightarrow\mathbb{R}\}_{n\in\mathbb{N}}$
of (\ref{2_1_proof:eq:1}), and a unique solution $f:\overline{D_{1}(0)}\rightarrow\mathbb{R},(x,y)\mapsto x$
of (\ref{2_1_proof:eq:2}) satisfying $f_{n}\rightarrow f$ in $C_{\text{loc}}^{\infty}(D_{1}(0))$
as $n\rightarrow\infty$.

\noindent According to Lemma 6.8.1 of \cite{key-8}, there exists
a $j^{(n)}-i-$holomorphic function $F_{n}:D_{1}(0)\rightarrow\mathbb{C}$
and a $i-i-$holomorhic function $F:D_{1}(0)\rightarrow\mathbb{C}$
such that $f_{n}=\Re(F_{n})$ and $f=\Re(F)$. Let us investigate
the unique extensions of the functions $F_{n}$ and $F$. For doing
this we set $F_{n}=f_{n}+ib$ and $F=f+ib$, where $b_{n},b:D_{1}(0)\rightarrow\mathbb{R}$
are harmonic functions. As $F_{n}$ and $F$ are $j^{(n)}-i-$holomorphic
and $i-i-$holomorphic, respectively, they solve the equations 
\[
dF_{n}+i\circ dF_{n}\circ j^{(n)}=0
\]
and 
\[
dF+i\circ dF\circ i=0,
\]
respectively, which in turn, are equivalent to 
\[
db_{n}=-df_{n}\circ j^{(n)}
\]
and 
\[
db=-df\circ i,
\]
respectively. By the harmonicity of $f_{n}$ and $f$, and the application
of Poincare lemma on $D_{1}(0)$, we find the solutions $b_{n}$ and
$b$ which are unique up to addition with some constant. They can
be make unique by requiring that $b_{n}(0)=0$ and $b(0)=0$. In particular,
we find $F(x,y)=x+iy$. Then we get $db_{n}\rightarrow db$ in $C_{\text{loc}}^{\infty}(D_{1}(0))$
as $n\rightarrow\infty$, and from $b_{n}(0)=0$ and $b(0)=0$, we
actually get $b_{n}\rightarrow b$ in $C_{\text{loc}}^{\infty}(D_{1}(0))$
as $n\rightarrow\infty$. Hence $F_{n}\rightarrow F=\text{id}$ in
$C_{\text{loc}}^{\infty}(D_{1}(0))$ as $n\rightarrow\infty$.

\noindent For $n$ large, $F_{n}$ is bijective onto its image (maybe
after shrinking the domain). This follows from the proof of the inverse
function theorem. With $\tilde{F}_{n}=F_{n}-f_{n}(0)$, the maps $\psi_{n}$
and $\psi$ are defined by $\psi_{n}=c\circ\tilde{F}_{n}:D_{1}(0)\rightarrow U_{n}$
and $\psi=c\circ F:D_{1}(0)\rightarrow U$ for sufficiently large
$n$, respectively.
\end{singlespace}
\end{proof}

\end{document}

%% file: fig28.pdf_tex
\begingroup%
  \makeatletter%
  \providecommand\color[2][]{%
    \errmessage{(Inkscape) Color is used for the text in Inkscape, but the package 'color.sty' is not loaded}%
    \renewcommand\color[2][]{}%
  }%
  \providecommand\transparent[1]{%
    \errmessage{(Inkscape) Transparency is used (non-zero) for the text in Inkscape, but the package 'transparent.sty' is not loaded}%
    \renewcommand\transparent[1]{}%
  }%
  \providecommand\rotatebox[2]{#2}%
  \ifx\svgwidth\undefined%
    \setlength{\unitlength}{103.763012bp}%
    \ifx\svgscale\undefined%
      \relax%
    \else%
      \setlength{\unitlength}{\unitlength * \real{\svgscale}}%
    \fi%
  \else%
    \setlength{\unitlength}{\svgwidth}%
  \fi%
  \global\let\svgwidth\undefined%
  \global\let\svgscale\undefined%
  \makeatother%
  \begin{picture}(1,2.23877395)%
    \put(0,0){\includegraphics[width=\unitlength,page=1]{fig28.pdf}}%
    \put(0.47505156,0.44844538){\color[rgb]{0,0,0}\makebox(0,0)[lb]{\smash{$\Bigg \uparrow$}}}%
    \put(0.47505156,1.55866923){\color[rgb]{0,0,0}\makebox(0,0)[lb]{\smash{$\Bigg \uparrow$}}}%
    \put(0.47505156,0.04753122){\color[rgb]{0,0,0}\makebox(0,0)[lb]{\smash{$\circlearrowright$}}}%
    \put(0.47505156,1.00355731){\color[rgb]{0,0,0}\makebox(0,0)[lb]{\smash{$\circlearrowright$}}}%
    \put(0.47505156,2.05210205){\color[rgb]{0,0,0}\makebox(0,0)[lb]{\smash{$\circlearrowright$}}}%
    \put(0,0){\includegraphics[width=\unitlength,page=2]{fig28.pdf}}%
    \put(0.89057482,1.88797465){\color[rgb]{0,0,0}\makebox(0,0)[lb]{\smash{$\underline{\Gamma}_{ij}$}}}%
    \put(0.89065013,0.33057731){\color[rgb]{0,0,0}\makebox(0,0)[lb]{\smash{$\overline{\Gamma}_{i-1,j}$}}}%
    \put(0.89062253,0.76135924){\color[rgb]{0,0,0}\makebox(0,0)[lb]{\smash{$\left \{ -\frac{1}{2} \right \} \times S^{1}$}}}%
    \put(0.89059641,1.45524926){\color[rgb]{0,0,0}\makebox(0,0)[lb]{\smash{$\left \{ \frac{1}{2} \right \} \times S^{1}$}}}%
    \put(0.6063205,0.50862275){\color[rgb]{0,0,0}\makebox(0,0)[lb]{\smash{$\overline{\Phi}_{ij}$}}}%
    \put(0.60634203,1.64968629){\color[rgb]{0,0,0}\makebox(0,0)[lb]{\smash{$\underline{\Phi}_{ij}$}}}%
  \end{picture}%
\endgroup%

%% file: fig12.pdf_tex
\begingroup%
  \makeatletter%
  \providecommand\color[2][]{%
    \errmessage{(Inkscape) Color is used for the text in Inkscape, but the package 'color.sty' is not loaded}%
    \renewcommand\color[2][]{}%
  }%
  \providecommand\transparent[1]{%
    \errmessage{(Inkscape) Transparency is used (non-zero) for the text in Inkscape, but the package 'transparent.sty' is not loaded}%
    \renewcommand\transparent[1]{}%
  }%
  \providecommand\rotatebox[2]{#2}%
  \ifx\svgwidth\undefined%
    \setlength{\unitlength}{350.30870361bp}%
    \ifx\svgscale\undefined%
      \relax%
    \else%
      \setlength{\unitlength}{\unitlength * \real{\svgscale}}%
    \fi%
  \else%
    \setlength{\unitlength}{\svgwidth}%
  \fi%
  \global\let\svgwidth\undefined%
  \global\let\svgscale\undefined%
  \makeatother%
  \begin{picture}(1,0.18452767)%
    \put(0,0){\includegraphics[width=\unitlength,page=1]{fig12.pdf}}%
    \put(0.11102902,0.07443495){\color[rgb]{0,0,0}\makebox(0,0)[lb]{\smash{$\infty$}}}%
    \put(0.30321448,0.07540502){\color[rgb]{0,0,0}\makebox(0,0)[lb]{\smash{$b_{1}$}}}%
    \put(0.49925853,0.07443495){\color[rgb]{0,0,0}\makebox(0,0)[lb]{\smash{$\infty$}}}%
    \put(0.8783522,0.07443495){\color[rgb]{0,0,0}\makebox(0,0)[lb]{\smash{$\infty$}}}%
    \put(0.70514638,0.07540502){\color[rgb]{0,0,0}\makebox(0,0)[lb]{\smash{$b_{1}$}}}%
  \end{picture}%
\endgroup%

%% file: fig13.pdf_tex
\begingroup%
  \makeatletter%
  \providecommand\color[2][]{%
    \errmessage{(Inkscape) Color is used for the text in Inkscape, but the package 'color.sty' is not loaded}%
    \renewcommand\color[2][]{}%
  }%
  \providecommand\transparent[1]{%
    \errmessage{(Inkscape) Transparency is used (non-zero) for the text in Inkscape, but the package 'transparent.sty' is not loaded}%
    \renewcommand\transparent[1]{}%
  }%
  \providecommand\rotatebox[2]{#2}%
  \ifx\svgwidth\undefined%
    \setlength{\unitlength}{547.61078632bp}%
    \ifx\svgscale\undefined%
      \relax%
    \else%
      \setlength{\unitlength}{\unitlength * \real{\svgscale}}%
    \fi%
  \else%
    \setlength{\unitlength}{\svgwidth}%
  \fi%
  \global\let\svgwidth\undefined%
  \global\let\svgscale\undefined%
  \makeatother%
  \begin{picture}(1,0.21157997)%
    \put(0,0){\includegraphics[width=\unitlength,page=1]{fig13.pdf}}%
    \put(-0.00013792,0.00504219){\color[rgb]{0,0,0}\makebox(0,0)[lb]{\smash{$-\sigma_{n} = h_{n}^{0}$}}}%
    \put(0.18118589,0.00504219){\color[rgb]{0,0,0}\makebox(0,0)[lb]{\smash{$h_{n}^{1}$}}}%
    \put(0.33836026,0.00504219){\color[rgb]{0,0,0}\makebox(0,0)[lb]{\smash{$h_{n}^{2}$}}}%
    \put(0.49310215,0.00504219){\color[rgb]{0,0,0}\makebox(0,0)[lb]{\smash{$h_{n}^{3}$}}}%
    \put(0.64200044,0.00504219){\color[rgb]{0,0,0}\makebox(0,0)[lb]{\smash{$h_{n}^{4}$}}}%
    \put(0.8073685,0.00504219){\color[rgb]{0,0,0}\makebox(0,0)[lb]{\smash{$h_{n}^{5}$}}}%
    \put(0.92412528,0.00504219){\color[rgb]{0,0,0}\makebox(0,0)[lb]{\smash{$\sigma_{n}=h_{n}^{6}$}}}%
  \end{picture}%
\endgroup%

%% file: fig15.pdf_tex
\begingroup%
  \makeatletter%
  \providecommand\color[2][]{%
    \errmessage{(Inkscape) Color is used for the text in Inkscape, but the package 'color.sty' is not loaded}%
    \renewcommand\color[2][]{}%
  }%
  \providecommand\transparent[1]{%
    \errmessage{(Inkscape) Transparency is used (non-zero) for the text in Inkscape, but the package 'transparent.sty' is not loaded}%
    \renewcommand\transparent[1]{}%
  }%
  \providecommand\rotatebox[2]{#2}%
  \ifx\svgwidth\undefined%
    \setlength{\unitlength}{433.94934579bp}%
    \ifx\svgscale\undefined%
      \relax%
    \else%
      \setlength{\unitlength}{\unitlength * \real{\svgscale}}%
    \fi%
  \else%
    \setlength{\unitlength}{\svgwidth}%
  \fi%
  \global\let\svgwidth\undefined%
  \global\let\svgscale\undefined%
  \makeatother%
  \begin{picture}(1,0.59759693)%
    \put(0,0){\includegraphics[width=\unitlength,page=1]{fig15.pdf}}%
    \put(0.48794593,0.29527143){\color[rgb]{0,0,0}\makebox(0,0)[lb]{\smash{$\Bigg \downarrow$}}}%
    \put(0.04127076,0.39210073){\color[rgb]{0,0,0}\makebox(0,0)[lb]{\smash{$h_{n}^{(m)}$}}}%
    \put(0.4759552,0.39378077){\color[rgb]{0,0,0}\makebox(0,0)[lb]{\smash{$h_{n}^{(m+1)}$}}}%
    \put(0.92651026,0.39190828){\color[rgb]{0,0,0}\makebox(0,0)[lb]{\smash{$h_{n}^{(m+2)}$}}}%
    \put(0.02600968,0.00825133){\color[rgb]{0,0,0}\makebox(0,0)[lb]{\smash{$h_{n}^{(m)}$}}}%
    \put(0.91966001,0.00596518){\color[rgb]{0,0,0}\makebox(0,0)[lb]{\smash{$h_{n}^{(m+2)}$}}}%
    \put(0.266269,0.50425618){\color[rgb]{0,0,0}\makebox(0,0)[lb]{\smash{$b_{1}$}}}%
    \put(0.71475041,0.50336719){\color[rgb]{0,0,0}\makebox(0,0)[lb]{\smash{$b_{1}$}}}%
    \put(0.48956973,0.11372474){\color[rgb]{0,0,0}\makebox(0,0)[lb]{\smash{$b_{1}$}}}%
    \put(0.53896973,0.3039955){\color[rgb]{0,0,0}\makebox(0,0)[lb]{\smash{$\text{sum up to}$}}}%
  \end{picture}%
\endgroup%

%% file: fig16.pdf_tex
\begingroup%
  \makeatletter%
  \providecommand\color[2][]{%
    \errmessage{(Inkscape) Color is used for the text in Inkscape, but the package 'color.sty' is not loaded}%
    \renewcommand\color[2][]{}%
  }%
  \providecommand\transparent[1]{%
    \errmessage{(Inkscape) Transparency is used (non-zero) for the text in Inkscape, but the package 'transparent.sty' is not loaded}%
    \renewcommand\transparent[1]{}%
  }%
  \providecommand\rotatebox[2]{#2}%
  \ifx\svgwidth\undefined%
    \setlength{\unitlength}{399.45636492bp}%
    \ifx\svgscale\undefined%
      \relax%
    \else%
      \setlength{\unitlength}{\unitlength * \real{\svgscale}}%
    \fi%
  \else%
    \setlength{\unitlength}{\svgwidth}%
  \fi%
  \global\let\svgwidth\undefined%
  \global\let\svgscale\undefined%
  \makeatother%
  \begin{picture}(1,0.14031522)%
    \put(0,0){\includegraphics[width=\unitlength,page=1]{fig16.pdf}}%
    \put(0.11110246,0.06377833){\color[rgb]{0,0,0}\makebox(0,0)[lb]{\smash{$\infty$}}}%
    \put(0.29567203,0.06419372){\color[rgb]{0,0,0}\makebox(0,0)[lb]{\smash{$b_{1}$}}}%
    \put(0.6841918,0.06419372){\color[rgb]{0,0,0}\makebox(0,0)[lb]{\smash{$b_{1}$}}}%
    \put(0.48760444,0.0637783){\color[rgb]{0,0,0}\makebox(0,0)[lb]{\smash{$\infty$}}}%
    \put(0.87026989,0.0637783){\color[rgb]{0,0,0}\makebox(0,0)[lb]{\smash{$\infty$}}}%
    \put(0,0){\includegraphics[width=\unitlength,page=2]{fig16.pdf}}%
  \end{picture}%
\endgroup%

%% file: fig18.pdf_tex
\begingroup%
  \makeatletter%
  \providecommand\color[2][]{%
    \errmessage{(Inkscape) Color is used for the text in Inkscape, but the package 'color.sty' is not loaded}%
    \renewcommand\color[2][]{}%
  }%
  \providecommand\transparent[1]{%
    \errmessage{(Inkscape) Transparency is used (non-zero) for the text in Inkscape, but the package 'transparent.sty' is not loaded}%
    \renewcommand\transparent[1]{}%
  }%
  \providecommand\rotatebox[2]{#2}%
  \ifx\svgwidth\undefined%
    \setlength{\unitlength}{187.49538545bp}%
    \ifx\svgscale\undefined%
      \relax%
    \else%
      \setlength{\unitlength}{\unitlength * \real{\svgscale}}%
    \fi%
  \else%
    \setlength{\unitlength}{\svgwidth}%
  \fi%
  \global\let\svgwidth\undefined%
  \global\let\svgscale\undefined%
  \makeatother%
  \begin{picture}(1,0.45535418)%
    \put(0,0){\includegraphics[width=\unitlength,page=1]{fig18.pdf}}%
    \put(-0.00175963,0.01472734){\color[rgb]{0,0,0}\makebox(0,0)[lb]{\smash{$-1$}}}%
    \put(0.23389046,0.00651567){\color[rgb]{0,0,0}\makebox(0,0)[lb]{\smash{$-\frac{1}{2}$}}}%
    \put(0.71886381,0.00651567){\color[rgb]{0,0,0}\makebox(0,0)[lb]{\smash{$\frac{1}{2}$}}}%
    \put(0.9474163,0.01472734){\color[rgb]{0,0,0}\makebox(0,0)[lb]{\smash{$1$}}}%
  \end{picture}%
\endgroup%

%% file: fig17.pdf_tex
\begingroup%
  \makeatletter%
  \providecommand\color[2][]{%
    \errmessage{(Inkscape) Color is used for the text in Inkscape, but the package 'color.sty' is not loaded}%
    \renewcommand\color[2][]{}%
  }%
  \providecommand\transparent[1]{%
    \errmessage{(Inkscape) Transparency is used (non-zero) for the text in Inkscape, but the package 'transparent.sty' is not loaded}%
    \renewcommand\transparent[1]{}%
  }%
  \providecommand\rotatebox[2]{#2}%
  \ifx\svgwidth\undefined%
    \setlength{\unitlength}{435.65929351bp}%
    \ifx\svgscale\undefined%
      \relax%
    \else%
      \setlength{\unitlength}{\unitlength * \real{\svgscale}}%
    \fi%
  \else%
    \setlength{\unitlength}{\svgwidth}%
  \fi%
  \global\let\svgwidth\undefined%
  \global\let\svgscale\undefined%
  \makeatother%
  \begin{picture}(1,0.25609655)%
    \put(0,0){\includegraphics[width=\unitlength,page=1]{fig17.pdf}}%
    \put(-0.00014416,0.00427337){\color[rgb]{0,0,0}\makebox(0,0)[lb]{\smash{$h_{n}^{(m-1)} - 3h$}}}%
    \put(0.30224284,0.00427337){\color[rgb]{0,0,0}\makebox(0,0)[lb]{\smash{$h_{n}^{(m-1)}$}}}%
    \put(0.72593919,0.00427337){\color[rgb]{0,0,0}\makebox(0,0)[lb]{\smash{$h_{n}^{(m)}$}}}%
    \put(0.93856521,0.00427337){\color[rgb]{0,0,0}\makebox(0,0)[lb]{\smash{$h_{n}^{(m)} + 3h$}}}%
    \put(0.34313304,0.19697076){\color[rgb]{0,0,0}\makebox(0,0)[lb]{\smash{$D_{r}(z_{1})$}}}%
    \put(0.49214276,0.21848156){\color[rgb]{0,0,0}\makebox(0,0)[lb]{\smash{$D_{r}(z_{2})$}}}%
    \put(0.52919661,0.10994318){\color[rgb]{0,0,0}\makebox(0,0)[lb]{\smash{$D_{r}(z_{3})$}}}%
    \put(0.6478345,0.22231806){\color[rgb]{0,0,0}\makebox(0,0)[lb]{\smash{$D_{r}(z_{4})$}}}%
    \put(0.64816247,0.11246816){\color[rgb]{0,0,0}\makebox(0,0)[lb]{\smash{$D_{r}(z_{5})$}}}%
  \end{picture}%
\endgroup%

%% file: fig19.pdf_tex
\begingroup%
  \makeatletter%
  \providecommand\color[2][]{%
    \errmessage{(Inkscape) Color is used for the text in Inkscape, but the package 'color.sty' is not loaded}%
    \renewcommand\color[2][]{}%
  }%
  \providecommand\transparent[1]{%
    \errmessage{(Inkscape) Transparency is used (non-zero) for the text in Inkscape, but the package 'transparent.sty' is not loaded}%
    \renewcommand\transparent[1]{}%
  }%
  \providecommand\rotatebox[2]{#2}%
  \ifx\svgwidth\undefined%
    \setlength{\unitlength}{275.20153024bp}%
    \ifx\svgscale\undefined%
      \relax%
    \else%
      \setlength{\unitlength}{\unitlength * \real{\svgscale}}%
    \fi%
  \else%
    \setlength{\unitlength}{\svgwidth}%
  \fi%
  \global\let\svgwidth\undefined%
  \global\let\svgscale\undefined%
  \makeatother%
  \begin{picture}(1,0.26176043)%
    \put(0,0){\includegraphics[width=\unitlength,page=1]{fig19.pdf}}%
    \put(0.22690798,0.11210081){\color[rgb]{0,0,0}\makebox(0,0)[lb]{\smash{$\infty$}}}%
    \put(0.51598053,0.11322583){\color[rgb]{0,0,0}\makebox(0,0)[lb]{\smash{$b_{1}$}}}%
    \put(0.69482395,0.11266332){\color[rgb]{0,0,0}\makebox(0,0)[lb]{\smash{$\text{half open cylinder}$}}}%
  \end{picture}%
\endgroup%

%% file: fig6.pdf_tex
\begingroup%
  \makeatletter%
  \providecommand\color[2][]{%
    \errmessage{(Inkscape) Color is used for the text in Inkscape, but the package 'color.sty' is not loaded}%
    \renewcommand\color[2][]{}%
  }%
  \providecommand\transparent[1]{%
    \errmessage{(Inkscape) Transparency is used (non-zero) for the text in Inkscape, but the package 'transparent.sty' is not loaded}%
    \renewcommand\transparent[1]{}%
  }%
  \providecommand\rotatebox[2]{#2}%
  \ifx\svgwidth\undefined%
    \setlength{\unitlength}{196.83010657bp}%
    \ifx\svgscale\undefined%
      \relax%
    \else%
      \setlength{\unitlength}{\unitlength * \real{\svgscale}}%
    \fi%
  \else%
    \setlength{\unitlength}{\svgwidth}%
  \fi%
  \global\let\svgwidth\undefined%
  \global\let\svgscale\undefined%
  \makeatother%
  \begin{picture}(1,1.3831842)%
    \put(0,0){\includegraphics[width=\unitlength,page=1]{fig6.pdf}}%
    \put(1.02335942,1.05310773){\color[rgb]{0,0,0}\makebox(0,0)[lb]{\smash{}}}%
    \put(1.05415203,1.00794521){\color[rgb]{0,0,0}\makebox(0,0)[lb]{\smash{}}}%
    \put(0.97409124,1.02026226){\color[rgb]{0,0,0}\makebox(0,0)[lb]{\smash{}}}%
    \put(0.13341741,1.3545348){\color[rgb]{0,0,0}\makebox(0,0)[lb]{\smash{$K$}}}%
    \put(0.09278683,1.17383342){\color[rgb]{0,0,0}\makebox(0,0)[lb]{\smash{$-K$}}}%
    \put(-0.00107486,1.11448619){\color[rgb]{0,0,0}\makebox(0,0)[lb]{\smash{$-K-2$}}}%
    \put(0.90033305,1.26178281){\color[rgb]{0,0,0}\makebox(0,0)[lb]{\smash{$\text{upper}$}}}%
    \put(0.90191095,1.2027982){\color[rgb]{0,0,0}\makebox(0,0)[lb]{\smash{$\text{boundary}$}}}%
    \put(0.9085148,0.72208683){\color[rgb]{0,0,0}\makebox(0,0)[lb]{\smash{$\text{decomposition into}$}}}%
    \put(0.90837239,0.66510496){\color[rgb]{0,0,0}\makebox(0,0)[lb]{\smash{$\text{essential and}$}}}%
    \put(0.90848685,0.60828293){\color[rgb]{0,0,0}\makebox(0,0)[lb]{\smash{$\text{cylinderical regions}$}}}%
    \put(1.13831855,-1.8126591){\color[rgb]{0,0,0}\makebox(0,0)[lb]{\smash{}}}%
    \put(1.78404503,0.737096){\color[rgb]{0,0,0}\makebox(0,0)[lb]{\smash{}}}%
    \put(0.82667646,0.6679722){\color[rgb]{0,0,0}\makebox(0,0)[lb]{\smash{$\begin{rcases}\\ \\\\ \\\\ \\\\ \\\\ \\  \\ \end{rcases}$}}}%
    \put(0.82616996,1.24097203){\color[rgb]{0,0,0}\makebox(0,0)[lb]{\smash{$\begin{rcases}\\ \\ \\ \end{rcases}$}}}%
    \put(0.90337945,0.1323754){\color[rgb]{0,0,0}\makebox(0,0)[lb]{\smash{$\text{bottom}$}}}%
    \put(0.90495741,0.07339079){\color[rgb]{0,0,0}\makebox(0,0)[lb]{\smash{$\text{boundary}$}}}%
    \put(0.82921639,0.11156449){\color[rgb]{0,0,0}\makebox(0,0)[lb]{\smash{$\begin{rcases}\\ \\ \\ \end{rcases}$}}}%
  \end{picture}%
\endgroup%